\documentclass[11pt,a4paper]{article}

\usepackage[english]{babel}
\usepackage[T1]{fontenc}
\usepackage[utf8]{inputenc}
\usepackage[
  margin=2cm,
  includefoot,
  footskip=30pt,
]{geometry}
\usepackage{lmodern}
\usepackage{amsmath}
\usepackage{amssymb}
\usepackage{amsthm}
\usepackage{booktabs}
\usepackage{array}
\usepackage{graphicx}
\usepackage{cite}
\usepackage[dvipsnames]{xcolor}
\usepackage{tikz-cd}
\usetikzlibrary{decorations.markings,angles}
\usepackage{stmaryrd} \SetSymbolFont{stmry}{bold}{U}{stmry}{m}{n}
\usepackage{mathrsfs}
\usepackage{extarrows}
\usepackage{mathtools}
\usepackage{tensor}
\usepackage{enumerate}
\usepackage{adjustbox}
\usepackage[mathcal]{eucal}
\usepackage{hyperref}
\hypersetup{
	colorlinks,
	urlcolor={red!55!black},
	citecolor={green!60!black},
	linkcolor={red!55!black}
}
\usepackage{bm}
\usepackage{tabularx}
\usepackage{cleveref}
\usepackage[strict]{changepage}
\usepackage{tensor}
\usepackage{faktor}

\definecolor{ao}{rgb}{0.0, 0.3, 0.0}

\mathchardef\mhyphen="2D
\def\on{\operatorname}

\makeatletter
\providecommand{\leftsquigarrow}{%
  \mathrel{\mathpalette\reflect@squig\relax}%
}
\newcommand{\reflect@squig}[2]{%
  \reflectbox{$\m@th#1\rightsquigarrow$}%
}
\makeatother

\definecolor{ao}{rgb}{0.0, 0.5, 0.0}

\newtheorem{theorem}{Theorem}[section]
\newtheorem{lemma}[theorem]{Lemma}

\newtheorem{proposition}[theorem]{Proposition}
\newtheorem{corollary}[theorem]{Corollary}
\newtheorem{introthm}{Theorem}
\newtheorem{introcor}{Corollary}

\newtheorem{introrem}{Remark}

\theoremstyle{definition}
\newtheorem{construction}[theorem]{Construction}
\newtheorem{definition}[theorem]{Definition}
\newtheorem{notation}[theorem]{Notation}
\newtheorem{remark}[theorem]{Remark}
\newtheorem{example}[theorem]{Example}

\newtheorem{recollection}[theorem]{Recollection}
\newtheorem{warning}[theorem]{Warning}

\newtheorem{introex}{Example}

\newcommand\noloc{%
  \nobreak
  \mspace{6mu plus 1mu}
  {:}
  \nonscript\mkern-\thinmuskip
  \mathpunct{}
  \mspace{2mu}
}

\newcommand{\laxpush}[3]{ {#1 \underset{#2}{\overset{\rightarrow}{\amalg}} #3} }

\newcommand{\cons}{{\on{Cons}}}
\newcommand{\loc}{{\on{Loc}}}
\newcommand{\dSt}{\on{dSt}_k}

\newcommand{\A}{\mathcal{A}}
\newcommand{\B}{\mathcal{B}}
\newcommand{\C}{\mathcal{C}}
\newcommand{\D}{\mathcal{D}}
\newcommand{\E}{\mathcal{E}}

\newcommand{\X}{\mathcal{X}}
\newcommand{\M}{\mathcal{M}}
\newcommand{\Cons}{\mathrm{Cons}}
\newcommand{\Loc}{\mathrm{Loc}}

\title{Lagrangian structures on the derived moduli of constructible sheaves}
\author{Merlin Christ and Enrico Lampetti}
\date{\today}

\begin{document}
\maketitle

\abstract{Given a compact oriented manifold of dimension $n$ with a conically smooth stratification, we show that the moduli of $\D(k)$-valued constructible sheaves and the moduli of perverse sheaves are $(2-n)$-shifted Lagrangian. The former statement follows from the construction of a relative left $n$-Calabi--Yau structure on the stable $\infty$-category of $\D(k)$-valued constructible sheaves. This is achieved via a lax gluing result for categorical cubes equipped with cubical Calabi--Yau structures. Given a codimension $2$ submanifold, we further identify symplectic leaves corresponding to perverse sheaves with prescribed monodromy.}

\tableofcontents

\section{Introduction}

	Shifted Lagrangian structures on morphisms of derived stacks were introduced by Pantev--Toën--Vezzosi--Vaquié in \cite{PTVV13}. Following ideas of Toën, Brav--Dyckerhoff \cite{BD19} introduced a non-commutative analogue of a shifted Lagrangian structure, attached to a dg functor, called a relative left Calabi--Yau structure.
	A special case are (absolute) left $n$-Calabi--Yau structures, which amount to an $S^1$-equivariant identification of the inverse Serre functor with the suspension functor $[-n]$. Passing to the To\"en-Vaquié moduli of objects \cite{TV07}, a relative left Calabi--Yau structure induces a Lagrangian structure \cite{BD21}.
	Many natural examples of Lagrangian morphisms arise from topology and algebraic geometry.
	In this paper, we will describe a new class of examples arising from the topology of stratified spaces.

	Consider an oriented closed manifold $X$ of dimension $n$.
	As shown by Brav--Dyckerhoff \cite{BD19}, the smooth dg category $\loc^{\on{dg}}(X)$ of $\on{Perf}(k)$-valued local systems on $X$, with $k$ a field, inherits from the orientation of $X$ a left $n$-Calabi--Yau structure. 	The moduli of objects $\mathcal{M}_{\loc^{\on{dg}}(X)}$ thus inherits a $(2-n)$-shifted symplectic structure in the sense of \cite{PTVV13}.
	If more generally $X$ has boundary $\partial X$, the functor $\Loc^{\on{dg}}(\partial X)\to \Loc^{\on{dg}}(X)$ inherits a relative left $n$-Calabi--Yau structure. The arising morphism of derived stacks $\mathcal{M}_{\Loc^{\on{dg}}(X)}\to \mathcal{M}_{\Loc^{\on{dg}}(\partial X)}$ is then Lagrangian with $(3-n)$-shifted symplectic target, which is also called a $(2-n)$-shifted symplectic morphism \cite{Cal15}.

	In this paper, we generalize these results to the category of constructible sheaves $\on{Cons}_P(X)$ on a compact oriented manifold $X$ with respect to a conically smooth stratification $P$ in the sense of \cite{AFT17}.
	For this, we work in the setting of $k$-linear stable $\infty$-categories instead of dg categories. In particular, we replace $\on{Loc}^{\on{dg}}(X)$ with the $k$-linear stable $\infty$-category $\on{Loc}(X)$ of $\D(k)$-valued local systems and the To\"en-Vaquié moduli of objects by its $\infty$-categorical version \cite{AG14}.

	The $k$-linear stable $\infty$-category of $\D(k)$-valued constructible sheaves $\on{Cons}_P(X)$ was described in \cite{PT22, Jan24, CJ24, HPT24} in terms of functors from the exit-path $\infty$-category of $(X,P)$ to $\D(k)$, generalizing results on spaces-valued constructible sheaves from \cite[Appendix A]{HA}.
	The $\infty$-category $\on{Cons}_P(X)$ and its moduli of objects were further studied in \cite{HPT23,HPT24,Lam25}, and can be identified with an intrinsically defined moduli $\mathbf{Cons}_P(X)$ of constructible sheaves on $(X,P)$.
	For any perversity function $\mathfrak{p}$, there is an open immersion ${}^{\mathfrak{p}}\mathbf{Perv}_P(X) \hookrightarrow \mathbf{Cons}_P(X)$ of the moduli of perverse sheaves.

	Our main result can be stated as follows:

\begin{introthm}[{\Cref{thm:final_result}, \Cref{corollary:lagrangian_general}}]\label{introthm:1}
	Let $X$ be a compact oriented manifold of dimension $d$ with boundary.
	Let $(X,P)$ be a conically smooth stratification of $X$ with boundary of depth $n$. \begin{enumerate}[(i)]
\item There is a canonical left $d$-Calabi--Yau structure on the categorical $(n+1)$-cube $\on{Cons}_P(X)_\ast$ from \Cref{def:Cons_X} with tip the $k$-linear stable $\infty$-category $\on{Cons}_P(X)$ of $\D(k)$-valued constructible sheaves on the interior of $(X,P)$. In particular, the functor
\[
\on{colim}_{I^{d+1}\backslash \{1,\dots,1\}}\on{Cons}_P(X)_\ast\longrightarrow \on{Cons}_P(X)_{(1,\dots,1)}\simeq \on{Cons}_P(X)
\]
inherits a relative left $n$-Calabi--Yau structure.
\item Let $\mathfrak{p}\colon P \to \mathbb{Z}$ be a function.
	The above relative left $d$-Calabi--Yau functor gives rise to $(2-d)$-shifted Lagrangian morphisms between locally geometric derived stacks 
\[
\mathbf{Cons}_P(X) \longrightarrow \mathcal{M}_{\on{colim}_{I^{n+1}\backslash \{1,\dots,1\}}\on{Cons}_P(X)_\ast}
\]
and 
\[
{}^{\mathfrak{p}}\mathbf{Perv}_P(X) \longrightarrow \mathcal{M}_{\on{colim}_{I^{n+1}\backslash \{1,\dots,1\}}\on{Cons}_P(X)_\ast}\,.
\]
\end{enumerate}
\end{introthm}

\begin{introrem}
	In the case where the boundary is non-empty in \Cref{introthm:1}, the $\infty$-category $\on{Cons}_P(X)$ is understood to mean constructible sheaves which are locally constant in the normal direction of the boundary, or equivalently constructible sheaves on the interior of $X$. See also \Cref{rem:corners_or_interior}.	
\end{introrem}

	For $X$ a smooth algebraic curve, the moduli of locally constant sheaves on $X$ carries a Poisson structure which has been extensively studied in the literature \cite{FR99, GHLW97, Gol06, GR98}.
	In higher dimension, this has been generalized by Pantev-Toën \cite{PT18} for the derived stack of locally constant sheaves (i.e., for $G=\mathrm{GL}_n$ in their paper).
	\Cref{introthm:1} allows us to offer the following generalization:
	
\begin{introcor}[{\Cref{coro:Poisson_structure}}]
	Let $X$ be a compact oriented manifold of dimension $d$ with boundary.
	Let $(X,P)$ be a conically smooth stratification of $X$ with boundary and let $\mathfrak{p} \colon P \to \mathbb{Z}$ be a function.
	Then the derived stacks $\mathbf{Cons}_P(X)$ and ${}^{\mathfrak{p}}\mathbf{Perv}_P(X)$ carry $(2-d)$-shifted Poisson structures.
\end{introcor}

	Cubical Calabi--Yau structures were introduced in \cite{CDW23}, generalizing relative Calabi--Yau structures corresponding to the case of $1$-cubes (i.e.~functors).
	They can also be considered as higher morphisms in higher categories of iterated Calabi--Yau cospans \cite{BCS24}.
	Passing to the moduli of objects, cubical Calabi--Yau structures give rise to what one may term iterated Lagrangian intersection cubes. 
	A central feature of relative Calabi--Yau structures is that they can be glued. Multiple versions of such gluing properties have been formulated in the literature \cite{BD19,CDW23,Chr23,BCS24,ST25}. The simplest instance is that relative left Calabi--Yau functors can be glued along pushouts \cite{BD19}. In this paper, we introduce and apply a novel gluing property of Calabi--Yau cubes, which uses so-called directed pushouts, which are $(\infty,2)$-categorical partially lax pushouts. This generalizes the special case of the lax gluing of relative Calabi--Yau structures on functors (i.e.~$1$-cubes) described in \cite{CDW23}.

The cubical left Calabi--Yau structure on $\on{Cons}_P(X)_\ast$ gives rise to a left Calabi--Yau structure on the functor from the colimit over the complement of the tip of the cube
\[
\on{colim}_{I^n\backslash \{1,\dots,1\}}\on{Cons}_P(X)_\ast\longrightarrow \on{Cons}_P(X)_{(1,\dots,1)}\simeq \on{Cons}_P(X)\,.
\]
	However, the domain is an abstract colimit, and in general we do not give an explicit description of it.
	In the case that $X$ is closed and the depth of $P$ is $1$, we can identify the domain as follows: 

\begin{introex}
	Let $X$ be a smooth closed manifold of dimension $n$ and $D\subset X$ a (not necessarily connected) smooth codimension $d$ submanifold.
	Then $X$ inherits a conically smooth stratification $P$ of depth $1$ with strata $D$ and $X\backslash D$. 
	Let $\partial D\to D$ be the circle bundle given by the link of $D$, i.e.~the boundary of the tubular neighborhood of $D$ in $X$.
	The categorical $1$-cube $\on{Cons}_P(X)_\ast$ amounts to a functor
\begin{equation}\label{eq:boundary_functor}
	\on{Loc}(\partial D)\longrightarrow \on{Cons}_P(X)\,.
\end{equation}

	The functor \eqref{eq:boundary_functor} is induced by an equivalence between $\on{Cons}_P(X)$ and the (partially lax) directed pushout
\[ \on{Cons}_P(X)\simeq \on{Loc}(D)\overset{\rightarrow}{\amalg}_{\on{Loc}(\partial D)}\on{Loc}(X\backslash D) \,,\]
see also \Cref{def:directed_pushout}.
	Note that this description as a directed lax pushout recovers the well-known recollement of $\on{Cons}_P(X)$ into $\on{Loc}(D)$ and $\on{Loc}(X\backslash D)$.
	Part (i) of \Cref{introthm:1} states that the functor \eqref{eq:boundary_functor} is relative left $n$-Calabi--Yau, and that 
\[
\mathbf{Cons}_P(X)\longrightarrow \mathbf{Loc}(\partial D)
\]
is a $(2-n)$-shifted Lagrangian morphism, where $\mathbf{Loc}(\partial D) \coloneqq \mathcal{M}_{\on{Loc}(\partial D)}$ is the stack of local systems on $\partial D$.
\end{introex}

	Finally, we further specialize to the case that the smooth closed submanifold $D \subset X$ is of codimension $2$. Suppose that $D$ has $m$ connected components.
	We further describe $(2-n)$-shifted symplectic leaves of the moduli of perverse sheaves on $X$, obtained from fixing the monodromy around the $m$ circles in the circle bundle $\partial D\to D$.

\begin{introthm}[\Cref{thm:symplectic_leaves}]\label{introthm:2}
	Consider a set of elements $\lambda_i \in GL_{r_i}$, $i= 0,\ldots, m$.
	Then the derived stack ${}^{\mathfrak{p}}\mathbf{Perv}^{\underline{r}, \underline{\lambda}}_D(X)$ of perverse sheaves with prescribed monodromy $\underline{\lambda}=\{\lambda_1,\dots,\lambda_m\}$ along the fibers of the $m$ connected components of the circle bundle $\partial D\to D$ carries a $(2-n)$-shifted symplectic structure.
\end{introthm}

\begin{introrem}\label{rem_intro:vanishing_cycles}
	The circle bundle $\partial D \to D$ in \Cref{introthm:2} can be identified with the normal bundle of $D \subset X$, and the given Lagrangian morphism therein can be identified with the vanishing cycles functor attached to the deformation to the normal cone.
	In particular, the symplectic stack from \Cref{introthm:2} associated with local systems with prescribed trivial monodromy recovers the moduli stack of local systems $\mathbf{Loc}(X)$.
\end{introrem}

	We exemplify \Cref{introthm:2} in the case of $D$ a knot in a $3$-manifold as follows.
	Another example of a $0$-shifted symplectic stack arising from a Riemann surface is given in \Cref{ex:Riemann_surface}.

\begin{introex}\label{ex:knots_embedding} 
	Let $K$ be a knot embedded in a $3$-manifold $M$ and let $\alpha\colon K \to BS^1$ be the $S^1$-bundle $\partial K \to K$ given by the boundary of a tubular neighborhood of $K$ in $M$.
	Let $r \in \mathbb{N}$ and $\lambda \in GL_r$.
	Then ${}^{\mathfrak{p}}\mathbf{Perv}^{r, \lambda}_K(M)$ carries a $(-1)$-shifted symplectic structure. 
	 As pointed out to us by T. Kinjo, if $\lambda$ is diagonalizable, then the stack ${}^{\mathfrak{p}}\mathbf{Perv}^{r, \lambda}_K(M)$ can be shown to be symmetric. 

The stack ${}^{\mathfrak{p}}\mathbf{Perv}^{r, \lambda}_K(M)$ was also considered in the case $M=S^3$ in \cite{BK16} and \cite[Section 8.5]{ST25}.
\end{introex}

	Recent works show that the existence of $0$-shifted and $(-1)$-shifted symplectic structures on a symmetric derived stack $\mathcal{X}$ can be used to construct BPS Lie algebras which allows the study of the cohomology of $\mathcal{X}$ via cohomological Hall algebras, see e.g. \cite{BDNIKP25, HK25}.
	Another key ingredient in order to apply the results in \textit{loc.cit.}~is the existence of (derived) good moduli spaces, which has been established for (derived) stacks of perverse sheaves in \cite{Lam25}.

\subsection*{Organization of the paper}

	We begin in \Cref{sec:CY_cubes} with a discussion of Calabi--Yau categorical cubes.
	We prove in \Cref{prop:laxgluingCYcubes} the gluing result for cubical Calabi--Yau structures along directed pushouts.
	In \Cref{sec:CY_str_constr_sheaves}, we prove the main result.
	For this, we first associate Calabi--Yau cubes with manifolds with corners.
	By lax gluing the Calabi--Yau cubes associated with the unzippings of $X$ and the tubular neighborhoods of the strata, we construct the cubical Calabi--Yau structure on the categorical cube $\on{Cons}_P(X)_\ast$.
	Finally, we describe the Lagrangian structures on the moduli of constructible and perverse sheaves on $(X,P)$, as well as shifted symplectic leaves, in \Cref{sec:Lagrangian_str}.

\subsection*{Acknowledgements}

	We thank Mauro Porta and Jean-Baptiste Teyssier for their suggestions, as well as Tasuki Kinjo for fruitful discussions and suggestions. We also thank Marco Volpe for his comments on a draft of the paper.

	M.C.~is a member of the Hausdorff Center for Mathematics at the University of Bonn (DFG GZ 2047/1, project ID 390685813).
	
	E.L. is thankful for the hospitality of the Hausdorff Center for Mathematics at the University of Bonn and of the University of Seattle, where part of this research took place.

\section{Calabi--Yau categorical cubes}\label{sec:CY_cubes}

	In \Cref{subsec:categorical_cubes,subsec:cubical_CY_structures}, we recall the notion of a left Calabi--Yau structure on a cube of linear stable $\infty$-categories, as introduced in \cite{CDW23}.
	We will refer to such Calabi--Yau structures as cubical Calabi--Yau structures.
	In \Cref{subsec:lax_gluing_cubes}, we describe the directed lax gluing of categorical cubes and show that cubical Calabi--Yau structures also glue along these.
	Finally, in \Cref{subsec:non_lax_gluing} we show that cubical Calabi--Yau structure also glue under non-lax pushouts. 

\subsection{Linear stable \texorpdfstring{$\infty$}{infinity}-categories and categorical cubes}\label{subsec:categorical_cubes}

	We denote by $\mathcal{P}r^L_{\on{St}}$ the $\infty$-category of stable presentable $\infty$-categories and left adjoint functors.
	We fix once and for all a field $k$ and denote by $\on{Mod}_k\in \mathcal{P}r^L_{\on{St}}$ the symmetric monoidal $\infty$-category of $k$-modules.
	Note that $\on{Mod}_k\simeq \D(k)$ as symmetric monoidal $\infty$-categories.
	The $\infty$-category of $k$-linear $\infty$-categories is defined as $\on{LinCat}_k=\on{Mod}_{\on{Mod}_k}(\mathcal{P}r^L_{\on{St}})$.
	We note that the forgetful functor $\on{LinCat}_k\to \mathcal{P}r^L_{\on{St}}$ preserves colimits. 

	Recall that a $k$-linear $\infty$-category $\C$ is called dualizable if it admits a dual $\C^{\vee}$ in the symmetric monoidal $\infty$-category $\on{LinCat}_k$, equipped with the evaluation functor $\C^{\vee}\otimes \C\to \on{Mod}_k$ and coevaluation functor $\on{Mod}_k\to \C\otimes \C^{\vee}$, satisfying the triangle identities.
	A $k$-linear functor is called dualizable if its right adjoint preserves colimits and is thus a $k$-linear functor as well.
	We denote the $\infty$-category of dualizable $k$-linear $\infty$-categories and dualizable functors by $\on{LinCat}_k^{\on{dual}}$.
	Note that any compactly generated $k$-linear $\infty$-category is dualizable.
	The forgetful functor $\on{LinCat}_k^{\on{dual}}\to \on{LinCat}_k$ preserves colimits, see \cite[Prop.~1.65]{Efi24}.

	We denote by $I=N(\{0\to 1\})$ the nerve of the poset $\{0\to 1\}$.
	Note that $I\simeq \Delta^1$. All cubes, considered as multi chain complexes, will thus be in the cohomological grading convention, with the differential increasing the degree. 

\begin{definition}~
\begin{enumerate}[(1)]
	\item A $k$-linear categorical $n$-cube is a functor $\C_\ast\colon I^n\to \on{LinCat}_k$. 
	\item The $\infty$-category of $k$-linear categorical $n$-cubes is defined as 
\[
\on{Cube}_n(\on{LinCat}_k)\coloneqq \on{Fun}(I^n,\on{LinCat}_k)\,.
\] 
	The $\infty$-category of dualizable $k$-linear categorical $n$-cubes is similarly defined as 
\[
\on{Cube}_n(\on{LinCat}_k^{\on{dual}})\coloneqq \on{Fun}(I^n,\on{LinCat}_k^{\on{dual}})\,.
\]
\end{enumerate}
\end{definition}

\begin{remark}
	There is a canonical equivalence of $\infty$-categories 
\[ \on{Fun}(\Delta^1,\on{Cube}_n(\on{LinCat}_k))\simeq \on{Cube}_{n+1}(\on{LinCat}_k)\,.\]
	This equivalence amounts to identifying a morphism between $n$-cubes $\C_\ast\to \D_\ast$ with an $(n+1)$-cube $\E_\ast$ with $\E_{0,\ast}=\C_\ast$ and $\E_{1,\ast}=D_\ast$.
	We will also use the notation 
\[ \E_\ast=\left[\C_\ast\to \D_\ast\right]\,.\]
\end{remark}

\begin{definition}
   Let $\C_\ast \in \on{Cube}_n(\on{LinCat}_k)$.
\begin{enumerate}
    \item  For $1\leq i\leq n$, we denote by 
    \[ \C_{\partial_i^0\ast}\in \on{Cube}_{n-1}(\on{LinCat}_k)\]
    the restriction of $\C_{\ast}$ to $I^{i-1}\times \{0\}\times I^{n-i-1}\subset I^n$. We similarly denote by
    \[ \C_{\partial_i^1\ast}\in \on{Cube}_{n-1}(\on{LinCat}_k)\]
    the restriction to $I^{i-1}\times \{1\}\times I^{n-i-1}\subset I^n$.
\item    Repeated restriction is denoted by
    \[ \C_{\partial_{i_1}^{\epsilon_1}\dots \partial_{i_j}^{\epsilon_j}\ast}\in \on{Cube}_{n-j}(\on{LinCat}_k)\]
    for any $1\leq i_1<\dots<i_j\leq n$ and $\epsilon_1,\dots,\epsilon_j\in \{0,1\}$. 
\end{enumerate}
\end{definition}

	We denote by $\on{cof}\colon \on{Fun}(\Delta^1,\on{LinCat}_k)\to \on{LinCat}_k$ the (categorical) cofiber functor that maps a $k$-linear functor $F\colon \A\to \B$ to the pushout along $0$ in $\on{LinCat}_k$: 
\[
\begin{tikzcd}
\A \arrow[d] \arrow[r, "F"] \arrow[rd, "\ulcorner", phantom] & \B \arrow[d] \\
0 \arrow[r]                                                  & \on{cof}(F) 
\end{tikzcd}
\] 

\begin{definition}
	We denote by $\on{cof}_i\colon \on{Cube}_n(\on{LinCat}_k)\to \on{Cube}_{n-1}(\on{LinCat}_k)$ the cofiber functor in the $i$-th coordinate direction, defined as the composite functor 
\begin{align*}
    \on{Cube}_n(\on{LinCat}_k)=&\on{Fun}(I^{k-1}\times I\times I^{n-k},\on{LinCat}_k)\\
    \simeq& \on{Fun}(I^{k-1}\times I^{n-k},\on{Fun}(I,\on{LinCat}_k))\\
    \xrightarrow{\on{Fun}(I^{k-1}\times I^{n-k},\on{cof})} & \on{Fun}(I^{k-1}\times I^{n-k},\on{LinCat}_k)\\
    \simeq &\on{Cube}_{n-1}(\on{LinCat}_k)\,.
\end{align*}
\end{definition}

\subsection{Cubical Calabi--Yau structures}\label{subsec:cubical_CY_structures}

	Given a $k$-linear $\infty$-category $\C\in \on{LinCat}_k$, we denote by $\on{HH}(\C)\in \on{Mod}_k$ the $k$-linear Hochschild homology of $\C$.
	It can be defined via the formalism of traces, see \cite{HSS17} and also \cite[Section 2.6]{Chr23} for a discussion using the same notation as this article.
	The Hochschild homology $\on{HH}(\C)$ comes with an $S^1$-action, whose fixed points (limit over $BS^1$) are denoted by $\on{HH}^{S^1}(\C)$ and called the ($k$-linear) negative cyclic homology of $\C$.
	Hochschild homology and negative cyclic homology form functors $\on{HH}(\mhyphen),\on{HH}^{S^1}(\mhyphen)\colon \on{LinCat}_k^{\on{dual}}\to \on{Mod}_k$, together with a natural transformation $\on{HH}^{S^1}(\mhyphen)\to \on{HH}(\mhyphen)$.

	Recall that a $k$-linear dualizable $\infty$-category $\C$ is called smooth if the evaluation functor $\on{ev}_\C\colon \C^\vee\otimes \C\to \on{Mod}_k$ admits a colimit preserving left adjoint.
	In this case, the left adjoint $\on{ev}_\C^L$ can be identified with a $k$-linear endofunctor of $\C$, denoted $\on{id}_\C^!$, and called the inverse dualizing bimodule.
	We note that $\on{HH}(\C)\simeq \on{Mor}(\on{id}_\C^!,\on{id}_\C)$ describes the morphism object in the $k$-linear $\infty$-category of $k$-linear endofunctors of $\C$. 

\begin{definition}[$\!\!$\cite{BD19} for dg categories, \cite{Chr23} for linear $\infty$-categories]
	Let $F\colon \C\to \D$ be a dualizable $k$-linear functor between smooth $k$-linear $\infty$-categories.
	A relative\footnote{The adjective relative for the Calabi--Yau structure on a functor is redundant.
	We will however always include it to distinguish relative Calabi--Yau structures from cubical Calabi--Yau structures.} left $n$-Calabi--Yau structure on $F$ consists of a class 
\[ R[n]\longrightarrow \on{cof}(\on{HH}^{S^1}(\C)\xrightarrow{\on{HH}^{S^1}(F)} \on{HH}^{S^1}(\D))\]
such that in the corresponding diagram of $k$-linear endofunctors of $\C$ (see \cite{BD19,Chr23})
\[
\begin{tikzcd}
\on{id}_\D^! \arrow[r, "\on{cu}^!"] \arrow[d] & F_!(\on{id}_\C^!) \arrow[r] \arrow[d]            & \on{cof} \arrow[d]  \\
\on{fib} \arrow[r]                       & {F_!(\on{id}_\C)[1-n]} \arrow[r, "\on{cu}"] & {\on{id}_\D[1-n]}
\end{tikzcd}
\]
with horizontal fiber and cofiber sequences, the vertical morphisms are equivalences. 
\end{definition}

	To define cubical Calabi--Yau structures, we consider the total Hochschild homology of a categorical cube.
	Recall that taking the total cofiber defines a functor $(\mhyphen)^{\on{tot}}\colon \on{Fun}(I^n,\on{Mod}_k)\to \on{Mod}_k$, see for instance \cite[Appendix A.2]{DGT19}. The functor $(\mhyphen)^{\on{tot}}$ can be obtained as the composite of the iterated cofiber functors in the $n$ directions in $I^n$, in any order. 

\begin{definition}\label{def:CY_Cubes}
	The total $k$-linear negative cyclic homology functor 
\[
\on{HH}^{S^1,\on{tot}}(\mhyphen)\colon \on{Cube}_n(\on{LinCat}_k^{\on{dual}})\longrightarrow \on{Mod}_k
\]
is defined as the composite of $\on{Fun}(I^n,\on{HH}^{S^1}(\mhyphen))$ with $(\mhyphen)^{\on{tot}}\colon \on{Fun}(I^n,\on{Mod}_k)\to \on{Mod}_k$.

	We similarly define the functor $\on{HH}^{\on{tot}}(\mhyphen)\colon \on{Cube}_n(\on{LinCat}_k^{\on{dual}})\longrightarrow \on{Mod}_k$.
\end{definition}

	A total negative cyclic homology class $R[m]\to \on{HH}^{\on{tot}}(\C_\ast)$ yields the following data\footnote{The total negative cyclic homology class can in fact also be recovered from this data.}:
\begin{itemize}
    \item A relative negative cyclic homology class $R[m]\to \on{cof}(\on{HH}(\on{colim}_{I^n\backslash (\{1\}^{\times n})} \C_\ast)\to \on{HH}(\C_{1,\dots,1}))$ of the functor $\on{colim}_{I^n\backslash (\{1\}^{\times n})} \C_\ast \to \C_{1,\dots,1}$. 
    \item A total relative negative cyclic homology class $R[m-1] \to \on{HH}^{\on{tot}}(C_{\partial_i^0 \ast})$ of a boundary cube for every $1\leq i\leq n$. 
\end{itemize}

	The notion of a cubical left Calabi--Yau structure is recursively obtained from the notion of a relative Calabi--Yau structure on a functor as follows:

\begin{definition}[\!\!\cite{CDW23}]\label{def:cubical_CY_structure}
    Let $\C_\ast\in \on{Cube}_n(\on{LinCat}_k)$. For any $m\in \mathbb{Z}$, a cubical $m$-Calabi--Yau structure on $\C$ consists of a class $R[m]\to \on{HH}^{S^1,\on{tot}}(\C_\ast)$ satisfying that
    \begin{enumerate}[(1)]
        \item The relative negative cyclic homology class $R[m]\to \on{cof}(\on{HH}(\on{colim}_{I^n\backslash \{1\}^{\times n}} \C_\ast)\to \on{HH}(\C_{1,\dots,1}))$ defines a relative left $m$-Calabi--Yau structure on the functor 
        \[ \on{colim}_{I^n\backslash \{1\}^{\times n}} \C_\ast \to \C_{1,\dots,1}\,.\] 
        \item For every $1\leq i\leq n$, the class $R[m-1] \to \on{HH}^{\on{tot}}(C_{\partial_i^0 \ast})$ defines a cubical $(m-1)$-Calabi--Yau structure on $C_{\partial_i^0\ast}$. 
    \end{enumerate}
\end{definition}

	Unraveling the definition, we see that a cubical left $m$-Calabi--Yau structure on a categorical $n$-cube $\C_\ast$ amounts to compatible left Calabi--Yau structures on the functors
\[
\on{colim}_{(I^n)_{/{\bf i}}\backslash \{{\bf i}\}} \C_\ast \to \C_{\bf i} 
\]
for all ${\bf i}\in I^n$, where $(I^n)_{/{\bf i}}\backslash \{{\bf i}\}$ describes the cube above ${\bf i}$ (of dimension less or equal to $n$) with the bottom tip ${\bf i}$ removed.  

\begin{example}
    Let $\C_\ast\in \on{Cube}_2(\on{LinCat}_k^{\on{dual}})$.  We can depict $\C_\ast$ as follows:
\[
\C_\ast =\begin{tikzcd}
{\C_{0,0}} \arrow[r, "F"] \arrow[d, "F'"'] & {\C_{0,1}} \arrow[d] \\
{\C_{1,0}} \arrow[r]                       & {\C_{1,1}}          
\end{tikzcd}
\]
	The colimit $\on{colim}_{I^2\backslash \{1\}^{\times 2}} \C_\ast$ amounts to the pushout $\C_{0,1}\amalg_{\C_{0,0}}\C_{1,0}$ of the upper span, and comes with a functor 
\[ F''\colon \C_{0,1}\amalg_{\C_{0,0}}\C_{1,0}\to \C_{1,1}\,.\] 

	A cubical $m$-Calabi--Yau structure on $\C_\ast$ consists of a total negative cyclic homology class $R[m]\to \on{HH}^{S^1,\on{tot}}(\C_\ast)$, which induces 
 \begin{itemize}
     \item a left $(m-2)$-Calabi--Yau structure on $\C_{0,0}$.
     \item relative  left $(m-1)$-Calabi--Yau structures on $F\colon \C_{0,0}\to \C_{0,1}$ and $F'\colon \C_{0,0}\to \C_{1,0}$. These glue to a left $(m-1)$-Calabi--Yau structure on $\C_{0,1}\amalg_{\C_{0,0}}\C_{1,0}$. 
     \item a relative left $m$-Calabi--Yau structure on $F''\colon \C_{0,1}\amalg_{\C_{0,0}}\C_{1,0}\to \C_{1,1}$. 
 \end{itemize}
\end{example}

	The following gives a convenient characterization of cubical Calabi--Yau structures for double cubes:

\begin{proposition}\label{prop:equivalent_characterization_CY_cube}
    Let $(\C_{\ast_1})_{\ast_2}\in \on{Fun}(I^{n_2},\on{Cube}_{n_1}(\on{LinCat}_k^{\on{dual}}))\simeq \on{Cube}_{n_1+n_2}(\on{LinCat}_k^{\on{dual}})$ be an $(n_1+n_2)$-cube, with $\ast_1\in I^{n_1}$ and $\ast_2 \in I^{n_2}$.
    Consider a total negative cyclic homology class $\eta\colon R[m]\to \on{HH}^{S^1,\on{tot}}((\C_{\ast_1})_{\ast_2})$. If
    \begin{enumerate}[1)]
        \item the class $\eta$ induces a cubical left $m$-Calabi--Yau structure on the $(n_1+1)$-cube 
        \[ \left[\on{colim}_{I^{\times n_2}\backslash \{1\}^{\times n_2}} (\C_{\ast_1})_{\ast_2}\to (\C_{\ast_1})_{1,\dots,1} \right] \]
        \item and $\eta$ restricts to a cubical left $(m-1)$-Calabi--Yau structure on the $k$-th face $(n_1+n_2-1)$-cube $(\C_{\ast_1})_{\partial^0_i \ast_2}$ for all $1\leq i\leq n_2$,
    \end{enumerate}
    then $\eta$ defines a cubical left $m$-Calabi--Yau structure on $(\C_{\ast_1})_{\ast_2}$.
\end{proposition}

\begin{proof}
	A total negative cyclic homology class $\eta$ for $(\C_{\ast_1})_{\ast_2}$ satisfies the (recursively appearing) condition 1) in \Cref{def:cubical_CY_structure} at all indices of the cube $(\C_{{\bf i}})_{{\bf j}}$ with ${\bf i}\in I^{n_1}$ arbitrary and ${\bf j}\in I^{n_2}\backslash \{1,\dots,1\}^{\times n_2}$ if and only it satisfies the recursive condition 2) in \Cref{prop:equivalent_characterization_CY_cube}. 

	We suppose that these equivalent conditions are satisfied in the following and show that if the induced total negative cyclic homology class of the categorical $(n_1+1)$-cube 
\[
\E_\ast\coloneqq \left[\on{colim}_{I^{\times n_2}\backslash \{1\}^{\times n_2}} (\C_{\ast_1})_{\ast_2}\to (\C_{\ast_1})_{1,\dots,1} \right] 
\]
defines a cubical left $m$-Calabi--Yau structure then the class $\eta$ satisfies the recursive definition for cubical left Calabi--Yau structures at the indices of the form $(\C_{{\bf i}})_{1,\dots,1}$, ${\bf i}\in I^{n_1}$.
	For this, we show that the non-degeneracy conditions arising from the $(n_1+1)$-th front face $\E_{\partial^1_{n_1+1}\ast}$ coincide with the remaining non-degeneracy conditions of $\eta$ at entries $(\C_{{\bf i}})_{1,\dots,1}$, where ${\bf i}\in I^{n_1}$, by showing that the relevant functors are equivalent. 

	We fix a front face index $({\bf i},1)$ of $\E_\ast$.
	We have 
\begin{equation}\label{eq:decomp_n+1_cube}
I^{n_1+1}_{/({\bf i},1)} \backslash \{({\bf i},1)\}\simeq (I^{n_1}_{/{\bf i}}\times \{0\}) \amalg_{(I^{n_1}_{/{\bf i}}\backslash \{{\bf i}\})\times \{0\}} (I^{n_1}_{/{\bf i}}\backslash {\bf i}\times I)
\end{equation}
and
\begin{equation}\label{eq:decomp_n+j_cube} 
(I^{n_1}_{/{\bf i}}\times I^{n_2})\backslash \{({\bf i},(1,\dots,1))\} \simeq 
\big(I^{n_1}_{/{\bf i}}\times (I^{n_2}\backslash \{1\}^{\times n_2})\big) \amalg_{(I^{n_1}_{/{\bf i}}\backslash \{{\bf i}\}\times (I^{n_2}\backslash \{1\}^{\times n_2})} ((I^{n_1}_{/{\bf i}}\backslash \{{\bf i}\})\times I^{n_2} \big)\,.
\end{equation}
	Using \cite[\href{https://kerodon.net/tag/03DB}{Tag 03DB}]{Ker} and \eqref{eq:decomp_n+1_cube}, we find
\begin{align*}
\on{colim}_{I^{n_1+1}_{/({\bf i},1)}\backslash ({\bf i},1)}\E_\ast  
& \simeq  \E_{({\bf i}, 0)} \amalg_{\on{colim}_{I^{n_1}_{/{\bf i}}\backslash {\bf i}} \E_{(\ast_1,0)}}  \on{colim}_{I^{n_1}_{/{\bf i}}\backslash {\bf i}} \E_{(\ast_1,1)}\\
& \simeq  \on{colim}_{I^{n_2}\backslash \{1\}^{\times n_2}} (\C_{\bf i})_{\ast_2} \amalg_{\on{colim}_{I^{n_1}_{/{\bf i}}\backslash {\bf i}}\on{colim}_{I^{n_2}\backslash \{1\}^{\times n_2}} (\C_{\ast_1})_{\ast_2}} \on{colim}_{I^{n_1}_{/{\bf i}}\backslash \{{\bf i}\}} (\C_{\ast_1})_{1,\dots,1}\,.
\end{align*}
	Similarly, using \eqref{eq:decomp_n+j_cube}, we find 
\begin{align*}
&\on{colim}_{(I^{n_1}_{/{\bf i}}\times I^{n_2})\backslash \{({\bf i},(1,\dots,1))\}}(\C_{\ast_1})_{\ast_2}\\
& \simeq \on{colim}_{I^{n_1}_{/{\bf i}}\times (I^{n_2}\backslash \{1\}^{\times n_2})}(\C_{\ast_1})_{\ast_2}\amalg_{ \on{colim}_{I^{n_1}_{/{\bf i}}\backslash \{{\bf i}\})\times I^{n_2}\backslash \{1\}^{\times n_2}} (\C_{\ast_1})_{\ast_2}} \on{colim}_{I^{n_1}_{/{\bf i}}\backslash \{{\bf i}\}\times I^{n_2}}(\C_{\ast_1})_{\ast_2}\,.
\end{align*}
	Note that 
\[
\on{colim}_{I^{n_2}\backslash \{1\}^{\times n_2}} (\C_{\bf i})_{\ast_2} \simeq \on{colim}_{I^{n_1}_{/{\bf i}}\times (I^{n_2}\backslash \{1\}^{\times n_2})}(\C_{\ast_1})_{\ast_2}\]
and
\[
\on{colim}_{I^{n_1}_{/{\bf i}}\backslash \{{\bf i}\}} (\C_{\ast_1})_{1,\dots,1}\simeq \on{colim}_{I^{n_1}_{/{\bf i}}\backslash \{{\bf i}\}\times I^{n_2}}(\C_{\ast_1})_{\ast_2}
\]
by cofinality. We thus find 
\[
\on{colim}_{I^{n_1+1}_{/({\bf i},1)}\backslash ({\bf i},1)}\E_\ast  \simeq \on{colim}_{(I^{n_1}_{/{\bf i}}\times I^{n_2})\backslash \{({\bf i},(1,\dots,1))\}}(\C_{\ast_1})_{\ast_2}\simeq \on{colim}_{(I^{n_1}\times I^{n_2})_{/\{({\bf i},1,\dots,1)\}}\backslash \{({\bf i},(1,\dots,1))\}}(\C_{\ast_1})_{\ast_2}\]
which implies that the induced functors to $(\C_{\bf i})_{1,\dots,1}\simeq \E_{({\bf i},1)}$ are equivalent too. 
\end{proof}

\begin{remark}
	The reverse implication in \Cref{prop:equivalent_characterization_CY_cube} can be shown to also hold.
\end{remark}

\begin{example}
\begin{enumerate}[(1)]
\item In the case that $n_2=1$ in \Cref{prop:equivalent_characterization_CY_cube}, we have that 
\[
(\C_{\ast_1})_{\ast_2}\simeq \left[\on{colim}_{I\backslash \{1\}} (\C_{\ast_1})_{\ast_2}\to (\C_{\ast_1})_{1} \right]\,.
\]
	Hence, \Cref{prop:equivalent_characterization_CY_cube} is immediate.
\item Suppose that $n_1=1$ and $n_2=2$. We can depict $(\C_{\ast_1})_{\ast_2}$ as follows:
\[
\begin{tikzcd}[column sep=small, row sep=small]
{\C_{0,0,0}} \arrow[rd] \arrow[rr] \arrow[dd] &                                    & {\C_{1,0,0}} \arrow[rd] \arrow[dd] &                         \\
                                              & {\C_{0,1,0}} \arrow[rr] \arrow[dd] &                                    & {\C_{1,1,0}} \arrow[dd] \\
{\C_{0,0,1}} \arrow[rd] \arrow[rr]            &                                    & {\C_{1,0,1}} \arrow[rd]            &                         \\
                                              & {\C_{0,1,1}} \arrow[rr]            &                                    & {\C_{1,1,1}}           
\end{tikzcd}
\]
	Condition 1) states that the square
\[
\begin{tikzcd}
{\C_{0,0,1}\amalg_{\C_{0,0,0}}\C_{0,1,0}} \arrow[r] \arrow[d] & {\C_{1,0,1}\amalg_{\C_{1,0,0}}\C_{1,1,0}} \arrow[d] \\
{\C_{0,1,1}} \arrow[r]                                        & {\C_{1,1,1}}                                       
\end{tikzcd}
\]
inherits a cubical left $n$-Calabi--Yau structure.
	Condition 2) states that the top and the back faces inherit cubical left $(n-1)$-Calabi--Yau structures. Together, these conditions imply that the class $\eta$ defines a cubical $n$-Calabi--Yau structure on $(\C_{\ast_1})_{\ast_2}$.
\end{enumerate}
\end{example}

\subsection{Lax gluing}\label{subsec:lax_gluing_cubes}

\begin{definition}\label{def:directed_pushout}
	Given a span 
\[
\begin{tikzcd}
\A \arrow[r] \arrow[d] & \B \\
\C                     &   
\end{tikzcd}
\] 
in $\on{LinCat}_k$, the directed pushout $\laxpush{\B}{\A}{\C}$ is defined as the colimit of the diagram
\begin{equation}\label{eq:double_span}
\begin{tikzcd}
                                  & \A \arrow[r] \arrow[d, "\iota_0"] & \B \\
\A \arrow[d] \arrow[r, "\iota_1"] & {\on{Fun}(\Delta^1,\A)}           &    \\
\C                                &                                   &   
\end{tikzcd}
\end{equation}
where $\iota_0\colon \A\to \on{Fun}(\Delta^1,\A), X\mapsto (X\to 0)$ and $\iota_1\colon \A\to \on{Fun}(\Delta^1,\A), X\mapsto (0\to X)[1]$ are the apparent fully faithful functors. 
\end{definition}

	The directed pushout can also be interpreted more conceptually as a partially lax $(\infty,2)$-categorical colimit, see for instance \cite{CDW23,AHM26}.
	Since the functor $\on{LinCat}_k^{\on{dual}}\to \on{LinCat}_k$ preserves colimits, the directed pushout of a span in $\on{LinCat}_k^{\on{dual}}$ is again dualizable. 

\begin{construction}
	The passage to the directed pushout defines a functor
\[ \laxpush{(\mhyphen)}{(\mhyphen)}{(\mhyphen)}\colon \on{Fun}(\Lambda^2_0,\on{LinCat}_k)\to \on{LinCat}_k\,,\]
which we construct in the following. 

	Below, we construct a functor 
\[
\iota_{0,1}(\mhyphen)\colon \on{LinCat}_k\to \on{Fun}(\Lambda^2_2,\on{LinCat}_k)
\]
mapping a $k$-linear $\infty$-category $\mathcal{A}\in \on{LinCat}_k$ to the following diagram.
\[
\begin{tikzcd}
                        & \A \arrow[d, "\iota_1"] \\
\A \arrow[r, "\iota_0"] & {\on{Fun}(\Delta^1,\A)}
\end{tikzcd}
\]

	We define $\laxpush{(\mhyphen)}{(\mhyphen)}{(\mhyphen)}$ as the following composite functor
\begin{align*}
\on{Fun}(\Lambda^2_0,\on{LinCat}_k)&\simeq \on{Fun}(\Delta^1,\on{LinCat}_k)\times_{\on{LinCat}_k}\on{Fun}(\Delta^1,\on{LinCat}_k)\\
&\simeq \on{Fun}(\Delta^1,\on{LinCat}_k)\times_{\on{LinCat}_k}\on{LinCat}_k\times_{\on{LinCat}_k}\on{Fun}(\Delta^1,\on{LinCat}_k)\\
& \to \on{Fun}(\Delta^1,\on{LinCat}_k)\times_{\on{LinCat}_k}\on{Fun}(\Lambda^2_2,\on{LinCat}_k)\times_{\on{LinCat}_k}\on{Fun}(\Delta^1,\on{LinCat}_k)\\
& \simeq \on{Fun}(\Delta^1\amalg_{\Delta^0}\Lambda^2_2\amalg_{\Delta^0}\Delta^1,\on{LinCat}_k)\\
&\xrightarrow{\on{colim}} \on{LinCat}_k\,.
\end{align*}
	The second morphism above is given by composing with $\iota_{0,1}$ at the central $\on{LinCat}_k$.

	For $i=0,1$, let $f_i\colon \Delta^0\to \Delta^1$ be the inclusion of $i$. Pulling back along $f_i$ defines a functor 
\[
\on{LinCat}_k\to \on{Fun}(\Delta^1,\on{LinCat}_k)
\]
mapping a $k$-linear stable $\infty$-category $\A$ to $\on{Fun}(\Delta^1,\A)\xrightarrow{\iota_i^*}\A$.
	Passing to the right Kan extension along $\Delta^1\to \Lambda^2_2$, mapping $0\to 1$ to $1\to 2$, and to the left Kan extension along $\Lambda^2_2\to \Delta^1\times \Delta^1$ defines a functor 
\[
\on{LinCat}_k\to \on{Fun}(\Delta^1\times \Delta^1,\on{LinCat}_k)
\]
mapping a $k$-linear stable $\infty$-category $\A$ to the pullback square
\[
\begin{tikzcd}
\A \arrow[d] \arrow[r, "\iota_{j}"] & {\on{Fun}(\Delta^1,\A)} \arrow[d, "f_i^*"] \\
0 \arrow[r]                         & \A                                        
\end{tikzcd}
\]
where $j=1$ if $i=2$ and $j=2$ if $i=1$.
	Restricting to the upper morphism in the above squares defines functors $\on{LinCat}_k\to \on{Fun}(\Delta^1,\on{LinCat}_k)$, which combine into the desired functor $\iota_{0,1}(\mhyphen)$. 
\end{construction}

\begin{definition}\label{def:directed_pushout_cubes}
	Given a diagram in $\on{Cube}_n(\on{LinCat}_k)$ of the form
\[\begin{tikzcd}
\A_\ast \arrow[r] \arrow[d] & \B_\ast \\
\C_\ast                    &       
\end{tikzcd}\]
we denote by 
\[ \laxpush{\B_\ast}{\A_\ast}{\C_\ast}\in \on{Cube}_{n}(\on{LinCat}_k)\] 
the component-wise directed pushout.
	Thus, $\laxpush{\B_\ast}{\A_\ast}{\C_\ast}$ satisfies for all ${\bf i}\in I^n$
\[
\left(\laxpush{\B_\ast}{\A_\ast}{\C_\ast}\right)_{\bf i}\simeq \laxpush{\B_{\bf i}}{\A_{\bf i}}{\C_{\bf i}}\,.
\]
\end{definition}

\begin{remark}\label{rem:functor_into_directed_pushout}~
\begin{enumerate}[(1)]
\item Consider the setting of \Cref{def:directed_pushout}. Composing the fully faithful functor 
\[ \iota\colon \A\to \on{Fun}(\Delta^1,\A),\quad X\mapsto (X\xrightarrow{\on{id}}X)\] 
with the functor $\on{Fun}(\Delta^1,\A)\to \laxpush{\B}{\A}{\C}$ in the colimit diagram defines a functor 
\[ \A \to \laxpush{\B}{\A}{\C}\,.\]
	This functor has the property that the (non-lax) pushout $\B\amalg_{\A}\C$ is equivalent to its cofiber, meaning that the following diagram is a pushout square in $\on{LinCat}_k$:
\[
\begin{tikzcd}
\A \arrow[r] \arrow[d] \arrow[rd, "\ulcorner", phantom] & \laxpush{\B}{\A}{\C} \arrow[d] \\
0 \arrow[r]                                             & \B \amalg_{\A}\C              
\end{tikzcd}
\]
\item In the setting of \Cref{def:directed_pushout_cubes}, there is similarly a morphism of categorical $n$-cubes $\A_\ast \to \laxpush{\B_\ast}{\A_\ast}{\C_\ast}$ and a cofiber sequence in $\on{Cube}_n(\on{LinCat}_k)$ as follows:
\[
\begin{tikzcd}
\A_\ast \arrow[r] \arrow[d] \arrow[rd, "\ulcorner", phantom] & \laxpush{\B_\ast}{\A_\ast}{\C_\ast} \arrow[d] \\
0 \arrow[r]                                             & \B_\ast \amalg_{\A_\ast}\C_\ast              
\end{tikzcd}
\]
with $\B_\ast \amalg_{\A_\ast}\C_\ast$ the (non-lax) pushout $n$-cube. 
\end{enumerate}
\end{remark}

\begin{proposition}[$\!\!$\cite{CDW23}]\label{prop:laxCYgluing}
	Consider a diagram in $\on{LinCat}_k^{\on{dual}}$ of the form
\[
\begin{tikzcd}
                          &                                     & \A \arrow[d] \\
                          & \A' \arrow[r] \arrow[d, "\iota_0"] & \B             \\
\A'' \arrow[r, "\iota_1"] & \C                                  &               
\end{tikzcd}
\]
and assume that the functors $\A\times \A'\to \B$ and $\A'\times \A''\to \C$ are equipped with relative left $m$-Calabi--Yau structures, which are compatible at $\A'$.
	Then the functor
\[
\A\times \A'\times \A''\to \laxpush{\B}{\A'}{\C}
\]
inherits a canonical relative left $n$-Calabi--Yau structure.
\end{proposition}

	A cubical generalization of the above proposition is as follows:

\begin{proposition}\label{prop:laxgluingCYcubes}
	Consider a diagram in $\on{Cube}_n(\on{LinCat}_k)$ of the form
\[\begin{tikzcd}
                    &                               & \A_\ast\arrow[d] \\
                    & \A_\ast' \arrow[r] \arrow[d] & \B_\ast             \\
\A_\ast'' \arrow[r] & \C_\ast                       &                    
\end{tikzcd}\]
such that the $(n+1)$-cubes $\left[ \A_\ast\times \A'_\ast\to \B_\ast \right]$ and $\left[ \A'_\ast\times \A''_\ast\to \C_\ast \right]$ carry cubical $m$-Calabi--Yau structures, compatible at $\A'_\ast$. Then the $(n+1)$-cube 
\[
\left[ \A_\ast \times \A_\ast' \times \A_\ast''\to \laxpush{\B_\ast}{\A_\ast'}{\C_\ast}\right]
\]
inherits a canonical cubical left $m$-Calabi--Yau structure.

	In particular, if $\A_\ast\simeq \A_\ast''\simeq 0$, the categorical $(n+1)$-cube 
\[
\left[ \A_\ast'\to \laxpush{\B_\ast}{\A_\ast'}{\C_\ast}\right]
\]
arising from the morphisms of categorical $n$-cubes from \Cref{rem:functor_into_directed_pushout} inherits a canonical cubical left $m$-Calabi--Yau structure.
\end{proposition}

\begin{proof}
	The proof uses results from \cite{CDW23} on categorical chain complexes, totalizations of categorical cubes, and the interactions between their relative Calabi--Yau structures.
	We refer to \cite{CDW23} for background on categorical chain complexes. We also cite the relevant locations in \cite{CDW23} below.

	Consider the categorical $(n+2)$-cube $\E_\ast$ described by the square of categorical $n$-cubes
\[
\begin{tikzcd}
\A_\ast\times \A_\ast'\times \A_\ast'' \arrow[r] \arrow[d] & \C_\ast \arrow[d] \\
\B_\ast\arrow[r]                                & 0           
\end{tikzcd}
\]
where the functors $\A''_\ast\to \B_\ast$ and $\A_\ast\to \C_\ast$ are trivial.
	We choose the coordinates $n+1,n+2$ of $\E_\ast$ to correspond to the morphisms of $n$-cubes depicted above.
	The given cubical left $m$-Calabi--Yau structures determine a cubical left $(m+1)$-Calabi--Yau structure on $\E_\ast$. Let $\on{tot}^{\amalg}(\E_\ast)$ be the coproduct totalization of $\E_\ast$ in the sense of \cite[Def.~4.4.2]{CDW23}.
	This is a categorical chain complex in the sense of \cite[Def.~2.2.5]{CDW23} concentrated in homological degrees $m+3,\dots,1$.
	By \cite[Thm.~6.4.6]{CDW23}, the totalization $\on{tot}^{\amalg}(\E_\ast)$ inherits a left $(m+1)$-Calabi--Yau structure in the sense of \cite[Def.~6.2.2]{CDW23}.
	Shifting $\on{tot}^{\amalg}(\E_\ast)$ a degree down, we obtain the categorical chain complex $\on{tot}^{\amalg}(\E_\ast)[-1]$ (concentrated in degrees $m+2,\dots 0$) with a left $m$-Calabi--Yau structure. 

	Consider the categorical $(n+1)$-cube $\D_\ast=\left[\A_\ast\times \A_\ast'\times \A_\ast''\to \laxpush{\B_\ast}{\A_\ast'}{\C_\ast}\right]$, where the morphism of $n$-cubes defines the $(n+1)$-th coordinate. Using the notation of \cite[Constr.~4.4.1]{CDW23}, it is immediate that 
\[ \on{tot}_{n+1,n+2}^{\amalg}(\E_\ast)[-1] \simeq \D_\ast\,,\]
where the left $(n+1)$-cube is shifted down a degree in the $(n+1)$-th coordinate.
	It follows that $\on{tot}^{\amalg}(E_\ast)[-1]\simeq \on{tot}^{\amalg}(\D_\ast)$, so that $\on{tot}^{\amalg}(\D_\ast)$ inherits a left $(n+1)$-Calabi--Yau structure.
	By \cite[Rem.~6.4.13]{CDW23}, this determines the desired cubical left $(n+1)$-Calabi--Yau structure on $\D_\ast$.
\end{proof}

\begin{lemma}\label{lem:lax_gluing_commutes_with_colimits}~
\begin{enumerate}[(1)]
\item  The functor $\laxpush{(\mhyphen)}{(\mhyphen)}{(\mhyphen)} \colon \on{Fun}(\Lambda^2_0,\on{LinCat}_k)\to \on{LinCat}_k$ preserves colimits, and thus in particular cofibers.
\item The functor $\laxpush{(\mhyphen)}{(\mhyphen)}{(\mhyphen)} \colon  \on{Fun}(\Lambda^2_0,\on{Cube}_n(\on{LinCat}_k))\to \on{Cube}_n(\on{LinCat}_k)$ preserves colimits, and thus in particular cofibers.
\end{enumerate}
\end{lemma}

\begin{proof}
	Part (2) follows from part (1), using the equivalence 
\[ \on{Fun}(\Lambda^2_0,\on{Cube}_n(\on{LinCat}_k))\simeq \on{Fun}(I^n,\on{Fun}(\Lambda^2_0,\on{LinCat}_k))\] 
and that colimits in functor categories are computed pointwise \cite[\href{https://kerodon.net/tag/02X9}{Tag 02X9}]{Ker}.

	For the proof of part (1), it suffices to note that the functor that assigns to a span \[
\begin{tikzcd}
\A \arrow[r] \arrow[d] & \B \\
\C                     &   
\end{tikzcd}
\]
the diagram \eqref{eq:double_span} preserves colimits, which follows from inspecting the construction.
\end{proof}

	We note that it is also shown more generally in \cite[Thm.~5.6.5]{AHM26} that the passage to a (partially) lax colimit cone is a left adjoint and thus preserves colimits.

\subsection{Non-lax gluing}\label{subsec:non_lax_gluing}

	In this section, we deduce from \Cref{prop:laxgluingCYcubes} a non-lax gluing result for cubical Calabi--Yau structures. 

\begin{lemma}\label{lem:vertical_cofiber_is_CY}
    Consider a categorical $2$-cube in $\on{LinCat}_k^{\on{dual}}$
    \[\begin{tikzcd}
\A \arrow[d] \arrow[r] & \B \arrow[d] \\
\C \arrow[r]           & \D          
\end{tikzcd}
\]
equipped with a cubical left $m$-Calabi--Yau structure. Then the induced functor
\[ \C/\A\to \D/\B\]
inherits a canonical relative left $m$-Calabi--Yau structure.
\end{lemma}

\begin{proof}
	Denote the functor $\C\amalg_\A \B \to \D$ by $F$ and the functor $\C/\A\to \D/\B$ by $F'$. We first note that $\left(\C\amalg_{\A}\B\right)/\B\simeq \C/\A$.
	By \cite[Lemma 2.33.(1)]{Chr23}, the functors
\[ \Big(\left(\C\amalg_{\A}\B\right)\to\left(\C\amalg_{\A}\B\right)/\B\simeq \C/\A\Big)_!\,,\quad\quad \Big(\D\to \D/\B\Big)_!\]
map the inverse dualizing bimodules $\on{id}^!_{\C\amalg_{\A}\B}, \on{id}^!_\D$ to the inverse dualizing bimodules $\on{id}^!_{\C/\A},\on{id}^!_{\D/\B}$.  

	The fiber and cofiber sequence of $k$-linear endofunctors of $\D$ describing the relative Calabi--Yau structure of $F$
\[
\on{id}_{\D}^! \to F_!(\on{id}_{\C\amalg_{\A}\B}^!) \to \on{id}_{\D}[1-m]
\]
thus gets mapped by the exact functor $(\D\to \D/\B)_!$ to a fiber and cofiber sequence of endofunctors
\[
\on{id}_{\D/\B}^! \to F'_!(\on{id}_{\C/\A}^!) \to \on{id}_{\D/\B}[1-m]
\]
showing the non-degeneracy of the induced Hochschild homology class of $F'$. 
\end{proof}

\begin{lemma}\label{lem:half_cube_colimit_CY}~
\begin{enumerate}[(1)]
\item Let $\C_\ast \in \on{Cube}_n(\on{LinCat}_k^{\on{dual}})$ be a categorical $n$-cube equipped with a cubical left $m$-Calabi--Yau structure.
	Then $\on{colim}_{I^n\backslash \{1,\dots,1\}} \C_\ast\in \on{LinCat}_k^{\on{dual}}$ inherits a canonical left $(m-1)$-Calabi--Yau structure.
\item  Let $\left[ \A_\ast \to \B_\ast \right]\in \on{Cube}_{n+1}(\on{LinCat}_k^{\on{dual}})$ be a categorical $(n+1)$-cube equipped with a left $m$-Calabi--Yau structure.
	Then the functor
\begin{equation}\label{eq:A_to_B}
{\on{colim}_{I^n\backslash \{1,\dots,1\}} \A_\ast}\to {\on{colim}_{I^n\backslash \{1,\dots,1\}} \B_\ast}
\end{equation}
inherits a canonical left $(m-1)$-Calabi--Yau structure.
\end{enumerate}
\end{lemma}

\begin{proof}
	We begin with the proof of part (1).
	The $k$-linear $\infty$-category $\on{colim}_{I^n\backslash \{1,\dots,1\}} \C_\ast$ arises as the cofiber of the differential $\on{tot}^{\amalg}(\C_\ast)_2\to \on{tot}^{\amalg}(\C_\ast)_1$ of the coproduct totalization $\on{tot}^{\amalg}(\C_\ast)$, as follows from \cite[Lemma 6.4.9.(2)]{CDW23}, and thus inherits a left $(m-1)$-Calabi--Yau structure (by \cite[Definition 6.2.2]{CDW23} and the fact that the cofiber of a Calabi--Yau functor is a Calabi--Yau category). 

	Part (2) follows from a similar argument: we pass to the coproduct totalizations, obtaining the following commutative diagram with veritical cofiber sequences in $\on{LinCat}_k^{\on{dual}}$:
 \[
 \begin{tikzcd}
\on{tot}^{\amalg}(\A_\ast)_2/\on{tot}^{\amalg}(\A_\ast)_3 \arrow[d] \arrow[r] & \on{tot}^{\amalg}(\B_\ast)_2/\on{tot}^{\amalg}(\B_\ast)_3 \arrow[d] \\
\on{tot}^{\amalg}(\A_\ast)_1 \arrow[r] \arrow[d]                              & \on{tot}^{\amalg}(\B_\ast)_1 \arrow[d]                              \\
{\on{colim}_{I^n\backslash \{1,\dots,1\}} \A_\ast} \arrow[r]                  & {\on{colim}_{I^n\backslash \{1,\dots,1\}} \B_\ast}                 
\end{tikzcd}
 \]
 The upper square is left $(m-1)$-Calabi--Yau, as can be deduced from \cite[Thm.~6.4.6]{CDW23} and \cite[Lem.~6.3.5.(3)]{CDW23}. The desired Calabi--Yau-ness of the functor \eqref{eq:A_to_B} thus follows from \Cref{lem:vertical_cofiber_is_CY}.
 \end{proof}

\begin{lemma}\label{lem:bottom_CY_square_morphism_of_cubes}
	Let $\C_\ast=\left[\A_\ast\to \B_\ast \right]\in \on{Cube}_{n+1}(\on{LinCat}_k^{\on{dual}})$ be a categorical $(n+1)$-cube equipped with a left $m$-Calabi--Yau structure. Then the square
\[
\begin{tikzcd}
{\on{colim}_{I^n\backslash \{1,\dots,1\}} \A_\ast} \arrow[d] \arrow[r] & {\on{colim}_{I^n\backslash \{1,\dots,1\}} \B_\ast} \arrow[d] \\
{\A_{1,\dots,1}} \arrow[r]                                              & {\B_{1,\dots 1}}                                             
\end{tikzcd}
\]
inherits a canonical cubical left $m$-Calabi--Yau structure. 
\end{lemma}

\begin{proof}
	\Cref{lem:half_cube_colimit_CY} shows that $\on{colim}_{I^n\backslash \{1,\dots,1\}} \A_\ast$ is left $(m-2)$-Calabi--Yau and that the functor $\on{colim}_{I^n\backslash \{1,\dots,1\}} \A_\ast\to \on{colim}_{I^n\backslash \{1,\dots,1\}} \B_\ast$ is relative left $(m-1)$-Calabi--Yau.
	The functor $\on{colim}_{I^n\backslash \{1,\dots,1\}} \A_\ast \to \A_{1,\dots,1}$ is relative left $(m-1)$-Calabi--Yau since $\A_\ast$ is cubical left $(m-1)$-Calabi--Yau.
	The pushout morphism 
\[
\A_{1,\dots,1}\amalg_{\on{colim}_{I^n\backslash \{1,\dots,1\}} \A_\ast}{\on{colim}_{I^n\backslash \{1,\dots,1\}} \B_\ast}\to \B_{1,\dots,1}
\] 
is equivalent to the functor
\[\on{colim}_{I^{n+1}\backslash \{1,\dots,1\}} \C_\ast \to \C_{1,\dots,1}=\B_{1,\dots 1}\,,\] 
which is relative left $m$-Calabi--Yau by the assumption that $\C_\ast$ is cubical $m$-Calabi--Yau.
\end{proof}

\begin{proposition}\label{prop:cofiber_and_pushout_CY_cubes}~
	Consider a diagram in $\on{Cube}_n(\on{LinCat}_k^{\on{dual}})$ of the form
\[
\begin{tikzcd}
\A_\ast \arrow[r] \arrow[d] & \B_\ast \\
\C_\ast                     &        
\end{tikzcd}
\]
such that the $(n+1)$-cubes $\left[ \A_\ast \to \B_\ast\right]$ and $\left[ \A_\ast \to \C_\ast\right]$ carry compatible cubical left $m$-Calabi--Yau structures.
	Then the pushout cube $\B_\ast \amalg_{\A_\ast} \C_\ast$ inherits a canonical $m$-Calabi--Yau structure.
\end{proposition}

	We remark that \Cref{prop:cofiber_and_pushout_CY_cubes} can also be deduced from the results of \cite{BCS24} via the composition of morphisms in higher categories of iterated Calabi--Yau spans. 

\begin{proof}[Proof of \Cref{prop:cofiber_and_pushout_CY_cubes}]
	We first prove the assertion in the case that $\C_\ast\simeq 0$, so that the pushout is given by the cofiber cube $\B_\ast/\A_\ast$.
	Inspecting the recursive definition of cubical Calabi--Yau structures, we see that it suffices to prove for ${\bf i}\in I^n$ that the functor $\on{colim}_{I^n_{/{\bf i}}\backslash \{{\bf i}\}} \B_\ast/\A_\ast \to \B_{{\bf i}}/\A_{{\bf i}}$ inherits a relative left $m$-Calabi--Yau structure.
	This follows from combining \Cref{lem:vertical_cofiber_is_CY,lem:bottom_CY_square_morphism_of_cubes}.

	Let $\C_\ast$ again be arbitrary. From \Cref{prop:laxgluingCYcubes}, we deduce that the cube $\left[ \A_\ast \to \laxpush{\B_\ast}{\A_\ast}{\C_\ast}\right]$ has a cubical left $m$-Calabi--Yau structure. The cofiber cube is given by the pushout cube $\B_\ast\amalg_{\A_\ast}\C_\ast$ and thus inherits the desired cubical Calabi--Yau structure by the above. 
\end{proof}

\section{Calabi--Yau structures on constructible sheaves categories}\label{sec:CY_str_constr_sheaves}

	In \Cref{subsec:constr_sheaves}, we recall background on the $k$-linear stable $\infty$-category of constructible sheaves on a conically stratified space.
	In \Cref{section:conically_smooth}, we recall the notion of a conically smooth stratification, possibly with corners, and recall how these can be resolved into manifolds with corners via unzippings.
	We proceed in \Cref{subsec:cat_cube_of_stratified_space} by associating with every conically smooth stratified space $(X,P)$ the categorical cube $\on{Cons}_P(X)_\ast$, by lax gluing categorical cubes associated with the tubular neighborhoods of the strata.
	We use this lax gluing construction of $\on{Cons}_P(X)_\ast$ to produce cubical Calabi--Yau structures in \Cref{section:main_theorem}.
	In \Cref{subsec:manifolds_with_corners}, we apply the construction of $\on{Cons}_P(X)_\ast$ to reprove the result of Brav--Dyckerhoff on Calabi--Yau structures arising from local systems on smooth manifolds with boundary in the $\infty$-categorical setting.
	The proof is formulated in a way that generalizes to coefficients in an arbitrary $\mathbb{E}_\infty$-ring spectrum.
	In \Cref{subsec:example}, we spell out the constructions of this section in the example of the $2$-dimensional disk stratified by the coordinate axes. 
 
\subsection{Constructible sheaves on stratified spaces}\label{subsec:constr_sheaves}

	In this section, we recall some aspects of the theory of constructible sheaves.	Our main references are \cite{HA, PT22}.
\begin{notation}
	Throughout this section we will work with hypersheaves instead of sheaves. 	Since there is no risk of confusion, we will not make any reference to this in our notation. 
	Moreover, in the rest of the paper we will work with conically smooth stratified spaces, where the notions of sheaves and hypersheaves valued in a compactly generated $\infty$-category $\C$ coincide, as follows from \cite[Lemma A.2]{Vol25} using that $\on{Sh}(X; \C) \simeq \on{Sh}(X;\mathcal{S}) \otimes \C$ and $\on{Sh}^{\on{hyp}}(X;\mathcal{S}) \otimes \C \simeq \on{Sh}^{\on{hyp}}(X;\C)$.
\end{notation}

\begin{definition}
	A stratified space is the data of a continuous map $\rho \colon X \to P$ where $X$ is a topological space and $P$ is a poset endowed with the Alexandroff topology.
	A morphism of stratified spaces $(f,r) \colon (X,P) \to (Y,Q)$ is a commutative diagram as follows:
\[\begin{tikzcd}[column sep=scriptsize,row sep=small]
	X && Y \\
	\\
	P && Q
	\arrow["f", from=1-1, to=1-3]
	\arrow["{\rho_x}"', from=1-1, to=3-1]
	\arrow["{\rho_Y}", from=1-3, to=3-3]
	\arrow["r"', from=3-1, to=3-3]
\end{tikzcd}\]
	Abusing notation, we will often refer to a stratified space as a pair $(X,P)$, or sometimes simply $X$. For $p \in P$, the $p$-stratum of $X$ is defined as $X_p \coloneqq \rho^{-1}(p) \subset X$.
\end{definition}
	 
\begin{definition}[{$\!\!$\cite[Definition A.5.3]{HA}}]
	Let $(X,P)$ be a stratified space.
	The cone $(C(X), P \sqcup \left\lbrace - \infty \right\rbrace$) is the stratified space defined by:
	\begin{enumerate}\itemsep=0.2cm
    	\item as a set, $C(X) = \left\lbrace \ast \right\rbrace \sqcup (X \times \mathbb{R}_{>0}) $;
		\item a subset $U \subset C(X)$ is open if and only if $U \cap X \times \mathbb{R}_{>0}$ is open, and if $\ast \in U$ then there exist $\epsilon >0$ such that $X \times (0, \epsilon) \subset U$;
		\item the stratification of $C(X)$ is given by the map $\overline{\rho} \colon X \to P \sqcup \left\lbrace - \infty \right\rbrace$ defined by $\overline{\rho}(\ast) = - \infty$ and $\overline{\rho}(x,t) = \rho(x)$ for $(x,t) \in X \times \mathbb{R}_{>0}$.
	\end{enumerate}
\end{definition}

\begin{definition}[$\!\!${\cite[Definition A.5.5]{HA}}]\label{conicallly_strat_def}
	Let $(X,P)$ be a stratified space.
	We say that $(X,P)$ is conically stratified if for every $p \in P$ and $x \in X_p$, there exists a stratified space $(Y, P_{>p})$ and a topological space $Z$ such that there exists an open immersion $j \colon C(Y) \times Z \to X$ for which $x$ lies in the image of $j$ and the following diagram commutes:
	\[\begin{tikzcd}[sep=small]
	{C(Y) \times Z} && X \\
	{C(Y)} \\
	{P_{\geq p}} && P
	\arrow["j", hook, from=1-1, to=1-3]
	\arrow["{\pi_1}"', from=1-1, to=2-1]
	\arrow["{\rho_X}", from=1-3, to=3-3]
	\arrow["{\overline{\rho}_Y}"', from=2-1, to=3-1]
	\arrow[from=3-1, to=3-3]
\end{tikzcd}\]
\end{definition}

\begin{definition}[{$\!\!$\cite[Definition A.6.2]{HA}}]
    Let $(X,P)$ be a stratified space. 
    We define the exit path category $\on{Exit}(X,P)$ as the simplicial subset of $\on{Sing}(X)$ formed by the simplices $\sigma\colon  | \Delta^n | \to X$ such that there exists a chain $p_1 \leq \cdots \leq p_n$ of elements of $P$ such that for every $(t_0,\ldots ,t_i,0,\ldots,0) \in | \Delta^n |$ with $t_i > 0$, we have $\sigma(t_0,\ldots,t_i,0,\ldots,0) \in X_{p_i}$
\end{definition}

\begin{theorem}[{$\!\!$\cite[Theorem A.6.4]{HA}}]\label{thm:exit_paths}
    Let $(X,P)$ be a conically stratified space.
    Then $\on{Exit}(X,P)$ is an $\infty$-category.
\end{theorem}

\begin{notation}
    In the situations where $\on{Exit}(X,P)$ is an $\infty$-category, we will denote it by $\Pi_\infty(X,P)$ and refer to it as the stratified homotopy type of $(X,P)$.
\end{notation}

	For later reference, we recall the following finiteness conditions on conically stratified spaces.
\begin{definition}[{$\!\!$\cite[Definition 2.2.1]{PT22}}]\label{finite_strat}
	Let $(X,P)$ be a conically stratified space with locally weakly contractible strata.
\begin{enumerate}[(1)]
    \item We say that $(X,P)$ is categorically compact if $\Pi_\infty(X,P)$ is a compact object in $\on{Cat}_\infty$.
    \item We say that $(X,P)$ is locally categorically compact if for every $p \in P$ and every $x\in X_p$, there exists a conical chart of the form $Z\times \on{C}(Y)$ containing $x$ such that $(Y,P_{>p})$ is categorically compact and $Z$ is weakly contractible and locally weakly contractible.
\end{enumerate}
\end{definition}

\begin{definition}
	Let $(X,P)$ be a stratified space and let $\E \in \on{Cat}_\infty$. 
	For $p \in P$, denote by $i_p \colon X_p \to X$ the inclusion.
	An object $F \in \on{Sh}(X; \E)$ is constructible if $i_p^\ast F$ is locally constant for all $p \in P$.
	We denote by
	\[
	\on{Cons}_P(X; \E) \subset \on{Sh}(X ; \E)
	\]
	the full subcategory spanned by constructible sheaves. In the case $\E=\on{Mod}_k$, we also write $\on{Cons}_P(X)$ for $\on{Cons}_P(X;\on{Mod}_k)$.
\end{definition}

\begin{definition}
	Let $(X,P)$ be a stratified space and let $\E \in \on{Cat}_\infty$.
	We let 
\[
\on{Cons}_{P, \omega}(X; \E) \subset \on{Cons}_P(X; \E)
\]
be the full subcategory spanned by constructible sheaves  with compact stalks.
\end{definition}

	The following result is a natural generalization of the monodromy equivalence to the stratified setting. 
	We will refer to it as the exodromy equivalence.
	
\begin{theorem}[{$\!\!$\cite[Theorem 5.4.1 \& Remark 5.4.2]{PT22}}]\label{thm:exodromy_equivalence}
    Let $(X,P)$ be a conically stratified space with locally weakly contractible strata and let $\E$ be a compactly generated $\infty$-category.
    Then we have an equivalence
    \[
    \on{Cons}_P(X; \E) \simeq \on{Fun}(\Pi_\infty(X,P), \E).
    \]
\end{theorem}

\begin{remark}
    A $2$-categorical version of the exodromy equivalence first appeared in \cite{Tr09}, where it is presented as a generalization of an unpublished $1$-categorical version due to MacPherson.
    Lurie proved \Cref{thm:exodromy_equivalence} for $\E = \mathcal{S}$ in \cite[Appendix A]{HA}. 
    Other and more general versions of the exodromy equivalence appeared in \cite{HPT24, CJ24, Jan24}.
\end{remark}

\begin{proposition}[{$\!\!$\cite[Corollary 6.7.2]{PT22}}]\label{prop:recollements}
	Let $(X,P)$ be a conically stratified space with locally weakly contractible strata and let $S \subset P$ be a closed subset and let $U \coloneqq P \setminus S$.
	Consider the corresponding open and closed immersions $j \colon X_U \hookrightarrow X$ and $i \colon X_S \to X$.
	Let $\E$ be a presentable $\infty$-category.
	The fully faithful functors
\[
i_\ast \colon \on{Cons}_S(X_S; \E) \rightarrow \on{Cons}_P(X; \E) \leftarrow \on{Cons}_U(X_U; \E) \colon j_\ast
\]
	extend to a recollement of $\on{Cons}_P(X; \E)$ into $\on{Cons}_S(X_S; \E)$ and $\on{Cons}_U(X_U; \E)$.
\end{proposition}

\subsection{Conically smooth stratified spaces}\label{section:conically_smooth}

	A conically smooth stratified space, as introduced in \cite[Def.~3.2.21]{AFT17}, consists of a $C^0$-stratified space $\rho\colon X\to P$ together with a choice of conically smooth atlas.
	The notion of a $C^0$-stratified space is a refinement of the notion of a conically stratified space, satisfying that the strata are topological manifolds.
	We refer to \cite{AFT17} for more details.

	We note that any smooth manifold equipped with the trivial stratification is conically smooth, see \cite[Ex.~3.2.15]{AFT17}.
	The strata of a conically smooth stratified space are smooth manifolds, see \cite[Lem.~3.4.5]{AFT17}.
	We can thus assign to each conically smooth stratified space $X$ a dimension $\on{dim}(X)$, given by the maximal dimension of the strata. 

\begin{definition}
	Let $X$ be a conically smooth stratified space. 
\begin{enumerate}
    \item The depth $\on{depth}(Y)$ of a stratum $Y$ is defined as $\on{dim}(X)-\on{dim}(Y)$.
    \item The depth of $X$ is defined as the maximum depth of the strata of $X$.
\end{enumerate}
\end{definition}

	The notion of a conically smooth stratified space can be further refined to a conically smooth stratified space with $\langle m\rangle$-corners: 

\begin{definition}
	Let $m\geq 1$ and $[m]=\{1,\dots,m\}$. Let $\mathcal{P}([m])$ be the power set of $[m]$, considered as a poset via reverse inclusion.
	A conically smooth stratified space with $\langle m\rangle$-corners consists of a conically smooth stratified space $\rho\colon X\to P$ together with a morphism of posets $P\to \mathcal{P}([m])$, such that the composite $\rho'\colon X\to \mathcal{P}([m])$ is topologically coCartesian in the sense of \cite[Def.~7.2.2]{AFT17}, meaning that for $T\in \mathcal{P}([m])$ the inclusion $(\rho')^{-1}(T)\subset X$ extends to a conically smooth open embedding 
\[
\mathbb{R}_{\geq 0}^{\times |T|}\times (\rho')^{-1}(T)\subset X\,.
\]
	A conically smooth stratified space with $\langle 1 \rangle$-corners is also called a conically smooth stratified space with boundary.
\end{definition}

\begin{definition}
	Given a conically smooth stratified space with $\langle m\rangle$-corners $\rho'\colon X\xrightarrow{\rho} P\to \mathcal{P}([m])$, we denote by $X^\circ=(\rho')^{-1}(\emptyset)$ the interior of $X$, which inherits the structure of a conically smooth stratified space by restricting $\rho$. 
\end{definition}

\begin{remark}\label{rem:finite_poset}
	Let $(X,P)$ be a conically smooth stratified space. Then $P$ must be a locally finite poset, and finite if $X$ is compact. This is proven by Volpe in \cite[Lemma 3.1]{Vol24}.
\end{remark}

\begin{notation}\label{notation:Cons_corners}
	Let $(X,P)$ be a conically smooth stratified space with $\langle m \rangle$-corners and $\E$ be an $\infty$-category.
	Exceptionally, we denote the category of constructible sheaves on $X$ with coefficient in $\E$ by $\cons_{P, \partial}(X;\E)$.
\end{notation}

\begin{warning}\label{warning:notation_Cons_corners}
	In the setting of \Cref{notation:Cons_corners}, the notation $\cons_{P}(X;\E)$ will be reserved for a full subcategory of $\cons_{P, \partial}(X;\E)$, that we introduce next.
\end{warning}

	Let $(X,P)$ a conically smooth stratified space with $\langle m \rangle$-corners.
	Consider the set of arrows $W_\partial \subset \Pi_\infty(X,P)_1$ defined as
\begin{align*}
W_\partial \coloneqq \left\lbrace \gamma \in \Pi_\infty(X,P)_1 \ \left| \ 
\begin{aligned}
& \exists U = \mathbb{R}^i \times (\mathbb{R}_{\geq 0})^j \times \on{C}(Z) \ \text{local chart such that} \ \on{Im}(\gamma)\subset U, \ \text{and} \\
& \gamma=(y, \gamma', w) \ \text{where} \ y \in \mathbb{R}^i, w \in \on{C}(Z), \gamma' \in \Pi_\infty((\mathbb{R}_{\geq 0})^j, \mathcal{P}([j]))_1
\end{aligned}
\right.
\right\rbrace
\ .
\end{align*}

\begin{definition}\label{def:cons_P_with_corners}
	Let $(X,P)$ be a conically smooth stratified space with $\langle m \rangle$-corners and $\E$ be a compactly generated $\infty$-category.
	We let
\[
\cons_{P}(X;\E) \subset \cons_{P, \partial}(X;\E)
\]
be the full subcategory corresponding to $\on{Fun}(\Pi_\infty[W_\partial^{-1}], \E)$ under the equivalence of \Cref{thm:exodromy_equivalence}.
\end{definition}

\begin{remark}\label{rem:corners_or_interior}
Let $(X,P)$ be a conically smooth stratification with $\langle m\rangle$-corners. By definition, the $\infty$-category $\cons_{P}(X)$ is the full subcategory of $\cons_{P, \partial}(X)$ spanned by those constructible sheaves that are locally constant in the normal directions emanating from the corners.
	Using the existence of collar neighborhoods \cite[Lemma 8.2.1]{AFT17}, it is easy to see that $\cons_{P}(X)$ is equivalent to the $\infty$-category $\on{Cons}_P(X^\circ)$ of constructible sheaves on the interior stratified space $X^\circ$.
\end{remark}

\begin{definition}
	Let $m\geq 1$.
	A smooth manifold with $\langle m\rangle$-corners (or smooth $\langle m\rangle$-manifold for short) consists of a conically smooth stratified space with $\langle m\rangle$-corners $\rho'\colon X\xrightarrow{\rho} P\xrightarrow{c} \mathcal{P}([m])$, where $c$ is a partial order preserving bijection. 
\end{definition}

	For $T\in \mathcal{P}([m])$, let $\mathcal{P}([m])_{\leq T}\subset \mathcal{P}([m])$ be the sub-poset with elements $\{T'\in \mathcal{P}([m])|T'\leq T\}$.

\begin{definition}
	Given a smooth manifold with $\langle m\rangle$-corners $\rho\colon X\to \mathcal{P}([m])$, the closed face associated with $T\in \mathcal{P}([m])$ is the conically smooth stratified space with $\langle m-|T|\rangle$-corners \[\rho^{-1}(\mathcal{P}([m])_{\leq T})\to \mathcal{P}([m])_{\leq T}\simeq \mathcal{P}([m-|T|])\,.\]
	If $T=\{1,\dots,m\}\backslash \{i\}$, we also write $\partial_i X$ for the associated face.
\end{definition}

	The faces of a smooth manifold $X$ with $\langle m\rangle$-corners are thus again smooth manifolds with corners.
	Furthermore, the topological space underlying $X$ inherits the structure of a smooth manifold with boundary given by the union of all faces.

\begin{remark}
	Given a smooth manifold with $\langle m\rangle$-corners $X$, we have that $\partial_i X\cap \partial_j X$ is a face of both $\partial_iX$ and $\partial_j X$.
	Thus, a conically smooth manifold with $\langle m\rangle$-corners gives rise to an $\langle m\rangle$-manifold in the sense of \cite{Jan68,Lau00}.
	This is also noted in \cite[Ex.~7.2.12]{AFT17}.
\end{remark}

We proceed with recalling the unzippings of conically smooth stratified spaces with corners.

\begin{definition}[{$\!\!$\cite[Def.~7.3.2]{AFT17}}]
	Let $X$ be a conically smooth stratified space with $\langle m\rangle$-corners whose interior is of depth $n$. The unzipping $\on{Unzip}_n(X)$ is defined as the terminal object of the $1$-category $\on{Res}_{\bf X}^n$ with 
\begin{itemize}
    \item objects given by conically smooth stratified spaces $Y$ with $\langle m+1\rangle$-corners, together with a conically smooth map $p\colon Y\to X$, such that the square
    \[
    \begin{tikzcd}
Y \arrow[d] \arrow[r, "p"]                                  & X \arrow[d]        \\
{\mathcal{P}([m+1])} \arrow[r, "{(\mhyphen)\cap [m]}"] & {\mathcal{P}([m])}
\end{tikzcd}
    \]
    commutes and $\partial_{m+1}Y$ agrees with the inverse image under $p$ of the largest stratified subspace with corners of $X$ whose interior is the deepest stratum $X^{\circ,n}$ of $X^\circ$.
    \item morphisms given by the conically smooth maps $Y\to Y'$ over $\mathcal{P}([m+1])\times X$.  
\end{itemize}
\end{definition}

\begin{definition}\label{def:iterated_unzip}
	Let $X$ be a conically smooth stratified space of depth $n$ with $\langle m\rangle$-corners.
\begin{enumerate}[(1)]
    \item  We define the $k$-th unzipping as 
    \[ \on{Unzip}^{[k,n]}(X)\coloneqq \on{Unzip}_k\circ \dots \circ \on{Unzip}_n(X) \,.\]
    It comes with a conically smooth map $\pi_k\colon \on{Unzip}^{[k,n]}(X)\to X$. We set $\on{Unzip}^{[n+1,n]}(X)=X$.

    We will denote the depth $(k-1)$-stratum of $\on{Unzip}^{[k,n]}(X)$ (this is the lowest dimensional stratum) by $\widetilde{X^{k-1}}\subset \on{Unzip}^{[k,n]}(X)$.
    \item We define the $k$-th link $\on{Link}_k(X)$ as $\partial_{m+n}\on{Unzip}^{[k,n]}(X)$. Note that $\on{Link}_k(X)$ is the inverse image of $\widetilde{X^{k}}$ under the map $\on{Unzip}^{[k,n]}(X)\to \on{Unzip}^{[k+1,n]}(X)$. 
    
    The restriction of $\pi_k$ defines a map $\tilde{\pi}_k\colon \on{Link}_k(X)\to \widetilde{X^{k}}$.
\end{enumerate}    
\end{definition}

\begin{proposition}[{\!\!\cite[Prop.~8.2.3]{AFT17}}]\label{prop:tubular_neighborhood}
    Let $X$ be a conically smooth stratified space of depth $n$ with $\langle m\rangle$-corners.
    For any $1\leq k\leq n$, the deepest stratum $\widetilde{X^{k-1}} \subset \on{Unzip}^{[k,n]}(X)$ admits a tubular neighborhood $T^\circ(\widetilde{X^{k}})$, given by a conically smooth open embedding 
    \[ T^\circ(\widetilde{X^{k}})=C(\tilde{\pi}_k)\hookrightarrow \on{Unzip}^{[k,n]}(X)\,.\]
\end{proposition}

\begin{remark}\label{rem:tubular_neigh}
	In the setting of \Cref{prop:tubular_neighborhood}, we denote the closed version of the tubular neighborhood by $\on{T}(\widetilde{X^k})\subset T^\circ(\widetilde{X^{k}})$, defined as the cone of $\tilde{\pi}_k$ over $[0,1]\subset \mathbb{R}_{\geq 0}$.
	We note that $\on{T}(\widetilde{X^k})$ is canonically a compact smooth $\langle n+m-k+1\rangle$-manifold.
\end{remark}

\begin{lemma}[{\!\!\cite[Ex.~8.3.5]{AFT17}}]\label{lem:unzip_gluing}
    Let $X$ be a conically smooth stratified space of depth $n$ with $\langle m\rangle$-corners.
    Then there is a homeomorphism
    \begin{equation}\label{eq:Unzip_gluing}
    \on{Unzip}^{[k+1,n]}(X)\simeq \on{T}(\widetilde{X^{k}}) \underset{\mathbb{R}\times \on{Link}_k(X)}{\amalg} \on{Unzip}^{[k,n]}(X)\,.
    \end{equation}
\end{lemma}

\begin{remark}\label{rem:lax_gluing_tub_neighb}
	Let $X$ be a conically smooth stratified space of depth $n$ with $\langle m\rangle$-corners.
	There exists a homeomorphism $\on{Unzip}_n(X)\backslash \on{Link}_n(X) \simeq X\backslash X_n$ , see \cite[Prop.~7.3.10]{AFT17}. 

	By the existence of collar neighborhoods \cite[Lemma 8.2.1]{AFT17}, we thus find that $\on{Unzip}^{[k,n]}(X)$ is homeomorphic to a subspace of $X$.
	Similarly, we find that $\on{Unzip}^{[k,n]}(X)$ is homeomorphic to a subspace of $\on{Unzip}^{[k+1,n]}(X)$.
	We can choose these embeddings such that the arising diagram
\[
\begin{tikzcd}
{\on{Unzip}^{[k,n]}(X)} \arrow[rd, hook] \arrow[rr, hook] &   & {\on{Unzip}^{[k+1,n]}(X)} \arrow[ld, hook] \\
                                                          & X &                                           
\end{tikzcd}
\]
commutes for all $1\leq k\leq n$.
	Furthermore, we can embed $\on{T}(\widetilde{X^k})$ in $\on{Unzip}^{[k+1,n]}(X)$ such that the pushout in \eqref{eq:Unzip_gluing} becomes an equality of subspaces of $X$:
\[
\on{Unzip}^{[k+1,n]}(X) \simeq \on{T}(\widetilde{X^k}) \underset{\mathbb{R}_{\geq 0}\times \on{Link}_k(X)}{\bigcup} \on{Unzip}^{[k,n]}(X)\,.
\]
	By \cite[Lemma 7.3.5-(3)]{AFT17}, the above equivalence is compatible with restriction to the corners of $\on{Unzip}^{[k+1,n]}(X)$.
\end{remark}

\begin{remark}
    Let $X$ be a conically smooth stratified space of depth $n$ with $\langle m\rangle$-corners, such that the topological space underlying $X$ is a topological manifold with boundary.
    Then the topological space underlying each unzipping $\on{Unzip}^{[k,n]}(X)$ is also a topological manifold with boundary.
    This follows from the existence of collar neighborhoods \cite[Lemma 8.2.1]{AFT17} and the fact that each open subset of a topological manifold with boundary is a topological manifold with boundary. 

    An orientation of the topological manifold with boundary $X$ thus induces orientations on the manifolds with boundary $\on{Unzip}^{[k,n]}(X), \on{T}(\widetilde{X^k}), \mathbb{R}_{\geq 0}\times \on{Link}_k(X)$ for all $1\leq k\leq n$, and the chosen embeddings of these into $X$ respect the orientations.
\end{remark}

\begin{lemma}\label{lm:pushout}
	Let $(X,P)$ be a conically smooth stratified space with $\langle m \rangle$-corners and let $X_p$ be the deepest stratum in $X^\circ$, with depth $n$.
	The diagram of stratified spaces, with stratification induced by $X$,
\[\begin{tikzcd}[sep=small]
	{\on{Link}_n(X) \sqcup (\on{Link}_n(X) \times \mathbb{R}_{>0})} && {\on{Link}_n(X) \times \mathbb{R}_{\geq 0}} \\
	\\
	{X_p \sqcup (X \setminus \on{T}(X_p))} && X
	\arrow[from=1-1, to=1-3]
	\arrow[from=1-1, to=3-1]
	\arrow[from=1-3, to=3-3]
	\arrow[from=3-1, to=3-3]
\end{tikzcd}\]
	induces a pushout square in $\on{Cat}_\infty$
	\[\begin{tikzcd}[sep=small]
	{\Pi_\infty(\on{Link}_{n}(X), P \setminus \{p\}) \sqcup \Pi_\infty(\on{Link}_{p}(X), P \setminus \{p\})} && {\Pi_\infty(\on{Link}_{n}(X) \times \mathbb{R}_{\geq 0}, (P \setminus \{p\})^{\vartriangleright})} \\
	\\
	{\Pi_\infty(X_p, \{p\}) \sqcup \Pi_\infty(X \setminus \on{T}(X_p), P \setminus \{p\})} && {\Pi_\infty(X,P)}
	\arrow[from=1-1, to=1-3]
	\arrow[from=1-1, to=3-1]
	\arrow[from=1-3, to=3-3]
	\arrow[from=3-1, to=3-3]
\end{tikzcd}\,.\]
\end{lemma}

\begin{proof}
	We mimic the proof of \cite[Prop.~3.3.8]{AFR19}. By Van-Kampen theorem for exit paths (\cite[Prop.~3.3.8-(1)]{AFR19}, or \cite[Th.~5.1.7-(4)]{HPT24}) we can reduce the statement to a local one.
	Hence we are reduced to prove that the square of stratified spaces 
\[\begin{tikzcd}[sep=small]
	{(\mathbb{R}^{i} \times (\mathbb{R})_{\geq 0}^{j} \times \{0\} \times Z)  \sqcup (\mathbb{R}^{i} \times (\mathbb{R})_{\geq 0}^{j} \times \{0\} \times Z \times \mathbb{R}_{>0})} && {\mathbb{R}^{i} \times (\mathbb{R})_{\geq 0}^{j} \times \{0\} \times Z \times \mathbb{R}_{\geq 0}} \\
	\\
	{(\mathbb{R}^{i} \times (\mathbb{R})_{\geq 0}^{j}) \sqcup (\mathbb{R}^{i} \times (\mathbb{R})_{\geq 0}^{j} \times R_{>\epsilon} \times Z)} && {\mathbb{R}^{i} \times (\mathbb{R})_{\geq 0}^{j} \times \on{C}(Z)}
	\arrow[from=1-1, to=1-3]
	\arrow[from=1-1, to=3-1]
	\arrow[from=1-3, to=3-3]
	\arrow[from=3-1, to=3-3]
\end{tikzcd}\]
where $j \leq m$, $0 < \epsilon \ll 1$, and $Z$ is a conically smooth stratified space.
	One immediately checks that this is a pushout of stratified spaces and that the exit path category functor respects this pushout.
\end{proof}

\begin{corollary}\label{cor:lax_pushout_Unzip}
	Let $X$ be a conically smooth stratified space with $\langle m \rangle$-corners and let $X_p$ be the deepest stratum in $X^\circ$, with depth $n$. \Cref{lm:pushout} induces equivalences in $\on{LinCat}_k$
\[
\on{Cons}_{P, \partial}(X;\on{Mod}_k)
\simeq 
\laxpush{\on{Loc}(\on{T}(X_p);\on{Mod}_k)}{\Cons_{P \setminus \{p\}, \partial}(\on{Link}_{n}(X);\on{Mod}_k)}{\Cons_{P \setminus \{p\}, \partial}(X \setminus T^\circ(X_p);\on{Mod}_k)}\,,
\]
and
\[
\on{Cons}_{P}(X;\on{Mod}_k)
\simeq 
\laxpush{\on{Loc}(\on{T}(X_p);\on{Mod}_k)}{\Cons_{P \setminus \{p\}}(\on{Link}_{n}(X);\on{Mod}_k)}{\Cons_{P \setminus \{p\}}(X \setminus T^\circ(X_p);\on{Mod}_k)}\,.
\]
\end{corollary}

\begin{warning}\label{rem:open_unzip_notations}
	Notice that the stratifications on the right hand side of the equivalence in \Cref{cor:lax_pushout_Unzip} do not agree with the stratification induced by $(X,P)$.
	They agree instead with the stratifications induced by that of \Cref{lm:pushout}.
	For example, the induced stratification on $\on{T}(X_p) \subset X$ is non trivial.
	Precisely, we consider:
	\begin{itemize}
	\item  $\on{T}(X_p)$ with the trivial stratification, which agrees with the stratification of its homotopy equivalent space $X_p$;
	\item $X \setminus T^\circ(X_p)$ with the stratification induced by its homotopy equivalent space $X \setminus T(X_p)$;
	\item $X$ and $\on{Link}_{p}(X)$ with their natural stratifications.
\end{itemize}
\end{warning}

\begin{proof}[Proof of \Cref{cor:lax_pushout_Unzip}]
	In order to ease the notation, we drop the decoration $\on{Mod}_k$ in this proof.
	We need to prove that there is a (non-lax) pushout square in $\on{LinCat}_k$:
	\[\begin{tikzcd}[sep=small]
	{\Cons_{P \setminus \{p\}}(\on{Link}_k(X)) \sqcup \Cons_{P \setminus \{p\}}(\on{Link}_k(X))} && {\on{Fun}(\Delta^1,\Cons_{P \setminus \{p\}}(\on{Link}_k(X)))} \\
	\\
	{\Loc(\on{T}(X_p)) \sqcup \Cons_{P \setminus \{p\}}(X \setminus T^\circ(X_p))} && {\on{Cons}_{P \setminus \{p\}}(X \setminus \on{T}^\circ(X))}
	\arrow[from=1-1, to=1-3]
	\arrow[from=1-1, to=3-1]
	\arrow[from=1-3, to=3-3]
	\arrow[from=3-1, to=3-3]
\end{tikzcd}\]
	We have equivalences in the symmetric monoidal $\infty$-category $\on{LinCat}_k$
\begin{align*}
\on{Fun}(\Delta^1,\Cons_{P \setminus \{p\}}(\on{Link}_k(X))) & \simeq \on{Fun}(\Delta^1,\on{Mod}_k) \otimes \Cons_{P \setminus \{p\}}(\on{Link}_k(X);\on{Mod}_k) \simeq \\
& \simeq \Cons_{0<1}(\mathbb{R}_{\geq 0}) \otimes \Cons_{P \setminus \{p\}}(\on{Link}_k(X)) \simeq \\
& \simeq \Cons_{\{0<1\} \times (P \setminus \{p\})}(\on{Link}_k(X) \times \mathbb{R}_{\geq 0}) \,.
\end{align*}
where the last equivalence follows from \cite[Prop.~5.2.11]{HPT24}.
	By the homotopy invariance of constructible sheaves \cite{HPT23}, we have $\Loc(\on{T}(X_p)) \simeq \Loc(X_p)$ and $\on{Cons}_{P \setminus \{p\}}(X \setminus \on{T}^\circ(X)) \simeq \on{Cons}_{P \setminus \{p\}}(X \setminus \on{T}(X))$.
	Hence the above diagram can be rewritten as
		\[\begin{tikzcd}[sep=small]
	{\Cons_{P \setminus \{p\}}(\on{Link}_k(X)) \sqcup \Cons_{P \setminus \{p\}}(\on{Link}_k(X))} && {\Cons_{\{0<1\} \times (P \setminus \{p\})}(\on{Link}_k(X) \times \mathbb{R}_{\geq 0})} \\
	\\
	{\Loc(X_p) \sqcup \Cons_{P \setminus \{p\}}(X \setminus T(X_p))} && {\on{Cons}_{P \setminus \{p\}}(X \setminus \on{T}^\circ(X))}
	\arrow[from=1-1, to=1-3]
	\arrow[from=1-1, to=3-1]
	\arrow[from=1-3, to=3-3]
	\arrow[from=3-1, to=3-3]
\end{tikzcd}\]
	The first equivalence then follows by applying $\on{Fun}(-;\on{Mod}_k)$ to the pushout square of \Cref{lm:pushout} and \Cref{thm:exodromy_equivalence}.
	By \Cref{rem:corners_or_interior}, the second equivalence is proved in the same way by replacing $(\mathbb{R}_{\geq 0})^j$ with $(\mathbb{R}_{> 0})^j$ everywhere.
\end{proof}

	The work of Volpe \cite{Vol25, Vol24} shows that conically smooth manifolds enjoy nice finiteness conditions. Indeed we have the following:

\begin{theorem}\label{thm:finiteness_conically_smooth}
	Conically smooth stratified spaces are locally categorically compact.
	Compact conically smooth stratified spaces are categorically compact.
	The same holds for the interior of a conically smooth stratified space with corners.
\end{theorem}

\begin{proof}
	The second statement is \cite[Thm.~3.10]{Vol24}.
	The first statement is a local one and it follows from the first and the fact that a conically smooth stratified spaces admits conical charts of the form $\on{C(Z)} \times \mathbb{R}^n$ with $(Z,Q)$ a compact conically smooth stratified space, see \cite[Cor.~7.3.7]{AFT17}.
	The last assertion is \cite[Cor.~2.20]{Vol25}, combined with the first one.
\end{proof}

	The following result of Nocera-Volpe provides a large class of examples of conically smooth stratified spaces.
\begin{theorem}[{$\!\!$\cite[Prop.~2.9 \& Thm.~3.7]{NV23}}]\label{thm:Whitney_implies_conically_smooth}
	Whitney stratified spaces are conically smooth.
\end{theorem}

\subsection{The categorical cube of a conically smooth stratified space}\label{subsec:cat_cube_of_stratified_space}

The main purpose of this section is to carry out \Cref{construction:cube}, which associates a categorical cube $\cons(X)_\ast$ to each compact conically smooth stratified space $(X,P)$ with (possibly empty) boundary. 
	
	We can summarize \Cref{construction:cube} as follows. Let $n$ be the depth of the stratification. 
\begin{itemize}
\item The unzipping $\on{Unzip}^{[1,n]}(X)$ is a smooth $\langle n+1\rangle$-manifold, to which we associated a categorical $n+1$-cube of local systems in \Cref{def:cube_local_systems}.
	The tip of this cube is given by $\on{Loc}(\on{Unzip}^{[1,n]}(X))\simeq \on{Loc}(X^0)$, with $X^0$ being the depth $0$ stratum, that is the top dimensional stratum of $X$. 
\item More generally, each closed tubular neighborhood $T(\widetilde{X^k})$ of the depth $k$-stratum $\widetilde{X^k}$ of $\on{Unzip}^{[n-k+1,n]}(X)$ is itself a manifold with $\langle n-k+2\rangle$ corners, and thus determines a categorical $(n-k+2)$-cube with tip $\on{Loc}(T(\widetilde{X^k}))\simeq \on{Loc}(\widetilde{X^k})$.
	We will extend this categorical $(n-k+2)$-cube to a categorical $n$-cube by filling it with $0$'s.
\item We can reconstruct $X$ from $\on{Unzip}^{[1,n]}(X)$ by iteratively gluing in the tubular neighborhoods of the lower strata $T(\widetilde{X^k})$, see \Cref{lem:unzip_gluing}.
	The corresponding $\infty$-category of constructible sheaves $\on{Cons}_P(X;\on{Mod}_k)$ arises via the iterated directed pushout of the categories of local systems on the strata of $X$.
	We consider a sequence of directed pushouts of the categorical $n$-cubes associated with the strata of $X$, yielding the $n$-cube $\on{Cons}_P(X)_\ast$.
	Restricting to the tips, we recover the directed pushouts describing $\on{Cons}_P(X;\on{Mod}_k)$, and thus find that the tip is given by
\[
\Cons_P(X)_{1,\dots,1}\simeq \Cons_P(X;\on{Mod}_k)\,.
\]
\end{itemize}

Recall that $I=\{0\to 1\}$. Given a compact smooth $\langle n\rangle$-manifold $X$, we denote by $\underline{X}$ the functor $I^n\to \on{Top}$ defined by
\[
\underline{X}((i_1,\dots,i_n))=\begin{cases} X & (i_1,\dots,i_n)=(1,\dots,1)\\  \bigcap_{1\leq j\leq n~\text{s.t.}~i_j=0} \partial X_j & (i_1,\dots,i_n)\not =(1,\dots,1)\,. \end{cases}
\]
We have for any ${\bf j}\in I^n$ 
\[
\bigcup_{{\bf i}<{\bf j}} \underline{X}({\bf i})=\partial \underline{X}({\bf j})\subset  \underline{X}({\bf j})\,. 
\]

We denote by $\loc(\mhyphen)\colon \mathcal{S}\to \on{LinCat}_k$ the essentially unique colimit preserving functor, determined by mapping $\ast$ to $\on{Mod}_k$. We note that for any space $X\in \mathcal{S}$, we have $\loc(X)\simeq \on{Fun}(X,\on{Mod}_k)$. 

\begin{definition}\label{def:cube_local_systems} 
	Let $X$ be a compact smooth $\langle n \rangle$-manifold.
	We define the categorical $n$-cube 
	\[
	\Loc(X)_\ast \coloneqq \Loc \circ \underline{X} \colon I^n \to \on{LinCat}_k\]
	\[
	(i_1, \ldots, i_n) \mapsto \begin{cases} \Loc(X) & (i_1,\dots,i_n)=(1,\dots,1) \\ \Loc\big(\bigcap_{1\leq j\leq n~\text{s.t.}~i_j=0} \partial X_j\big) & (i_1,\dots,i_n)\not =(1,\dots,1)\end{cases}\,. 
	\]
\end{definition}

	Recall the conventions for constructible sheaves on conically smooth stratified spaces with corners form \Cref{def:cons_P_with_corners} (see also \Cref{notation:Cons_corners} and \Cref{warning:notation_Cons_corners}).

\begin{construction}\label{construction:cube}
	Let $\rho \colon X \to P$ be a conically smooth stratified space with boundary whose underlying space $X$ is a compact topological $d$-manifold with boundary. We let $n$ be the depth of the stratification of $X$.
	Denote by $X^k$ the depth $k$-stratum of $(X,P)$.
	We will use the notation from \Cref{section:conically_smooth} in the following.
	Let $\on{Loc}(\on{Unzip}^{[1,n]}(X))_\ast, \on{Loc}(\on{T}(\widetilde{X^1}))_\ast$ be the categorical $(n+1)$-cubes of local systems associated to the smooth $\langle n+1 \rangle$-manifolds $\on{Unzip}^{[1,n]}(X), \on{T}(\widetilde{X^1})$, as in  \Cref{def:cube_local_systems}.
	Unraveling the definitions, we find a canonical equivalence of the faces
\[
\Loc(\on{Unzip}^{[1,n]}(X))_{\partial_1^0 \ast} \simeq \Loc(\on{T}(\widetilde{X^1}))_{\partial_1^0 \ast}\,.
\]
	This determines a morphism of categorical $n$-cubes
\[
\Loc(\on{Unzip}^{[1,n]}(X))_{\partial_1^0 \ast} \to \Loc(\on{T}(\widetilde{X^1}))_{\partial_1^1 \ast}\,.
\]
	We define the $(n+1)$-cube $\Cons(\on{Unzip}^{[2,n]}(X))_\ast$ as follows:
\[
\big[ \Loc(\on{Unzip}^{[1,n]}(X))_{\partial_1^0 \ast} \to \Loc(\on{T}(\widetilde{X^1}))_{\partial_1^1 \ast} \bigsqcup^\rightarrow_{\Loc(\on{Unzip}^{[1,n]}(X))_{\partial_1^0 \ast}} \Loc(\on{Unzip}^{[1,n]}(X))_{\partial_1^1 \ast} \big]
\]
	By \Cref{rem:functor_into_directed_pushout}, the pushout of the span of categorical $n$-cubes
\[\begin{tikzcd}[sep=small]
	{\Cons(\on{Unzip}^{[2,n]}(X))_{\partial_1^0 \ast} \simeq \Loc (\on{Unzip}^{[1, n]}(X))_{\partial_1^0 \ast}} && 0_\ast \\
	\\
	{\Cons(\on{Unzip}^{[2,n]}(X))_{\partial_1^1 \ast}}
	\arrow[from=1-1, to=1-3]
	\arrow[from=1-1, to=3-1]
\end{tikzcd}\]
	is equivalent to the pushout $n$-cube
\[
\Loc(\on{T}(\widetilde{X^1}))_{\partial_1^1 \ast} \bigsqcup_{{\Loc}(\on{Unzip}^{[1,n]}(X))_{\partial_1^0 \ast}} \Loc(\on{Unzip}^{[1,n]}(X))_{\partial_1^1 \ast}
\]
	Using \Cref{rem:lax_gluing_tub_neighb}, the latter can be canonically identified with the cube of local systems $\Loc(\on{Unzip}^{[2, n]}(X))_\ast$ associated with the smooth $\langle n \rangle$-manifold $\on{Unzip}^{[2, n]}(X)$. \\ \indent
	We proceed iteratively as follows.
    
    We suppose that we have defined $\cons(\on{Unzip}^{[k, n]}(X))_\ast$, with $0 < k <n+1$. 
    We  make the following further assumption:\\

    \noindent {\bf Assumption $(k)$.} There exists an equivalence of categorical $(n-k+2)$-cubes
    \[\on{cof}_1^{\times k-1}\cons(\on{Unzip}^{[k, n]}(X))_\ast \simeq \Loc(\on{Unzip}^{[k,n]}(X))_\ast\] 
    between the iterated cofiber cube along the first $k-1$ coordinates to the categorical cube associated with the smooth $\langle n-k+2\rangle$-manifold $\on{Unzip}^{[k,n]}(X)$.\\
    
    In the following, we construct $\Cons(\on{Unzip}^{[k+1, n]}(X))_\ast$ and show that it satisfies Assumption $(k+1)$.
    Unraveling the definitions, we find a canonical equivalence of categorical $(n-k+1)$-cubes  
    \[
    \Loc(\on{Unzip}^{[k,n]}(X))_{\partial_k^0 \ast} \simeq \on{Loc}(\on{T}(\widetilde{X^k}))_{\partial_1^0 \ast}.
    \]
    Hence the $(n-k+2)$-cube associated to $\on{T}(\widetilde{X^k})$ can be presented as the morphism of $(n-k+1)$-cubes
    \[
    \Loc(\on{Unzip}^{[k,n]}(X))_{\partial_k^0 \ast} \to \Loc(\on{T}(\widetilde{X^k}))_{\partial_1^1 \ast}.
    \]
    Using the induction hypothesis, this gives to a morphism of categorical $n$-cubes
    \[
    \cons(\on{Unzip}^{[k, n]}(X))_{\partial_k^0 \ast} \to \Loc^0(\on{T}(\widetilde{X^k}))_\ast,
    \]
    where $\Loc^0(\on{T}(\widetilde{X^k}))_\ast$ is defined by replacing the tip of $0 \in \on{Fun}(I^{k-1}, \on{Cube}_{n-k+1}(\on{LinCat}_k^{\on{dual}}))$ with the $(n-k+1)$-cube $\Loc(\on{T}(\widetilde{X^k}))_{\partial_1^1 \ast}$.

    We then define $\Cons(\on{Unzip}^{[k+1, n]}(X))_\ast$ as
\[
\big[ \cons(\on{Unzip}^{[k, n]}(X))_{\partial_k^0  \ast} \to \laxpush{\Loc^0(\on{T}(\widetilde{X^k}))_\ast}{\cons(\on{Unzip}^{[k, n]}(X))_{\partial_k^0  \ast}}{ \cons(\on{Unzip}^{[k, n]}(X))_{\partial_k^1  \ast}}\big].
    \]
\end{construction}

	In order to complete \Cref{construction:cube}, we are left to prove the following

\begin{lemma}\label{lm:construction_cube}
    The categorical $(n+1)$-cube $\Cons(\on{Unzip}^{[k+1, n]}(X))_\ast$ satisfies Assumption $(k+1)$.
\end{lemma}

\begin{proof}
    We first compute separately the iterated cofibers in the first $k-1$ directions of the restrictions of $\Cons(\on{Unzip}^{[k+1, n]}(X))_\ast$ to $i_k=0$ and $i_k = 1$.
    By construction, we have a canonical equivalence
    \[
    \Cons(\on{Unzip}^{[k+1, n]}(X))_{\partial_k^0  \ast} \simeq \Cons(\on{Unzip}^{[k, n]}(X)))_{\partial_k^0 \ast}.
    \]
    Hence by Assumption $(k)$ we have that $\on{cof}_1^{k-1} \Cons(\on{Unzip}^{[k+1, n]}(X))_{\partial_k^0 \ast} \simeq \Loc(\on{Unzip}^{[k, n]}(X)))_{\partial_k^0 \ast}$.
    Using Assumption $(k)$ again and the fact that lax colimits commutes with colimits (\Cref{lem:lax_gluing_commutes_with_colimits}), we find equivalences
\begin{align*}
        \on{cof}^{k-1}_1 & (\Cons(\on{Unzip}^{[k+1, n]}(X))_{\partial_k^1 \ast}) \\ 
        & \simeq \on{cof}^{k-1}_1 (\laxpush{\Loc^0(\on{T}(\widetilde{X^k}))_\ast}{\cons(\on{Unzip}^{[k, n]}(X))_{\partial_k^0 \ast}}{\cons(\on{Unzip}^{[k, n]}(X))_{\partial_k^1 \ast})} \\
        & \simeq \big(\laxpush{\on{cof}^{k-1}_1(\Loc^0(\on{T}(\widetilde{X^k}))_\ast)}{\on{cof}^{k-1}_1(\cons(\on{Unzip}^{[k, n]}(X))_{\partial_k^0 \ast})} {\on{cof}^{k-1}_1(\cons(\on{Unzip}^{[k, n]}(X))_{\partial_k^1 \ast})}\big)_{\partial_k^1 \ast} \\
        & \simeq \big(\laxpush{\Loc(\on{T}(\widetilde{X^k}))_{\partial_1^1 \ast}} { \Loc(\on{Unzip}^{[k, n]}(X))_{\partial_k^0 \ast}} { \Loc(\on{Unzip}^{[k, n]}(X))_{\partial_k^1 \ast}}\big)_{\partial_k^1 \ast} .
\end{align*}
  Hence we get
\begin{align*}
\on{cof}_1^{\times k} & (\cons(\on{Unzip}^{[k+1, n]}(X))_\ast)  \\
& \simeq \on{cof}_1(\big[ \Loc(\on{Unzip}^{[k, n]}(X))_{\partial_k^0 \ast} \to \big(\laxpush{\Loc(\on{T}(\widetilde{X^k}))_{\partial_1^1 \ast}} { \Loc(\on{Unzip}^{[k, n]}(X))_{\partial_k^0 \ast}} { \Loc(\on{Unzip}^{[k, n]}(X))_{\partial_k^1 \ast}}\big)_{\partial_k^1 \ast}\big]) \\
& \simeq \Loc(\on{T}(\widetilde{X^k}))_{\partial_1^1 \ast} \bigsqcup_{\Loc(\on{Unzip}^{[k, n]}(X))_{\partial_k^0 \ast}} \Loc(\on{Unzip}^{[k, n]}(X))_{\partial_k^1 \ast} \\
& \simeq \Loc(\on{Unzip}^{[k+1,n]}(X))_\ast, 
\end{align*}
where the last equivalence follows from \Cref{rem:lax_gluing_tub_neighb}.
\end{proof}

\begin{definition}\label{def:Cons_X}
	Let $(X,P)$ be a canonically smooth stratified space with boundary, such that $X$ is a compact topological manifold with boundary. 
	Let $n$ be the depth of the stratification.
	We define the categorical $(n+1)$-cube $\on{Cons}_P(X)$ of $(X,P)$ as the output
\[
\Cons_P(X)_\ast\coloneqq \Cons(\on{Unzip}^{[n+1, n]}(X))_\ast\in \on{Cube}_{n+1}(\on{LinCat}_k^{\on{dual}})\,,
\]
of \Cref{construction:cube}.
\end{definition}

\begin{remark}
    To illustrate \Cref{construction:cube}, we spell out explicitly the construction of $\Cons(\on{Unzip}^{[2,n]}(X))_\ast$.
	The $n$-cube $\on{Loc}(\on{T}(\widetilde{X^2}))$ associated to the manifold with $\langle n\rangle$-corners $\on{T}(\widetilde{X^2})$ fits in the diagram of categorical $(n-1)$-cubes
\begin{equation}\label{cons_square}
\begin{tikzcd}[sep=small]
	{\Loc(\on{Unzip}^{[1, n]}(X))_{\partial_1^0, \partial_2^0 \ast}} && 0_\ast \\
	\\
	{\Cons(\on{Unzip}^{[2,n]}(X)})_{\partial_1^1, \partial2^0 \ast} && {\Loc(\on{Unzip}^{[2, n]}(X))_{\partial_2 \ast}} \\
	&&& {\Loc(\on{T}(\widetilde{X^2}))_{\partial_1^1 \ast}}
	\arrow[from=1-1, to=1-3]
	\arrow["\psi"', from=1-1, to=3-1]
	\arrow[from=1-3, to=3-3]
	\arrow[from=3-1, to=3-3]
	\arrow["\varphi", from=3-3, to=4-4]
\end{tikzcd}
\end{equation}
where it is presented as the diagonal map $\varphi$.
	Let $\Loc^0(\on{T}(\widetilde{X^2}))_\ast \coloneqq  [0_\ast \to \Loc(\on{T}(\widetilde{X^2}))_{\partial_1^1 \ast}]$. 
	By the definition of the lax pushout of cubes, the restriction of $\on{Cons}(\on{Unzip}^{[2,n]}(X))$ to $i_1=0$ is exactly $\Loc({\on{Unzip}}^{[1, n]}(X))_{\partial_1^0 \ast}$.
	Hence by \eqref{cons_square}, we have a morphism of $n$-cubes
\[
\Cons(\on{Unzip}^{[2,n]}(X))_{\partial_2^0 \ast} \to \Loc^0(\on{T}(\widetilde{X_2}))_\ast.
\]
	We define $\Cons(\on{Unzip}^{[3,n]}(X))$ by
\[
\big[\Loc^0(\on{T}(\widetilde{X^2}))_\ast \bigsqcup^\rightarrow_{\Cons(\on{Unzip}^{[2,n]}(X))_{\partial_2^0 \ast}} \Cons(\on{Unzip}^{[2,n]}(X))_{\partial_2^1 \ast} \big] .
\]
\Cref{lm:construction_cube} shows that we have a colimit diagram of $(n-3)$-cubes as follows:
\[\begin{tikzcd}[sep=small]
	{\Cons(\on{Unzip}^{[3, n]}(X))_{\partial_1^0, \partial_2^0 \ast}} && {\Cons (\on{Unzip}^{[3, n]}(X))_{\partial_1^0, \partial_2^1 \ast}} \\
	& 0 && 0 \\
	{\Cons(\on{Unzip}^{[3,n]}(X))_{\partial_1^1, \partial_2^0 \ast}} && {\Cons(\on{Unzip}^{[3,n]}(X))_{\partial_1^1, \partial_2^1 \ast}} \\
	& 0 && {\Loc(\on{Unzip}^{[3, n]}(X))_\ast}
	\arrow[from=1-1, to=1-3]
	\arrow[from=1-1, to=2-2]
	\arrow[from=1-1, to=3-1]
	\arrow[from=1-3, to=2-4]
	\arrow[from=1-3, to=3-3]
	\arrow[from=2-2, to=2-4]
	\arrow[from=2-2, to=4-2]
	\arrow[from=2-4, to=4-4]
	\arrow[from=3-1, to=3-3]
	\arrow[from=3-1, to=4-2]
	\arrow[from=3-3, to=4-4]
	\arrow[from=4-2, to=4-4]
\end{tikzcd}\]
\end{remark}

\begin{remark}\label{rem:tip_of_cons_cube}
Repeatedly applying \Cref{cor:lax_pushout_Unzip}, one can explicitly describe the entries of the categorical cubes 
\[
\cons(\on{Unzip}^{[k, n]}(X))_\ast \colon I^{n+1} \to \on{St}_k
\]
as follows:
\[
(\Cons(\on{Unzip}^{[k, n]}(X)))_{\bf i} = 
\begin{cases} 
\on{Cons}( (p_{1,n})^{-1}(\partial X) \cap \bigcap_{i_j=0, j\not=n} \partial_{i_j}^{{\on{min}\{\{i_j=0\},k\}}} X), & \text{if $ \ i_n =0$}\\
\on{Cons}( \on{Unzip}^{[\on{min}\{\{i_j=0\},k\},n]}(X) \cap \bigcap_{i_j=0, j\not=n} \partial_{i_j}^{\on{min}\{\{i_j=0\},k\}} X ) & \text{if $ \ i_n =1$}
\end{cases}
\]
	All the categories of sheaves in this cube are constructible with respect to the stratification induced by the stratification on $X$. Hence the notation $\cons(\on{Unzip}^{[k, n]}(X))_\ast$ is justified by the fact that the tip of the cube, that is the element corresponding to $\mathbf{i} = (1,1, \ldots ,1)$, is identified with constructible sheaves on $\on{Unzip}^{[k, n]}(X)$ for the stratification induced by $(X,P)$.
	In particular, the tip of $\cons(X)_\ast$ is identified with $\cons_P(X)$.
\end{remark}

\subsection{Calabi--Yau cubes associated with smooth manifolds with corners}\label{subsec:manifolds_with_corners}

	Denote by $C_*(\mhyphen;R)\colon \mathcal{S}\to \on{Mod}_k$ the $k$-valued homology functor, which is the essentially unique colimits preserving functor determined by $C_*(\ast)\simeq k$.
	Given $X\in \mathcal{S}$, we have $X\simeq \on{colim}_X \ast \in \mathcal{S}$. Applying the functor $\on{Loc}(\mhyphen)$ and the negative cyclic homology functor $\on{HH}^{S^1}(\mhyphen)$, this equivalence induces a morphism in $\on{Mod}_k$
\[
C_*(X)\simeq \on{colim}_X k\simeq \on{colim}_X \on{HH}^{S^1}(\on{Loc}(\ast))\to \on{HH}^{S^1}(\on{Loc}(X))\,.
\]
	The morphism is functorial in $X$. Given a manifold $X$ with boundary $\partial X$, we thus obtain a morphism in $\on{Mod}_k$
\begin{equation}\label{eq:sing_homology_to_cyclic_homology}
C_*(X,\partial X)\coloneqq \on{cof}( C_*(\partial X)\to C_*(X))\longrightarrow  \on{HH}^{S^1}(\on{Loc}(X),\on{Loc}(\partial X))\,.
\end{equation}

	The following is an $\infty$-categorical version of \cite[Thm.~5.7]{BD19} for smooth manifolds with boundary.

\begin{theorem}\label{thm:rel_CY_structure_local_systems_with_boundary}
	Let $X$ be a compact smooth orientable manifold of dimension $n$ with boundary $\partial X$. The composite of any orientation class $k[n]\to C_\ast(X,\partial X)$ with the morphism \eqref{eq:sing_homology_to_cyclic_homology} defines a relative left $n$-Calabi--Yau structure on the functor 
\[
\on{Loc}(\partial X)\longrightarrow \on{Loc}(X)
\]
which restricts to a left $(n-1)$-Calabi--Yau structure on $\on{Loc}(\partial X)$.
\end{theorem}

	We prove \Cref{thm:rel_CY_structure_local_systems_with_boundary} at the end of this section.
	We next record the following refinement of \Cref{thm:rel_CY_structure_local_systems_with_boundary}, which is also a variant of \cite[Thm.~7.4.2]{CDW23}.

\begin{corollary}\label{prop:CY_cube_manifold_w_corners}
	Let $X$ be a compact smooth $\langle n\rangle$-manifold of dimension $d$.
	An orientation of the topological manifold with boundary underlying $X$ induces a cubical left $d$-Calabi--Yau structure on the categorical $n$-cube 
\[
\loc(X)_\ast \coloneqq \loc(\underline{X}) \colon I^n\to \on{LinCat}_k\,.
\]
\end{corollary}

\begin{proof}
	The proof is analogous to the proof of \cite[Thm.~7.4.2]{CDW23}: 

	We have $\on{colim}_{{\bf i}\in I^n\backslash \{1,\dots,1\}} C_\ast(\underline{X}({\bf i})) \simeq C_\ast(\partial X)$, since the functor $C_\ast(\mhyphen)\colon \mathcal{S}\to \on{Mod}_k$ preserves colimits.
	There is thus an induced morphism in $\on{Mod}_k$
\[ C_\ast(X,\partial X) \coloneqq \on{cof}(C_\ast(\partial X) \to C_\ast(X))\to \on{HH}^{S^1,\on{tot}}(\loc(X)_\ast)\,.\]
	For each face $\partial_i X$ of $X$, there is a corresponding commutative square:
\[
\begin{tikzcd}
{C_\ast(X,\partial X)} \arrow[r] \arrow[d]                   & {\on{HH}^{S^1,\on{tot}}(\loc(X)_\ast)} \arrow[d]                             &                                                        \\
{C_{\ast}(\partial_i X,\partial \partial_i X)[1]} \arrow[r] & {\on{HH}^{S^1,\on{tot}}(\loc(X)_{\partial_i^0 \ast})[1]} \arrow[r, "\simeq"] & {\on{HH}^{S^1,\on{tot}}(\loc(\partial_i X)_\ast)[1]}
\end{tikzcd}
\]
	The total negative cyclic homology class of $\loc(X)_\ast$ induced by an orientation of $(X,\partial X)$ restricts to the total negative cyclic homology class of the $i$-th face $\loc(\partial_i X)_\ast$ induced by the restriction of the orientation to $(\partial_i X,\partial \partial_i X)$. 

	The negative cyclic homology class of the functor 
\[ \loc(\partial X)\simeq \on{colim}_{I^m\backslash \{1,\dots,1\}}\loc(X)_\ast \longrightarrow \loc(X)\]
induced by an orientation of $(X,\partial X)$ describes a relative left $d$-Calabi--Yau structure by \Cref{thm:rel_CY_structure_local_systems_with_boundary}.
	Using that the restricted classes arise from the restricted orientation of the faces of $X$, we can iterating this argument for all faces, showing the recursive condition (2) from the definition of a cubical left Calabi--Yau structure in \Cref{def:cubical_CY_structure}.
\end{proof}

	Let $X$ be a compact differentiable manifold. Then $X$ admits a triangulation \cite{Whi40}, meaning that $X$ is homeomorphic to a finite simplicial complex.
	Furthermore, the triangulation satisfies that the link of any vertex of the simplicial complex is homeomorphic to a sphere.

\begin{lemma}
	Let $X$ be a compact smooth manifold with boundary and a triangulation. Then there exists a conically smooth stratification with boundary on $X$, whose depth $d$ stratum is given by the interior of all $d$-cells in $X$. 
\end{lemma}

\begin{proof}
	The conically smooth stratification $X\to P$ is described in \cite[Example 3.5.6]{AFT17}.
	Refining $P$ if necessary, we may assume that for each $p\in P$, $X_p$ is a disjoint union of interiors of $d$-cells contained only in $X^\circ$ or only in $\partial X$.
	The morphism $X\to P$ then conically extends to a topologically coCartesian map $X\to P\to \mathcal{P}([1])$, specifying a conically smooth stratification with $\langle 1\rangle$-corners.
\end{proof}

\begin{lemma}\label{lem:local_systems_on_spheres_are_CY}
	For $m\geq 1$, let $D^m$ be the $m$-dimensional disc with boundary $S^{m-1}$.
	A choice of orientation class $\sigma\colon k[m]\simeq C_*(D^m,S^{m-1})$ gives rise via \eqref{eq:sing_homology_to_cyclic_homology} to a relative left $m$-Calabi--Yau structure on the functor $\on{Loc}(S^{m-1})\to \on{Loc}(D^m)$ which restricts to a left $(m-1)$-Calabi--Yau structure on $\on{Loc}(S^{m-1})$.
\end{lemma}

\begin{proof}
	We prove the assertion by an induction on $m$. The case $m=1$ follows from a direct argument and is left to the reader.
	Suppose the statement has been proved for $m-1$. Using the pushout diagram in $\on{LinCat}_k^{\on{dual}}$ 
\[
\begin{tikzcd}
\on{Loc}(S^{m-2}) \arrow[d] \arrow[r] \arrow[rd, "\ulcorner", phantom] & \on{Loc}(D^{m-1}) \arrow[d] \\
\on{Loc}(D^{m-1}) \arrow[r]                                              & \on{Loc}(S^{m-1})          
\end{tikzcd}
\]
and the gluing result for the left Calabi--Yau structures from \cite[Thm.~3.15]{Chr23}, we find that $\on{Loc}(S^{m-1})$ inherits a relative left $(m-1)$-Calabi--Yau structure arising from the restriction of $\sigma$. 

	The class $\sigma$ gives via \eqref{eq:sing_homology_to_cyclic_homology} rise to a diagram of $k$-linear endofunctors of $\on{Loc}(D^m)\simeq \on{Mod}_k$ of the form
\[
\begin{tikzcd}
\on{id}_{\on{Mod}_k} \arrow[d] \arrow[r, hook] & {\on{id}_{\on{Mod}_k}\oplus [1-m]} \arrow[d, "\simeq"] \arrow[r] & \on{cof} \arrow[d] \\
\on{fib} \arrow[r]                             & {\on{id}_{\on{Mod}_k}\oplus [1-m]} \arrow[r, two heads]          & {[1-m]}           
\end{tikzcd}
\]
with horizontal fiber and cofiber sequences, where $\on{id}_{\on{Mod}_k}\simeq \on{id}_{\on{Mod}_k}^!$ and the middle morphism arises from the equivalence $\on{id}_{\on{Loc}(S^{m-1})}^!\simeq \on{id}_{\on{Loc}(S^{m-1})}[1-m]$.
	The right lower horizontal morphism is given by a counit morphism of the adjunction $f_!\colon \on{Loc}(S^{m-1})\to \on{Loc}(\ast)\noloc f^*$, with $f\colon S^{m-1}\to \ast$, and thus by \cite[Lem.~3.3]{Chr20} equivalent to a projection from the direct sum.
	Similarly, the upper left horizontal morphism is by \cite[Lem.~2.34]{Chr23} and \cite[Lem.~3.3]{Chr20} equivalent to an inclusion of the direct summand.
	We choose the direct sum decompositions of $\on{id}_{\on{Mod}_{\mathbf{S}}}\oplus [1-m]$ such that the inclusion and projection are the canonical ones.
	Since $k$ is connective (it is even concentrated in degree $0$) and $m\geq 2$, there are no non-zero natural transformations $\on{id}_{\on{Mor}_k}\to [1-m]$. The middle vertical autoequivalence of $\on{id}_{\on{Mod}_k}\oplus [1-m]$ must thus be lower triangular with autoequivalences on the diagonal.
	The commutativity of the diagram shows that the lower off-diagonal term of the equivalence must vanish as well.
	From this, we conclude that the vertical morphisms are equivalences. 
\end{proof}

\begin{lemma}\label{lem:unzip_of_Delta_n_is_CY}
	Consider $\Delta^n$ as a conically smooth stratified space and let $\on{Unzip}^{[1,n]}(\Delta^n)$ be its unzipping, considered as a compact smooth $\langle n\rangle$-manifold.
	Then the categorical $n$-cube $\on{Loc}(\on{Unzip}^{[1,n]}(\Delta^n))_\ast$ admits a canonical cubical left $n$-Calabi--Yau structure. 
\end{lemma}

\begin{proof}
	We find that $\partial \on{Unzip}^{[1,n]}(\Delta^n)\simeq S^{n-1}$ and we thus obtain a negative cyclic homology class $\sigma\colon k[n]\simeq C_*(X,\partial X)\to \on{HH}^{S^1,\on{tot}}(\on{Loc}(\on{Unzip}^{[1,n]}(\Delta^n))_\ast)$ from the orientation of $\on{Unzip}(\Delta^n)$.
	For $1\leq i\leq n$, the back face $\on{Loc}(\on{Unzip}^{[1,n]}(\Delta^n))_{\delta^0_i \ast}$ is canonically equivalent to $\on{Loc}(\on{Unzip}^{[1,n]}(\Delta^{n-1}))_\ast$ and the total negative cyclic homology class of $\on{Loc}(\on{Unzip}^{[1,n]}(\Delta^n))_\ast$ restricts to the class arising from the orientation of $\on{Unzip}^{[1,n]}(\Delta^{n-1})$.
	The fact that $\sigma$ induces a cubical left $n$-Calabi--Yau structure thus reduces to an iterated application of \Cref{lem:local_systems_on_spheres_are_CY}.
\end{proof}

\begin{proof}[Proof of \Cref{thm:rel_CY_structure_local_systems_with_boundary}]
	We prove this by an induction on $n$, thus assuming that the statement has been shown for all compact orientable smooth manifolds of dimension $n-1$.
	Consider the categorical cubes $\Loc(\on{Unzip}^{[k,n]}(X))_\ast$ from \Cref{construction:cube}.
	In the case $k=1$, $\on{Unzip}^{[1,n]}(X)$ is given by the disjoint union of the unzippings of the n-cells $\Delta_n\subset X$ and $\Loc(\on{Unzip}^{[1,n]}(X))_\ast$ thus inherits a cubical $n$-Calabi--Yau structure form the orientation of $X$ by \Cref{lem:unzip_of_Delta_n_is_CY}. For $k\geq 1$, we have
\[ 
\Loc(\on{Unzip}^{[k+1,n]}(X))_\ast\simeq \Loc(\on{T}(\widetilde{X^k}))_{\partial_1^1 \ast} \bigsqcup_{\Loc(\on{Unzip}^{[k, n]}(X))_{\partial_k^0 \ast}} \Loc(\on{Unzip}^{[k, n]}(X))_{\partial_k^1 \ast}\,.
\]
	We iteratively produce cubical left $n$-Calabi--Yau structures on the cubes $\Loc(\on{Unzip}^{[k+1,n]}(X))_\ast$ by gluing.
	The boundary of $\on{T}(\widetilde{X^k})$ is homeomorphism to an $(n-1)$-sphere.
	Using \Cref{lem:local_systems_on_spheres_are_CY} and applying the induction assumption to the faces of $\Loc(\on{Unzip}^{[k+1,n]}(X))_\ast$, we thus find that the orientation of $X$ induces a cubical left $n$-Calabi--Yau structure on $\Loc(\on{T}(\widetilde{X^k}))_{\ast}$, compatible with the cubical Calabi--Yau stucture of $\Loc(\on{Unzip}^{[k,n]}(X))_\ast$ on their common face.
	By \Cref{prop:cofiber_and_pushout_CY_cubes}, $\Loc(\on{Unzip}^{[k+1,n]}(X))_\ast$ inherits a cubical left $n$-Calabi--Yau structure.

	In the case $k=n$, the $1$-cube $\Loc(\on{Unzip}^{[n+1,n]}(X))_\ast$ amounts to the functor $\on{Loc}(\partial X)\to \on{Loc}(X)$, yielding the desired Calabi--Yau structure.
\end{proof}

\begin{remark}
	The proof of \Cref{lem:local_systems_on_spheres_are_CY} directly translates from the $k$-linear setting to the base being any connective $\mathbb{E}_\infty$-ring spectrum.
	From this, one can deduce \Cref{lem:local_systems_on_spheres_are_CY} with coefficients in any $\mathbb{E}_\infty$-ring spectrum $R$. The proof of \Cref{thm:rel_CY_structure_local_systems_with_boundary} then also translates to the $R$-linear setting, assuming $X$ is endowed with an orientation with respect to $R$. 
\end{remark}

\subsection{The cubical Calabi--Yau structure on \texorpdfstring{$\Cons(X)_\ast$}{Cons(X)}}\label{section:main_theorem}

	As above, we fix a conically smooth stratified space with boundary $X\to P\to \mathcal{P}([1])$, such that the topological space underlying $X$ is a compact manifold with boundary of dimension $d$.
	In this section, we show that if $X$ is orientable, an orientation of $X$ induces a cubical left $d$-Calabi--Yau structure on the categorical cube $\Cons(X)_\ast$ from \Cref{def:Cons_X}. 

\begin{proposition}\label{prop:final_result_k}
	Let $(X,P)$ be a conically smooth stratified space with boundary of depth $n$ and dimension $d$.
	Assume that the underlying topological space $X$ is an orientable manifold with boundary.
	For every $1 \leq k \leq n+1$, a choice of orientation on $X$ induces a cubical left $d$-Calabi--Yau structure on the categorical $(n+1)$-cube $\cons(\on{Unzip}^{[k, n]}(X))_\ast$ of \Cref{construction:cube}.
\end{proposition}

\begin{proof}
	We prove by induction on $1 \leq k \leq n+1$ the following statement: a choice of orientation on $X$ induces a cubical left $n$-Calabi--Yau structure on the categorical $(n+1)$-cube $\cons(\on{Unzip}^{[k, n]}(X))_\ast$ of \Cref{construction:cube}, and the induced Calabi--Yau structure on 
	\[
	\on{cof}_1^{\times k-1}\Cons(\on{Unzip}^{[k,n]}(X))_{\ast}\simeq \Loc(\on{Unzip}^{[k,n]}(X))_\ast
	\]
	coincides with the one arising from \Cref{prop:CY_cube_manifold_w_corners} by restricting the orientation of $X$ to $\on{Unzip}^{[k,n]}(X)$.
	For $k=1$, the first part of the statement is \Cref{prop:CY_cube_manifold_w_corners}, while the second part is trivial.
	Suppose we have proved the claim for some $k < n+1$.
	By \Cref{prop:laxgluingCYcubes} and by the inductive definition of $\Cons(\on{Unzip}^{[k+1,n]}(X))_\ast$, it is enough to prove that the $(n+1)$-cube 
    \[
    \on{C}^k_\ast \coloneqq [\cons(\on{Unzip}^{[k, n]}(X))_{\partial_k^0 \ast} \to \loc^0(\on{T}(\widetilde{X^k}))_\ast]
    \]
    carries a cubical left $n$-Calabi--Yau structure compatible with the one on $\cons(\on{Unzip}^{[k, n]}(X))$ given by the induction hypothesis.
    
    \begin{itemize}
    \item[\textbf{Claim}:] The categorical $(n+1)$-cube 
\[
\on{C}^k_\ast \coloneqq [\cons(\on{Unzip}^{[k, n]}(X))_{\partial_k^0 \ast} \to \loc^0(\on{T}(\widetilde{X^k}))_\ast]
\]
    carries a cubical left $d$-Calabi--Yau structure whose restriction to the $\partial_k^0$-face is compatible with the restriction of the cubical left $d$-Calabi--Yau structure of $\cons(\on{Unzip}^{[k, n]}(X))_\ast$ to the $\partial_k^0$-face.
    \end{itemize}
    
    We suppose for the moment that the claim holds and we show how this allows us to finish the proof of the proposition.
    By \Cref{prop:laxgluingCYcubes} and by the inductive definition of $\Cons(\on{Unzip}^{[k+1,n]}(X))_\ast$, the claim implies that $\Cons(\on{Unzip}^{[k+1,n]}(X))_\ast$ has a cubical left $d$-Calabi--Yau structure induced by the ones of $\on{C}^k_\ast$ and $\cons(\on{Unzip}^{[k, n]}(X))_\ast$.
    To complete our induction, we are left to show that the induced cubical Calabi--Yau structure on 
	\[
	\on{cof}_1^{\times k}\Cons(\on{Unzip}^{[k+1,n]}(X))_{\partial_k^0\ast}\simeq \loc(\on{Unzip}^{[k+1,n]}(X))_\ast
	\]
	coincide with the one given by \Cref{prop:CY_cube_manifold_w_corners}.
	The proof of \Cref{lm:construction_cube} shows that we have equivalences
    \begin{align*}
        & \on{cof}_1^{\times k} (\cons(\on{Unzip}^{[k+1, n]}(X))_\ast) \\
        & \simeq \on{cof}_1 \big[\on{cof}_{1}^{k-1}(\cons(\on{Unzip}^{[k, n]}(X))_{\partial_k^0 \ast}) \to \big(\loc(\on{T}(\widetilde{X^k}))_{\partial_1^1 \ast} \! \! \! \! \! \! \! \! \!  \! \! \! \! \! \! \! \! \! \! \! \! \!\bigsqcup^{\rightarrow}_{\on{cof}_{1}^{k-1}(\cons(\on{Unzip}^{[k, n]}(X))_{\partial_k^0 \ast})} \! \! \! \! \! \! \! \! \!\! \! \! \! \! \! \! \! \! \! \! \! \!\on{cof}_{1}^{k-1}(\cons(\on{Unzip}^{[k, n]}(X))_{\partial_k^1 \ast})\big)_{\partial_k^1 \ast} \big]  \\
        & \simeq \loc(\on{Unzip}^{[k+1, n]}(X))_\ast.
    \end{align*}
    The second part of our inductive hypothesis then shows that the cubical $d$-Calabi--Yau structure on $\cons(\on{Unzip}^{[k+1, n]}(X))_\ast$ induces a cubical $d$-Calabi--Yau structure on $\on{cof}_1^{\times k} (\cons(\on{Unzip}^{[k+1, n]}(X))_\ast) \simeq  \loc(\on{Unzip}^{[k+1, n]}(X))_\ast$ which arises from the restriction of the orientation of $X$.
    
    We now proceed with the proof of the claim. 
    We begin by constructing a class $\widetilde{\alpha} \in \on{HH}^{S^1,\on{tot}}(\on{C}_\ast^k)$ which enjoys nice properties, allowing us to prove its non-degeneracy.
    \begin{itemize}
    \item \textbf{Construction of $\widetilde{\alpha} \in \on{HH}^{S^1, \on{tot}}(\on{C}_\ast^k)$}
    \end{itemize}
    By definition of $\on{C}^k_\ast$, we have $\on{C}^k_{\partial_k^0} \simeq \Cons(\on{Unzip}^{[k,n]}(X))_{\partial_k^0\ast}$.
	Hence the cubical left $d$-Calabi--Yau class $\alpha$ on $\Cons(\on{Unzip^{[k,n]}}(X))_\ast$ induces a cubical left $(d-1)$-Calabi--Yau class $\alpha_{\partial_k^0}$ on $\on{C}^k_{\partial_k^0 \ast}$.
	By construction, we have an equivalence $\on{cof}_1^{\times (k-1)} C_\ast^k \simeq \Loc(\on{T}(\widetilde{X^k}))_\ast$, hence inducing a morphism between horizontal fiber and cofiber sequences
\[\begin{tikzcd}[sep=small]
	{\on{HH}^{S^1, \on{tot}}(C^k_{\ast})} && {\on{HH}^{S^1, \on{tot}}(C^k_{\partial_k^0 \ast})[1]} && {\on{HH}^{S^1, \on{tot}}(C^k_{\partial_k^1 \ast})[1]} \\
	\\
	{\on{HH}^{S^1, \on{tot}}(\Loc(\on{T}(\widetilde{X^k}))_\ast)} && {\on{HH}^{S^1, \on{tot}}(\Loc(\on{T}(\widetilde{X^k}))_{\partial_k^0 \ast})[1]} && {\on{HH}^{S^1, \on{tot}}(\Loc(\on{T}(\widetilde{X^k}))_{\partial_k^1 \ast})[1]}
	\arrow["\eta", from=1-1, to=1-3]
	\arrow["{\widetilde{q}}"', from=1-1, to=3-1]
	\arrow["\psi", from=1-3, to=1-5]
	\arrow["q"', from=1-3, to=3-3]
	\arrow["q'", from=1-5, to=3-5]
	\arrow["\delta"', from=3-1, to=3-3]
	\arrow["\varphi"', from=3-3, to=3-5]
\end{tikzcd}\]

	We want to show that $-\alpha_{\partial_k^0}$ admits a lift $\widetilde{\alpha}$ along $\eta$ such that $\widetilde{q}(\widetilde{\alpha})$ agrees with the (image of the) orientation class of $\on{T}(\widetilde{X^k})$ (see \Cref{prop:CY_cube_manifold_w_corners}).
	
	Notice that since $C^k_{\partial_k^1 \ast} \simeq \Loc^0(\on{T}(\widetilde{X^k}))_\ast$, the morphism $q'$ is an equivalence.
	Hence to construct a lift of $-\alpha$ it is enough to show that $\varphi \circ q (-\alpha_{\partial_k^0 \ast})$ is homotopic to zero.
	
	By construction, the morphism $q$ agrees with the restriction to $\partial_k^0$ of the morphism induced on total Hochschild homology by
\[
\on{cof}_1^{\times (k-1)} \Cons(\on{Unzip}^{[k,n]}(X)) \simeq \Loc(\on{Unzip}^{[k,n]}(X)) \,.
\]
	Hence, by our inductive hypothesis, the class $q(\alpha_{\partial_k^0})$ agrees with the (image of the) orientation class of $T(\widetilde{X^k})$ induced by the orientation class of $\on{Unzip}^{[k,n]}(X)$.
	Notice that the (image of the) orientation class of class of $\on{T}(\widetilde{X^k})$ induces via $\delta$ the opposite class of $q(\alpha_{\partial_k^0})$.
	Hence the orientation class of $\on{T}(\widetilde{X^k})$ induces an homotopy to zero $\varphi \circ q (-\alpha_{\partial_k^0 \ast}) \simeq 0$.
	This defines a lift $\widetilde{\alpha}$ of $-\alpha_{\partial_k^0 \ast}$ and, by construction, we have that $\widetilde{q}(\widetilde{\alpha})$ agrees with the (image of the) orientation class of $\on{T}(\widetilde{X^k})$.
	
	\begin{itemize}
	\item \textbf{Proof of the claim}
	\end{itemize}
	Consider the class $\widetilde{\alpha} \colon k[d] \to \on{HH}^{S^1, \on{tot}}(C^k_{\ast})$ constructed in the previous paragraph.
	By construction, the class $\widetilde{\alpha}$ is compatible with the class $\alpha$ describing the cubical left $d$-Calabi--Yau on $\Cons(\on{Unzip}^{[k,n]}(X))_\ast$, meaning that it restricts to $-\alpha_{\partial_k^0}$.
	We will complete the proof by showing that $\widetilde{\alpha}$ satisfies the assumptions of \Cref{prop:equivalent_characterization_CY_cube}.
	Since we have
	\[
    \on{C}^k_{ \partial_j \ast} \simeq [\Cons(\on{Unzip}^{[k, n]}(X))_{\partial_k^0 \partial^0_j \ast} \to 0_\ast]\,
    \]
    the compatibility with $\alpha$ also shows that $\widetilde{\alpha}$ satisfies \Cref{prop:equivalent_characterization_CY_cube}-(1).
    Since $\widetilde{q}(\widetilde{\alpha})$ agrees with the (image of the) orientation class of $\on{T}(\widetilde{X^k})$, we also have that $\widetilde{\alpha}$ satisfies \Cref{prop:equivalent_characterization_CY_cube}-(1), hence completing the proof.
\end{proof}

\begin{remark}
	\Cref{construction:cube} can be generalized, associating a categorical $(n+m)$-cube with a conically smooth stratification with $\langle m\rangle$-corners of depth $n$ of a manifold with boundary.
	\Cref{thm:final_result} can also be generalized to this setting.
\end{remark}

\begin{theorem}\label{thm:final_result}
	Let $(X,P)$ be an conically smooth stratified space with boundary of depth $n$ and dimension $d$.
	Assume that the underlying topological space $X$ is an orientable manifold with boundary.
	A choice of orientation on $X$ induces an $d$-Calabi--Yau structure on the categorical $(n+1)$-cube $\cons_P(X)_\ast$ of \Cref{def:Cons_X}.
\end{theorem}

\begin{proof}
    This is the case $k=n+1$ in \Cref{prop:final_result_k}.
\end{proof}

\begin{corollary}\label{cor:relative_Calabi_Yau}
    Let $(X,P)$ be conically smooth stratification with boundary of depth $d$ and dimension $n$.
	Assume that the underlying topological space $X$ is an orientable manifold with boundary.
	A choice of orientation on $X$ induces a relative left $n$-Calabi--Yau structure on the functor
\begin{equation}\label{eq:constr_rel_CY_functor}
\on{colim}_{I^{n+1} \setminus \{(1, \ldots 1)\}}\cons_P(X)_\ast \to \cons_P(X)\,.
\end{equation}
\end{corollary}

\begin{proof}
    This follows directly by combining \Cref{thm:final_result} with \Cref{rem:tip_of_cons_cube}.
\end{proof}

\begin{example}
	\Cref{cor:relative_Calabi_Yau} applies to Whitney stratified spaces by \Cref{thm:Whitney_implies_conically_smooth}.
\end{example}

\begin{remark}
	It appears that \Cref{thm:final_result} cannot be naively generalized to conically smooth stratified spaces $(X,P)$ if $X$ is not a manifold.
	There appears to be, in general, no notion of orientation of $X$ suitable for this purpose.

	If $X$ is closed, the cofiber of the functor \eqref{eq:constr_rel_CY_functor} in $\on{LinCat}_k$ is equivalent to $\on{Loc}(X)$, which thus inherits an absolute left $n$-Calabi--Yau structure.
	If $X$ is not a closed oriented manifold, $\on{Loc}(X)$ does not in general admit a left Calabi--Yau structure. 
\end{remark}

\subsection{An illustrated example}\label{subsec:example}

\begin{example}
	Let $(X,P)=(\on{D}_2, 0<1<2)$ be a two-dimensional disk stratified by its axes and their intersection at the origin, as depicted in the following picture:
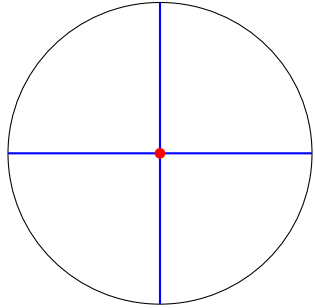
\begin{center}
\begin{tikzpicture}

\draw (0,0) circle (2);

\draw[blue, thick] (-2,0) -- (2,0);
\draw[blue, thick] (0,-2) -- (0,2);

\fill[red] (0,0) circle (2pt);

\end{tikzpicture}
\end{center}

	This is a stratified space arising as the interior (equivalently, by coarsening the stratification in the normal direction to the boundary, see \Cref{rem:corners_or_interior}) of the following conically smooth stratified space with boundary.
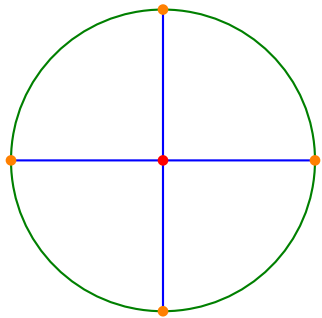
\begin{center}
\begin{tikzpicture}

\draw[ao, thick] (0,0) circle (2);

\draw[blue, thick] (-2,0) -- (2,0);
\draw[blue, thick] (0,-2) -- (0,2);

\fill[red] (0,0) circle (2pt);
\fill[orange] (0,2) circle (2pt);
\fill[orange] (0,-2) circle (2pt);
\fill[orange] (2,0) circle (2pt);
\fill[orange] (-2,0) circle (2pt);
\end{tikzpicture}
\end{center}

	The interiors of the stratified spaces $\on{Unzip}^{[2,2]}(X)$ and $\on{Unzip}^{[1,2]}(X)$ can be depicted as follows:
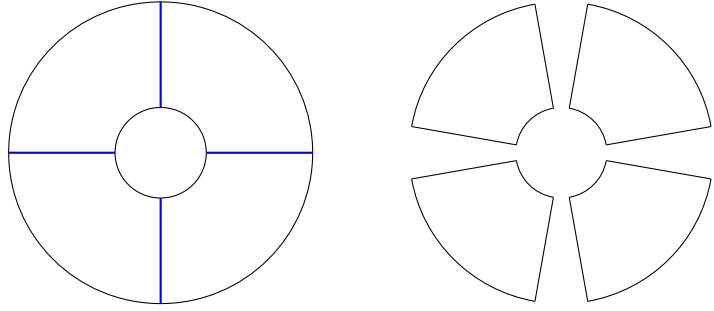
\begin{center}

\begin{tikzpicture}
\hskip -75pt
\draw (0,0) circle (2);

\draw (0,0) circle (0.6);

\draw[blue, thick] (-2,0) -- (-0.6,0);
\draw[blue, thick] (0.6,0) -- (2,0);
\draw[blue, thick] (0,-2) -- (0,-0.6);
\draw[blue, thick] (0,0.6) -- (0,2);

\hskip 150pt

\foreach \ang in {0,90,180,270}
{
    \begin{scope}[rotate=\ang]

\draw (10:0.6) -- (10:2);

\draw (10:2) arc (10:80:2);

\draw (80:2) -- (80:0.6);

\draw (80:0.6) arc (80:10:0.6);
	\end{scope}
}
\end{tikzpicture}
\end{center}
	The stratified spaces above are given by first removing a tubular neighborhood of the origin (in red), and then removing a tubular neighborhood of the blue lines.
	In this example, we have that $\widetilde{X^2}=X^2$ is the red origin, while $\widetilde{X^1}$ is given by the blue axes of $\on{Unzip}^{[2,2]}(X)$.
	
	In order to get a Calabi--Yau structure on $\Cons_P(X)$, we first consider the Calabi--Yau cube arising from the smooth manifold with corners $\on{Unzip}^{[1,2]}(X)$ via \Cref{prop:CY_cube_manifold_w_corners}.
	For this, we consider the $\langle 3\rangle$-corners that $\on{Unzip}^{[1,2]}(X)$ inherits, which can be depicted as follows:
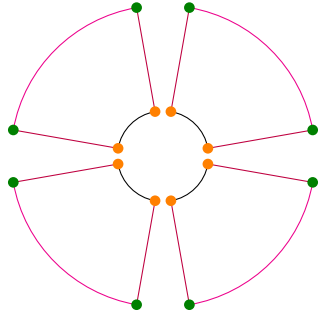
\begin{center}
\begin{tikzpicture}
	\foreach \ang in {0,90,180,270}
{
    \begin{scope}[rotate=\ang]

\draw[purple] (10:0.6) -- (10:2);

\draw[purple] (80:2) -- (80:0.6);

\draw[magenta] (10:2) arc (10:80:2);

\draw (80:0.6) arc (80:10:0.6);

\fill[orange] (10:0.6) circle (2pt);
\fill[orange] (80:0.6) circle (2pt);
\fill[ao] (10:2) circle (2pt);
\fill[ao] (80:2) circle (2pt);
	\end{scope}
}
\end{tikzpicture}
\end{center}
	Then the left $2$-Calabi--Yau categorical cube arising from \Cref{prop:CY_cube_manifold_w_corners} is given by
\begin{equation}\label{example:cube1}
\begin{tikzcd}[sep=small]
	0 && {\loc(\textcolor{ao}{\ast} \bigsqcup \textcolor{ao}{\ast})^{\times 4}} \\
	& {\loc(\textcolor{orange}{\ast} \bigsqcup \textcolor{orange}{\ast})^{\times 4}} && {\loc(\textcolor{purple}{|} \bigsqcup \textcolor{purple}{|})^{\times 4}} \\
	0 && {\loc(\textcolor{magenta}{\frown})^{\times 4}} \\
	& {\loc(\textcolor{black}{\frown})^{\times 4}} && {\loc(\on{Unzip}^{[1,2]}(X))}
	\arrow[from=1-1, to=1-3]
	\arrow[from=1-1, to=2-2]
	\arrow[from=1-1, to=3-1]
	\arrow[from=1-3, to=2-4]
	\arrow[from=1-3, to=3-3]
	\arrow[from=2-4, to=4-4]
	\arrow[from=3-1, to=3-3]
	\arrow[from=3-1, to=4-2]
	\arrow[from=3-3, to=4-4]
	\arrow[from=4-2, to=4-4]
	\arrow[crossing over, from=2-2, to=2-4]
	\arrow[crossing over, from=2-2, to=4-2]
\end{tikzcd}
\end{equation}
	Now consider a tubular neighborhood $\on{T}(\widetilde{X^1})$ of the blue axis of $\on{Unzip}^{[2,2]}(X)$, and consider it as equipped with $\langle 3\rangle$-corners as follows:
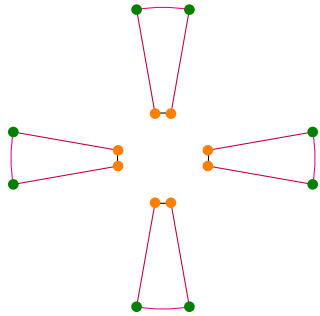
\begin{center}
\begin{tikzpicture}
	\foreach \ang in {0,90,180,270}
{
    \begin{scope}[rotate=\ang]

\draw[purple] (-10:0.6) -- (-10:2);

\draw[purple] (10:2) -- (10:0.6);

\draw[magenta] (-10:2) arc (-10:10:2);

\draw (10:0.6) arc (10:-10:0.6);

\fill[orange] (-10:0.6) circle (2pt);
\fill[orange] (10:0.6) circle (2pt);

\fill[ao] (-10:2) circle (2pt);
\fill[ao] (10:2) circle (2pt);
	\end{scope}
}
\end{tikzpicture}
\end{center}
	The left $2$-Calabi--Yau categorical cube arising from the smooth manifold with $\langle 3\rangle$-corners $\on{T}(\tilde{X}^1)$ is then given by
\begin{equation}\label{example:cube2}
\begin{tikzcd}[sep=small]
	0 && {\loc(\textcolor{ao}{\ast} \bigsqcup \textcolor{ao}{\ast})^{\times 4}} \\
	& {\loc(\textcolor{orange}{\ast} \bigsqcup \textcolor{orange}{\ast})^{\times 4}} && {\loc(\textcolor{purple}{|} \bigsqcup \textcolor{purple}{|})^{\times 4}} \\
	0 && {\loc(\textcolor{magenta}{\frown})} \\
	& {\loc(\textcolor{black}{\frown})} && {\loc(\on{T}(X^1))}
	\arrow[from=1-1, to=1-3]
	\arrow[from=1-1, to=2-2]
	\arrow[from=1-1, to=3-1]
	\arrow[from=1-3, to=2-4]
	\arrow[from=1-3, to=3-3]
	\arrow[from=2-4, to=4-4]
	\arrow[from=3-1, to=3-3]
	\arrow[from=3-1, to=4-2]
	\arrow[from=3-3, to=4-4]
	\arrow[from=4-2, to=4-4]
	\arrow[crossing over, from=2-2, to=2-4]
	\arrow[crossing over, from=2-2, to=4-2]
\end{tikzcd}
\end{equation}
	Following \Cref{construction:cube}, we lax glue the cubes \eqref{example:cube1} and \eqref{example:cube2} along their common top face, thus yielding the left $2$-Calabi--Yau categorical $3$-cube $\on{Cons}(\on{Unzip}^{[2,2]}(X))_\ast$:
\begin{equation}\label{example:cube3}
\begin{tikzcd}[sep=small]
	0 && {\loc(\textcolor{ao}{\ast} \bigsqcup \textcolor{ao}{\ast})^{\times 4}} \\
	& {\loc(\textcolor{orange}{\ast} \bigsqcup \textcolor{orange}{\ast})^{\times 4}} && {\loc(\textcolor{purple}{|} \bigsqcup \textcolor{purple}{|})^{\times 4}} \\
	0 && {\cons_{0<1}(S^1)} \\
	& {\cons_{0<1}(S^1)} && {\cons_{0<1}(\on{Unzip}^{[1,2]}(X))}
	\arrow[from=1-1, to=1-3]
	\arrow[from=1-1, to=2-2]
	\arrow[from=1-1, to=3-1]
	\arrow[from=1-3, to=2-4]
	\arrow[from=1-3, to=3-3]
	\arrow[from=2-4, to=4-4]
	\arrow[from=3-1, to=3-3]
	\arrow[from=3-1, to=4-2]
	\arrow[from=3-3, to=4-4]
	\arrow[from=4-2, to=4-4]
	\arrow[crossing over, from=2-2, to=2-4]
	\arrow[crossing over, from=2-2, to=4-2]
\end{tikzcd}
\end{equation}
Above, the stratified circle $S^1$ appearing at the back right position is the exterior circle of $\on{Unzip}^{[2,2]}(X)$, while the stratified circle $S^1$ appearing at the bottom left position is the interior circle of $\on{Unzip}^{[2,2]}(X)$.
	Notice that the circles arise since the four copies of the green and orange points are (cyclically) permuted in the gluing maps. 

The stratum $X^2$ is the origin, and its tubular neighborhood $T(X^2)$ inherits the structure of a manifold with $\langle 2\rangle$-corners. However, since $X^2$ does not intersect the boundary of $X$, the $\langle 2\rangle$-manifold structure on $\on{T}(X^2)$ arises from a $\langle 1\rangle$-manifold structure, namely the disc with boundary: 
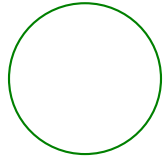
\begin{center}
\begin{tikzpicture}

\draw[ao, thick] (0,0) circle (1);

\end{tikzpicture}
\end{center}

Associated with the $\langle 2\rangle$-manifold $\on{T}(X^2)$ is thus the following $2$-Calabi--Yau square :
\[
\begin{tikzcd}
0 \arrow[d] \arrow[r]   & 0 \arrow[d]      \\
\on{Loc}(S^1) \arrow[r] & \on{Loc}(T(X^2))
\end{tikzcd}
\]
We extend this square with $0$'s to a $3$-cube and attach to the left the left face of the $3$-cube \eqref{example:cube3}.
\[
\begin{tikzcd}[sep=small]
0 \arrow[dd] \arrow[rd] \arrow[rr] &                                                                                                                                                   & 0 \arrow[rd] \arrow[dd] \arrow[rr] &                          & 0 \arrow[dd] \arrow[rd] &                       \\
                                   & \loc(\textcolor{orange}{\ast} \bigsqcup \textcolor{orange}{\ast})^{\times 4} \arrow[dd] \arrow[rr] \arrow[rrdd, "\ulcorner", phantom, bend right=10, near end] &                                    & 0 \arrow[dd] \arrow[rr]  &                         & 0 \arrow[dd]          \\
0 \arrow[rd] \arrow[rr]            &                                                                                                                                                   & 0 \arrow[rd] \arrow[rr]            &                          & 0 \arrow[rd]            &                       \\
                                   & \cons_{0<1}(S^1) \arrow[rr]                                                                                                                       &                                    & \on{Loc}(S^1) \arrow[rr] &                         & \on{Loc}(\on{T}(X^2))
\end{tikzcd}
\] 
	The front face of the left $3$-cube above is a pushout square, encoding the equivalence 
\[ \on{cof}_1 \on{Cons}(\on{Unzip}^{[1,2]}(X))_{\partial^0_1 \ast}\simeq \on{Loc}(T(X^2))_{\partial^0_1\ast}\,.\] 
	Composing the $3$-cubes, we obtain the $3$-cube $\loc^0(\on{T}(X^2))$, whose left face agrees with the left face of \eqref{example:cube3}:
\begin{equation}\label{example:cube4}
\begin{tikzcd}[sep=small]
	0 && 0 \\
	& {\loc(\textcolor{orange}{\ast} \bigsqcup \textcolor{orange}{\ast})^{\times 4}} && 0 \\
	0 && 0 \\
	& {\cons_{0<1}(S^1)} && {\on{Loc}(\on{T}(X^2))}
	\arrow[from=1-1, to=1-3]
	\arrow[from=1-1, to=2-2]
	\arrow[from=1-1, to=3-1]
	\arrow[from=1-3, to=2-4]
	\arrow[from=1-3, to=3-3]
	\arrow[from=2-4, to=4-4]
	\arrow[from=3-1, to=3-3]
	\arrow[from=3-1, to=4-2]
	\arrow[from=3-3, to=4-4]
	\arrow[from=4-2, to=4-4]
	\arrow[crossing over, from=2-2, to=2-4]
	\arrow[crossing over, from=2-2, to=4-2]
\end{tikzcd}
\end{equation}
	Finally, we lax glue the cubes \eqref{example:cube3} and \eqref{example:cube4} along their common left face, thus yielding the desired left $2$-Calabi--Yau categorical cube with tip $\cons_{0<1<2}(X)$:
\[
\begin{tikzcd}[sep=small]
	0 && {\loc(\textcolor{ao}{\ast} \bigsqcup \textcolor{ao}{\ast})^{\times 4}} \\
	& {\loc(\textcolor{orange}{\ast} \bigsqcup \textcolor{orange}{\ast})^{\times 4}} && {\loc(\textcolor{purple}{|} \bigsqcup \textcolor{purple}{|})^{\times 4}} \\
	0 && {\cons_{0<1}(S^1)} \\
	& {\cons_{0<1}(S^1)} && {\cons_{0<1<2}(X)}
	\arrow[from=1-1, to=1-3]
	\arrow[from=1-1, to=2-2]
	\arrow[from=1-1, to=3-1]
	\arrow[from=1-3, to=2-4]
	\arrow[from=1-3, to=3-3]
	\arrow[from=2-4, to=4-4]
	\arrow[from=3-1, to=3-3]
	\arrow[from=3-1, to=4-2]
	\arrow[from=3-3, to=4-4]
	\arrow[from=4-2, to=4-4]
	\arrow[crossing over, from=2-2, to=2-4]
	\arrow[crossing over, from=2-2, to=4-2]
\end{tikzcd}
\]
\end{example}

\section{Lagrangian structures on moduli of constructible sheaves}\label{sec:Lagrangian_str}

	In this section we analyze the derived algebraic geometric implications of \Cref{thm:final_result}.
	We show the existence of shifted Lagrangian structures for the derived moduli of constructible sheaves and perverse sheaves on an orientable conically smooth stratified space in \Cref{subsec:moduli_of_constr_sheaves}. For stratifications given by a smooth submanifold of codimension $2$, we further describe symplectic leaves of perverse sheaves with prescribed monodromy around the tubular neighborhoods in \Cref{subsec:symplectic_leaves}. Examples are $0$-shifted and $(-1)$-shifted symplectic stacks arise from Riemann surfaces and knot embeddings, see \Cref{ex:Riemann_surface} and \Cref{ex:knots_embedding}.

\subsection{Moduli of objects}

	In this paragraph we introduce the moduli of objects of a $k$-linear $\infty$-category of finite type (\Cref{def:finite_type}). This construction first appeared in \cite{TV07}, but we will follow the presentation given in \cite[Section 1.5]{Porta_HDR} (see also \cite[Section 5]{AG14}, where a different convention is used).	
\begin{definition}\label{def:finite_type}
	Let $\C$ be a compactly generated $k$-linear $\infty$-category.
	We say that $\C$ is of finite type if it is a compact object in the $\infty$-category $\on{LinCat}_k^\omega$ of compactly generated $k$-linear $\infty$-categories and $k$-linear compact objects preserving functors.
\end{definition}
	
\begin{recollection}[{$\!\!$\cite[Section 4.8]{HA}}]\label{tensor_product}
	For $\C, \D \in \on{LinCat}_k$, we can consider their tensor product
\[
\C \otimes_k \D \simeq \on{Fun}^{\on{R}}_k(\C^{op}, \D)\,,
\]
where the right hand side is the $\infty$-category of $k$-linear functors commuting with limits.
	This tensor product endows $\on{LinCat}_k$ with a symmetric monoidal structure which restricts to $\on{LinCat}_k^\omega$. 
	If $\C$ is compactly generated, there is a canonical equivalence
\[
\C \otimes_k \D \simeq \on{Fun}^{\on{ex}}_k((\C^{\omega})^{op}, \D)\,,
\]
where the right hand side is the $\infty$-category of exact $k$-linear functors.
\end{recollection}

\begin{definition} 
	Let $\C \in \on{LinCat}_k$ and $\on{Spec}(A) \in \on{dAff}_k$.
	We define the $\infty$-category of $A$-families of objects in $\C$ as
\[
\C_A \coloneqq  \C \otimes_k \on{Mod}_A\,.
\]
\end{definition}

\begin{definition}\label{pseudo_perfect}
	Let $\C \in \on{LinCat}_k^\omega$ and $\on{Spec}(A) \in \on{dAff}_k$.
	We define the $\infty$-category of pseudo-perfect families over $A$ in $\C$ as the full subcategory
\[
\on{Fun}_k^{\on{ex}}((C^\omega)^{op}, \on{Perf}(A)) \subset \C_A
\]
spanned by exact $k$-linear functors valued in $\on{Perf}(A)$.
\end{definition}

\begin{recollection}[{$\!\!$\cite[Recollection 2.2.1]{Lam25}}]\label{moduli_of_objects_recollection}
	Let $\C \in \on{LinCat}_k^\omega$.
	Consider the presheaf
\[
\widehat{\M}_\C \colon\on{dAff}_k^{op} \to \mathcal{S}
\]
\[
\on{Spec}(A) \to \C_A^{\simeq}.
\]
	The presheaf $\widehat{\M}_\C $ defines a derived stack over $k$.
	The stack $\widehat{\M}_\C$ is typically too big to be representable.
	The derived moduli of objects of $\C$ is defined as the presheaf 
\[
\M_\C:\on{dAff}_k^{op} \to \mathcal{S}
\]
\[
\on{Spec}(A) \to \C_A^{pp}\,,
\]
where $\C_A^{pp} \subset \C_A$ is the maximal groupoid spanned by pseudo-perfect families over $A$ in the sense of \Cref{pseudo_perfect}.
	The sub-presheaf
\[
\M_\C \subset \widehat{\M}_\C
\]
is a derived substack.
\end{recollection}

	The main theorem of \cite{TV07} translates as:
\begin{theorem}[{$\!\!$\cite[Theorem 5.8 \& Corollary 5.9]{AG14}}]
	Let $\C \in \on{LinCat}_k^\omega$  be of finite type.
	Then $\M_\C$ is a locally geometric derived stack locally of finite presentation over $k$.
	Furthermore, the tangent complex at a point $x \colon  \on{Spec}(A) \to \M_\C$ corresponding to a pseudo-perfect family of objects $M_x \in \C_A$ is given by
\[
x^\ast\mathbb{T}_{\M_\C} \simeq \on{Hom}(M_x, M_x)[1]\,.
\]
\end{theorem}

\begin{theorem}[$\!\!${\cite[Theorem 5.5]{BD21}}]\label{thm:CY_lagr}~
\begin{enumerate}[(i)]
	\item Let $\C$ be a compactly generated, smooth $k$-linear $\infty$-category.
	Then a left $d$-Calabi--Yau structure on $\C$ induces a $(2-d)$ shifted symplectic structure on $\M_\C$. 
	\item Let $\varphi \colon \C \to \D$ be a compact objects preserving $k$-linear functor between compactly generated, smooth $k$-linear $\infty$-categories with right adjoint $\varphi^r$.
	Then a relative left $d$-Calabi--Yau structure on $\varphi$ induces a $(2-d)$-shifted Lagrangian structure on $\varphi^r \colon \M_\D \to \M_\C$.
\end{enumerate}
\end{theorem}

We note that \Cref{thm:CY_lagr} was formulated in \cite{BD21} in terms of dg categories and the To\"en-Vaquié moduli objects. The proof given in \cite{BD21} is however of a formal categorical nature, and thus directly translates from the setting of dg categories to that of $k$-linear stable $\infty$-categories.

\subsection{Moduli of constructible and perverse sheaves}\label{subsec:moduli_of_constr_sheaves}
	In this paragraph we give the definition of the derived moduli of constructible and perverse sheaves, following \cite{HPT26}.
	Under mild assumptions that are always satisfied for compact conically smooth stratified spaces, the moduli of constructible sheaves is an example of a moduli of objects (\Cref{openness_flatness_perv}), of which the moduli of perverse sheaves is an open substack.\\
	
	We start by recalling the definition of perverse sheaves following \cite{BBDG18}.
\begin{definition}\label{def_perv}
	Let $(X,P)$ be a stratified space and $\mathfrak{p} \colon P \to \mathbb{Z}$ be a function. 
	Let $R$ be a simplicial commutative ring.
	Consider the pair of full subcategories $\on{Sh}(X;\on{Mod}_A)$ defined by
\[
{}^{\mathfrak{p}} \on{Sh}(X;\on{Mod}_A)_{\geq 0} \coloneqq \big\{F \in \on{Sh}(X;\on{Mod}_A) \mid \ \forall p \in P \ , \ \pi_i( i_p^{\ast}(F) ) = 0 \text{ for every } i \leq \mathfrak p(p) \big\}\,,
\]
\[
{}^{\mathfrak{p}} \on{Sh}(X;\on{Mod}_A)_{\leq 0} \coloneqq \big\{F \in \on{Sh}(X;\on{Mod}_A) \mid \ \forall p \in P \ , \ \pi_i( i_p^{!}(F) ) = 0 \text{ for every } i \geq \mathfrak p(p) \big\}\,.
\]
	The $\infty$-category of perverse sheaves is
\[
{}^{\mathfrak{p}} \on{Perv}(X; \on{Mod}_A) \coloneqq {}^{\mathfrak{p}} \on{Sh}(X;\on{Mod}_A)_{\leq 0} \cap {}^{\mathfrak{p}} \on{Sh}(X;\on{Mod}_A)_{\geq 0}\,.
\]
\end{definition}

\begin{definition}
    In the setting of \Cref{def_perv}, we set
\[
{}^{\mathfrak{p}} \Cons_P(X;\on{Mod}_A)_{\geq 0} \coloneqq {}^{\mathfrak p} \on{Sh}(X;\on{Mod}_A)_{\geq 0} \cap \Cons_P(X;\on{Mod}_A) \subset \Cons_P(X;\on{Mod}_A)\,.
\]
    We define analogously ${}^{\mathfrak{p}} \Cons_P(X;\on{Mod}_A)_{\leq 0}$ and ${}^{\mathfrak p} \on{Perv}_P(X;\on{Mod}_A)$.
    We denote by
\[
{}^{\mathfrak p} \on{Perv}_{P, \omega}(X;\on{Mod}_A) \coloneqq {}^{\mathfrak {p}} \on{Perv}_{P}(X;\on{Mod}_A) \cap \Cons_{P, \omega}(X;\on{Mod}_A)_{\geq 0} \subset {}^{\mathfrak p} \on{Perv}_{P}(X;\on{Mod}_A)
\]
the full subcategory spanned by $P$-constructible perverse sheaves with perfect stalks.
\end{definition}

	\Cref{prop:recollements} allows one to prove the following (see also \cite[Lemma 3.2.2 \& Remark 4.2.12]{HPT26}):

\begin{proposition}[{$\!\!$\cite[Proposition 4.1.5]{Lam25}}]\label{perverse_t_structure}
	Let $(X,P)$ be a stratified space with $P$ finite and let $\mathfrak{p} \colon P \to \mathbb{Z}$ be a function.
\begin{enumerate}\itemsep=0.2cm
    \item The pair of $\infty$-categories $({}^{\mathfrak p} \on{Sh}(X;\on{Mod}_A)_{\leq 0}, {}^{\mathfrak p} \on{Sh}(X;\on{Mod}_A)_{\geq 0})$ of \Cref{def_perv} define a $t$-structure ${}^{\mathfrak{p}} \tau$ on $\on{Sh}(X;\on{Mod}_A)$.
    \item If $(X,P)$ is conically stratified with locally weakly constructible strata, the $t$-structure ${}^{\mathfrak{p}} \tau$ restricts to a $t$-structure on $\on{Cons}_P(X;\on{Mod}_A)$.
\end{enumerate}
\end{proposition}

\begin{notation}\label{notation:standard}
	It follows from \cite[Lemma 3.1.8]{HPT26} that if $\mathfrak{p}$ is the constant $0$ function, then the perverse $t$-structure coincides with the standard $t$-structure on $\Cons_P(X; \on{Mod}_A)$.
	We will denote it by $\tau_{st}$.
\end{notation}

\begin{definition}\label{def:tau_flat}
	Let $(X,P)$ be a stratified space with $P$ finite and let $\mathfrak{p} \colon P \to \mathbb{Z}$ be a function.
	Let $\on{Spec}(A) \in \on{dAff}_k$ and $F \in \Cons_P(X; \on{Mod}_A)$.
	We say that $F$ is ${}^{\mathfrak{p}}\tau$-flat over $A$ if for every $M \in \on{Mod}_A^\heartsuit$ we have $F \otimes_A M \in {}^{\mathfrak{p}}\on{Perv}_P(X; \on{Mod}_A)$.
\end{definition}

\begin{remark}
	In the setting of \Cref{def:tau_flat}, when $A$ is discrete, a $\tau$-flat family $F$ over $A$ lies in ${}^{\mathfrak{p}}\on{Perv}_P(X; \on{Mod}_A)$.
	Indeed in this case one has $A \in \on{Mod}_A^\heartsuit$, so that $F \simeq F \otimes_A A \in {}^{\mathfrak{p}}\on{Perv}_P(X; \on{Mod}_A)$.
	The converse holds when $A$ is a field.
\end{remark}

\begin{remark}\label{standard_t_structure}
	In the setting of \Cref{def:tau_flat}, if we consider the standard $t$-structure $\tau_{st}$ of \Cref{notation:standard}, then $F \in \Cons_{P, \omega}(X; \on{Mod}_A)$ is $\tau_{st}$-flat if and only if its stalks are $A$-modules of Tor-amplitude $[0,0]$\footnote{Recall that an $A$-module is of Tor-amplitude $[0,0]$ if the derived tensor product with any discrete $A$-module is again discrete.}.
	In particular, when $A$ is discrete then $F$ is $\tau_{st}$-flat if and only if its stalks are locally free $A$-modules of finite rank.
\end{remark}

\begin{recollection}\label{recollection:moduli_cons}
	Let $(X,P)$ be a stratified space and $\mathfrak{p} \colon X \to P$ be a function.
	The derived prestack of constructible sheaves is defined by the assignment
\[
\mathbf{Cons}_P(X) \colon \on{dAff}_k \to \mathcal{S}
\]
\[
\on{Spec} (A) \mapsto \Cons_{P,\omega}(X; \on{Mod}_A)^{\simeq}
\]
with functoriality given by extension of scalars. 
	If $P$ is finite, we let
\[
{}^{\mathfrak{p}}\mathbf{Perv}_P(X) \subset \mathbf{Cons}_P(X)
\]
the sub-prestack over $k$ spanned by those constructible sheaves with perfect stalks that are ${}^{\mathfrak{p}}\tau$-flat.
	We will refer to ${}^{\mathfrak{p}}\mathbf{Perv}_P(X)$ as the derived prestack of perverse sheaves.

We emphasize that if $(X,P)$ is a conically smooth stratification with boundary, i.e.~$\langle 1\rangle$-corners, $\on{Cons}_{P,\omega}(X;\on{Mod}_A)^{\simeq}$
\end{recollection}

\begin{notation}
In the setting of \Cref{recollection:moduli_cons}, if the stratification is trivial, i.e., if $P = \{ \ast \}$, we denote $\mathbf{Cons}_P(X)$ by $\mathbf{Loc}(X)$. We further denote by $\mathbf{Loc}^{[0,0]}(X)\subset \mathbf{Loc}(X)$ the sub-prestack over $k$ spanned by local systems with perfect stalks of Tor-amplitude $[0,0]$.
\end{notation}

	In the following theorem, we summarize the results contained in (the proof of) \cite[Corollary 4.1.11 \& Proposition 4.1.12 \& Theorem 4.2.11]{HPT26}.

\begin{theorem}\label{openness_flatness_perv}
    Let $(X, R)$ be a conically stratified space with locally weakly contractible strata, let $\phi\colon R \to P$ be a map of posets, let $\mathfrak{p} \colon P\to \mathbb{Z}$ be a function. 
    Assume that $(X, R)$ is categorically compact.
    Then:
\begin{enumerate}\itemsep=0.2cm
	\item $\Cons_P(X; \on{Mod}_k) \in \on{LinCat}_k^\omega$ is of finite type.
    \item There is a canonical equivalence
\[
\mathcal{M}_{\Cons_P(X; \on{Mod}_k)} \simeq \mathbf{Cons}_P(X).
\]
    In particular $\mathbf{Cons}_P(X)$ is a locally geometric derived stack, locally of finite presentation over $k$.
    \item If furthermore $(X, R)$ is locally categorically compact and $R,P$ are finite posets, then the morphism of derived stacks
\[
{}^{\mathfrak{p}}\mathbf{Perv}_P(X) \hookrightarrow \mathbf{Cons}_P(X)
\]
is representable by an open immersion.
    In particular ${}^{\mathfrak{p}}\mathbf{Perv}_P(X)$ is a $1$-Artin stack locally of finite presentation over $k$.
\end{enumerate}
\end{theorem}

\begin{example}
	\Cref{openness_flatness_perv} applies to constructible sheaves on the interior of a compact conically smooth stratified space with corners by \Cref{thm:finiteness_conically_smooth}.
\end{example}

	For future reference, let us record also the following:

\begin{theorem}[{$\!\!$\cite[Lemma 4.2.5 \& Lemma 4.2.10]{HPT26}}]\label{thm:refinement_are_open}
	Let $(X, R)$ be a categorically compact, locally categorically conically stratified space with locally weakly contractible strata and $\phi \colon R \to P$ be a map of finite posets.
	Let $\mathfrak{p} \colon P \to \mathbb{Z}$ be a function.
	Then we have a pullback square in $\dSt$
\[\begin{tikzcd}[sep=small]
	{{}^{\mathfrak{p}}\mathbf{Perv}_P(X; \on{Mod}_A)} && {{}^{\mathfrak{p} \circ \phi}\mathbf{Perv}_R(X; \on{Mod}_A)} \\
	\\
	{\mathbf{Cons}_P(X; \on{Mod}_A)} && {\mathbf{Cons}_R(X; \on{Mod}_A)}
	\arrow[hook, from=1-1, to=1-3]
	\arrow[hook, from=1-1, to=3-1]
	\arrow[hook, from=1-3, to=3-3]
	\arrow[hook, from=3-1, to=3-3]
\end{tikzcd}\]
where all the arrows are open immersions.
\end{theorem}

	Given a functor $C_\ast \colon K \to \on{LinCat}_k^\omega$, we denote by $\M_{C_\ast} \colon K \to \on{dSt}_k$ to composite with the functor $\M_{(-)}$. \Cref{cor:relative_Calabi_Yau} translates to the following result on the moduli of objects:	
		
\begin{corollary}\label{corollary:lagrangian_general}
	Let $(X,P)$ a conically smooth stratified space of depth $n$ with (possibly empty) boundary.
	Assume that the underlying topological space $X$ is a compact orientable topological manifold of dimension $d$ with (possibly empty) boundary.
	Let $\mathfrak{p}\colon P \to \mathbb{Z}$ be a function.
	Then a choice of orientation of $X$ induces a $(2-d)$-shifted Lagrangian structures on the morphisms in $\dSt$
\[
\mathbf{Cons}_P(X) \to \on{lim}_{I^{n+1} \setminus \{(1, \ldots 1)\}}\M_{\cons_P(X)_\ast}\,.
\]
and
\[
{}^{\mathfrak{p}}\mathbf{Perv}_P(X)\longrightarrow  \mathcal{M}_{\on{colim}_{I^{n+1}\backslash \{1,\dots,1\}}\on{Cons}_P(X)_\ast}\,.
\]
\end{corollary}

\begin{proof}
	Recall that $P$ is finite by \Cref{rem:finite_poset}, so that ${}^{\mathfrak{p}}\mathbf{Perv}_P(X)$ is well defined and ${}^{\mathfrak{p}}\mathbf{Perv}_P(X) \hookrightarrow \mathbf{Cons}_P(X)$ is an open immersion by \Cref{openness_flatness_perv}.
	Hence it is enough to prove the first statement.
	By construction, every $k$-linear $\infty$-category appearing in the diagram $\cons(X)_\ast$ is a colimit of finite type $k$-linear $\infty$-categories and each functor is a colimit of dualizable functors.
	Since $\on{LinCat}_k^\omega \to \on{LinCat}_k$ preserves colimits, the left $d$-Calabi-Yau functor $\on{colim}_{I^{n+1} \setminus \{(1, \ldots 1)\}}\cons(X)_\ast \to \cons_P(X)$ of \Cref{cor:relative_Calabi_Yau} is a dualizable functor between finite type categories.
	The claim then follows by \Cref{thm:CY_lagr}.
\end{proof}

\begin{remark}
	\Cref{corollary:lagrangian_general} can be generalized to a stratified space $(X,P)$ admitting a refinement $(X,R)$ which is the interior of a conically smooth stratified space with (possible empty) boundary.
	Indeed this follows by applying \Cref{corollary:lagrangian_general} to $(X,R)$ and then restricting the Lagrangian morphism along the open immersion $\mathbf{Cons}_P(X) \hookrightarrow \mathbf{Cons}_R(X)$ (\Cref{thm:refinement_are_open}).
\end{remark}

	As a direct corollary of \Cref{corollary:lagrangian_general}, we obtain the following result:

\begin{corollary}\label{coro:Poisson_structure}
	Let $X$ be a compact oriented manifold of dimension $d$ with boundary.
	Let $(X,P)$ be a conically smooth stratification of $X$ with boundary and let $\mathfrak{p} \colon P \to \mathbb{Z}$ be a function.
	Then the derived stacks $\mathbf{Cons}_P(X)$ and ${}^{\mathfrak{p}}\mathbf{Perv}_P(X)$ carry $(2-d)$-shifted Poisson structures.
\end{corollary}

\begin{proof}
	Thanks to the results of \cite[Theorem 4.22]{MS24}, the space of $n$-shifted Lagrangian structures on a morphism of derived stacks is equivalent to the space of non-degenerate $n$-shifted coisotropic structures in the sense of \cite[Definition 4.3]{MS24}.
	The notions of $n$-shifted coisotropic structures of \cite{CPT+17} and of \cite{MS24} agree by \cite[Corollary 3.9]{Saf18}.
	Hence the statement follows directly from \Cref{corollary:lagrangian_general} and the definition of shifted coisotropic structures (see \cite[Definition 3.4.4]{CPT+17} and the discussion thereafter).
\end{proof}

	In \Cref{section:leaves_divisors} we study the symplectic leaves of the above Poisson structure for perverse sheaves in some particular cases.

\subsection{Symplectic leaves for moduli of perverse sheaves}\label{subsec:symplectic_leaves}

	In this paragraph we study conically smooth stratifications given by a closed smooth manifold $X$ and a smooth closed submanifold $D$.
	In this setting, we can give an explicit description of the Lagrangian morphism appearing in \Cref{corollary:lagrangian_general} (\Cref{corollary:lagrangian_from_divisor}).
	When $D$ is of codimension $2$, this also allow us to study the symplectic leaves of the moduli of perverse sheaves (\Cref{thm:symplectic_leaves}).
	
\subsubsection{Lagrangian structure in the case of smooth submanifolds}

\begin{construction}\label{construction:divisor_stratification}
    Let $X$ be a smooth closed manifold of dimension $d$ and 
    \[
    D \coloneqq \sqcup_{i=1}^m D_i
    \]
    be a disjoint union of smooth connected codimension $c$ closed submanifolds with $c\geq 1$.
        Consider the stratification 
    \[
    \rho \colon X \to \{0<1\}
    \]
\begin{align*}
x \mapsto
\left\{
    \begin{aligned}
        & 0 \ \  \ \ \text{if $x \in D$} \\
        & 1\ \ \ \ \text{if $x \in X \setminus D$.}                  
    \end{aligned}
\right.
\end{align*}
	The stratified space $(X, \{ 0<1 \})$ is conically stratified by \cite[Example 3.5.8]{AFT17}.
\end{construction}
	
\begin{notation}
	Exceptionally, in this paragraph we denote the stratified space $(X, \{ 0<1 \})$ by $(X,D)$.
	We modify the notation for constructible and perverse sheaves on $(X,D)$ accordingly.
\end{notation}

    Since there are only two strata, we need to apply the $\on{Unzip}$ procedure only once, and \Cref{prop:final_result_k} shows that we get a $d$-Calabi--Yau cube of dimension $1$, that is, a relative $d$-Calabi--Yau functor.
    Consider a set of tubular neighborhoods $\on{T}(D_i) \subset X$ of $D_i$, $1 \leq i \leq m$. 
    We can choose the tubular neighborhoods so that they do not intersect, that is $\on{T}(D_i) \cap \on{T}(D_j) = \emptyset$ for every $i \neq j$.
    Then we have 
    \[
    \on{Unzip}(X) \simeq X \setminus (\sqcup_{i=1}^m \on{T}(D_i))
    \]
    which is a manifold with boundary.
    For $1 \leq i \leq m$, let $q \colon \partial D_i \to D_i$ be the $S^{c-1}$-bundle associated to the $D^c$-bundle $\on{T}(D_i) \to D_i$.
    Notice that $\bigsqcup_{i=0}^m \partial D_i$ is the boundary of $\on{Unzip}(X)$.
    The projection map $p \colon \on{Unzip}(X) \to X$ and the inclusion of the boundary $\widetilde{i} \colon \bigsqcup_{i=0}^m \partial D_i \to \on{Unzip}(X)$ fit in a pushout square in $\on{Top}$ as follows:
\[\begin{tikzcd}[sep=small]
	{\bigsqcup_{i=0}^n \partial D_i} && {\on{Unzip}(X)} \\
	\\
	{\bigsqcup_{i=1}^m D_i} && X
	\arrow["{\widetilde{i}}", from=1-1, to=1-3]
	\arrow["q"', from=1-1, to=3-1]
	\arrow["p", from=1-3, to=3-3]
	\arrow["i"', from=3-1, to=3-3]
\end{tikzcd}\]

\begin{proposition}\label{prop:lagra_submanifold}
    The functor
    \[
    p_\# \widetilde{i}_\ast \colon \bigsqcup_{i=0}^m \on{Loc}(\partial D_i) \to \on{Cons}_{0<1}(X)
    \]
    carries a relative left $d$-Calabi--Yau structure.
\end{proposition}

\begin{proof}
	The functors
\[
\widetilde{i}_\# \colon \on{Loc}(\bigsqcup_{i=0}^m \partial D_i) \simeq \bigsqcup_{i=0}^m \on{Loc}(\partial D_i) \to \on{Loc}(\on{Unzip}(X))
\]
\[
q_\# \colon \on{Loc}(\bigsqcup_{i=0}^m \partial D_i) \simeq \bigsqcup_{i=0}^m \on{Loc}(\partial D_i) \to \bigsqcup_{i=0}^m \on{Loc}(\on{T}(D_i)) \simeq\on{Loc}( \bigsqcup_{i=0}^m \on{T}(D_i))
\]
carry compatibles relative left $d$-Calabi--Yau structures by \cite[Theorem 5.7]{BD19}.
    By \Cref{cor:lax_pushout_Unzip}, we have an equivalence
    \[
    \on{Cons}_{0<1}(X) \simeq \laxpush{\on{Loc}(\bigsqcup_{i=1}^m D_i)}{\on{Loc}(\bigsqcup_{i=0}^n \partial D_i)}{\on{Loc}(\on{Unzip}(X))}.
    \]
    The claim then follows by \Cref{prop:laxCYgluing}.
\end{proof}

\begin{notation}\label{notation_upper_shriek}
    We denote by $\widetilde{i}^! \colon \on{Cons}_{0<1}(\on{Unzip}(X)) \to \on{Loc}(\bigsqcup_{i=0}^m \partial D_i)$ the right adjoint to the extension by zero functor $\widetilde{i}_\ast \colon \on{Loc}(\bigsqcup_{i=0}^m \partial D_i) \to \on{Cons}_{0<1}(\on{Unzip}(X))$.
\end{notation}

\begin{corollary}\label{corollary:lagrangian_from_divisor}
    In the setting of \Cref{notation_upper_shriek}, let $\mathfrak{p} \colon \{0<1\} \to \mathbb{Z}$ be a function.
    The map in $\dSt$
\[
    \widetilde{i}_\ast^!p^\ast \colon \mathbf{Cons}_D(X) \to \prod_{i=0}^m \mathbf{Loc}(\partial D_i).
\]
carries a $(2-d)$-shifted Lagrangian structure.
\end{corollary}

\begin{proof}
	This follows directly by combining \Cref{prop:lagra_submanifold} with \Cref{thm:finiteness_conically_smooth} and \Cref{thm:CY_lagr}.
\end{proof}

\begin{corollary}
	In the setting of \Cref{notation_upper_shriek}, let $\mathfrak{p} \colon \{0<1\} \to \mathbb{Z}$ be a function.
	   The map in $\dSt$
\[
    \widetilde{i}_\ast^!p^\ast \colon {}^{\mathfrak{p}}\mathbf{Perv}_D(X) \to \prod_{i=0}^m \mathbf{Loc}(\partial D_i).
\]
carries a $(2-d)$-shifted Lagrangian structure.
\end{corollary}

\begin{proof}
	By \Cref{thm:finiteness_conically_smooth} and \Cref{openness_flatness_perv}-(3), the natural inclusion ${}^{\mathfrak{p}}\mathbf{Perv}_D(X) \subset \mathbf{Cons}_D(X)$ in an open immersion.
	Hence the results follows by \Cref{corollary:lagrangian_from_divisor}.
\end{proof}

\subsubsection{Symplectic leaves for smooth divisors}\label{section:leaves_divisors}

\begin{notation}\label{notation:perversity_divisor}
	In the setting of \Cref{construction:divisor_stratification}, we fix for the rest of this paragraph the perversity $\mathfrak{p} \coloneqq \on{id} \colon \{0,1\} \to \{0,1\} \subset \mathbb{Z}$.
\end{notation}

\begin{example}
	The perversity $\mathfrak{p}$ of \Cref{notation:perversity_divisor} coincides with the middle perversity if $\dim X=2$ (e.g. $X$ is a Riemann surface with a finite number of marked points), and with the lower middle perversity if $\dim X=3$ (e.g., $X= S^3$ stratified at a knot $K \subset S^3$).
	More generally, for $D \subset X$ a smooth divisor (i.e., a smooth submanifold of real codimension $2$), the perversity $\mathfrak{p}$ coincides with (a shift of) the lower middle perversity.
\end{example}

\begin{lemma}\label{lm:preserve_t_structure}
	In the setting of \Cref{notation:perversity_divisor}, let $D \subset X$ be of codimension $2$ and let $\on{Spec}(A) \in \on{dAff}_k$.
	Then the functor
\[
\widetilde{i}^!p^\ast \colon \Cons_{0<1}(X; \on{Mod}_A) \to \Loc(\partial D; \on{Mod}_A)
\]
restricts to a functor
\[
\widetilde{i}^!p^\ast \colon {}^{\mathfrak{p}}\on{Perv}_{0<1}(X; \on{Mod}_A) \to \Loc(\partial D; \on{Mod}_A)^\heartsuit \,.
\]	
\end{lemma}

\begin{proof}
	Let $x \in \partial D$ and $F \in {}^{\mathfrak{p}}\on{Perv}_{0<1}(X; \on{Mod}_A)$. 
	We need to show that $(\widetilde{i}^!p^\ast F)_x \in \on{Mod}_A^\heartsuit$.
	The question is local around $x$, hence we can suppose that $D$ is connected and that $D \hookrightarrow X$ is given by $\{(0,0)\} \times D^{n-2} \hookrightarrow D^n \simeq \on{C}(S^1) \times D^{n-2}$, where $D^{n-2}$ is the $(n-2)$-dimensional disk.
	By homotopy invariance \cite[Corollary 5.22]{HPT23}, we can reduce to $\{(0,0)\} \hookrightarrow D^2 \simeq \on{C}(S^1) $.
	In this setting, the unzipping and link can be described explicitly by the following diagram with pullback squares
\[\begin{tikzcd}[sep=small]
	{S^1 \times \{0\} } && {S^1 \times [0,1)} && {S^1 \times (0,1)} \\
	\\
	{\{(0,0)\}} && {\on{C}(S^1)} && {\on{C}(S^1) \setminus (\{(0,0)\}}
	\arrow["{\widetilde{i}}", hook, from=1-1, to=1-3]
	\arrow["q"', from=1-1, to=3-1]
	\arrow["p", from=1-3, to=3-3]
	\arrow["{\widetilde{j}}"', hook', from=1-5, to=1-3]
	\arrow[from=1-5, to=3-5]
	\arrow["i"', hook, from=3-1, to=3-3]
	\arrow["j"', hook', from=3-5, to=3-3]
\end{tikzcd}\]
where the left and central vertical arrows are the natural projections, the right vertical arrow is a homeomorphism, the left horizontal arrows are closed immersions and the right horizontal arrows are open immersions.	
	Consider the fiber/cofiber sequence
	\[
	\widetilde{i}^!p^\ast F \to \widetilde{i}^\ast p^\ast F \to \widetilde{i}^\ast \widetilde{j}_\ast \widetilde{j}^\ast p^\ast F \,.
	\]
Consider the long exact sequence in $\on{Mod}_A^\heartsuit$ of the homotopy groups of the stalks at $x$
	\[
	0 \to \pi_1((\widetilde{i}^!p^\ast F)_x) \to \pi_1((\widetilde{i}^\ast p^\ast F)_x) \xrightarrow{\varphi} \pi_1((\widetilde{i}^\ast j_\ast j^\ast p^\ast F)_x) \to \pi_0((\widetilde{i}^!p^\ast F)_x) \to  \pi_0((\widetilde{i}^\ast p^\ast F)_x) \to 0 \,,
	\]
where $\pi_i((\widetilde{i}^!p^\ast F)_x) \simeq \pi_n(i^\ast F)$ for all $i \neq 0,1$.
	Indeed, by homotopy invariance applied to the pushforward along $\widetilde{j} \colon S^1 \times (0,1) \to S^1 \times [0,1)$ and since $F$ is $\mathfrak{p}$-perverse, we get that $i^\ast j_\ast j^\ast p^\ast$ is concentrated in degree $1$.
	Since $F$ is $\mathfrak{p}$-perverse, we have $\pi_n(i^\ast F) \simeq 0$ for every $n < \mathfrak{p}(0)=0$.
	Also, since $F$ is $\mathfrak{p}$-perverse and $\mathfrak{p}(1)=1$, we have that it is $1$-coconnective \cite[Lemma 3.1.8]{HPT26}.
	Since $\widetilde{i}^!$ is left exact for the standard $t$-structure (its left adjoint being $t$-exact), we get $\pi_i((\widetilde{i}^!p^\ast F)_x)\simeq 0$ for every $n>1$. \\ \indent
	Hence we are left to show that $\pi_1((\widetilde{i}^!p^\ast F)_x) \simeq 0$, or equivalently that the morphism 
	\[
	\varphi \colon \pi_1((\widetilde{i}^\ast p^\ast F)_x) \to \pi_1((\widetilde{i}^\ast j_\ast j^\ast p^\ast F)_x)
	\]
is injective.
	Let $y \in \on{C}(S^1)\setminus \{0\}$ and $U_y \subset \on{C}(S^1) \setminus \{ 0\}$ a contractible open neighborhood of $y$.
	Unraveling the definitions, one finds that $\varphi$ factors as the composition
	\[
	\pi_1(F_0) \simeq \pi_1\Gamma(\on{C}(S^1),F) \xrightarrow{\varphi_1} \pi_1\Gamma(\on{C}(S^1) \setminus \{ 0 \},j^\ast F) \xrightarrow{\varphi_2} \pi_1\Gamma( U_y, j^\ast F) \simeq \pi_1(j^\ast F_y)\,.
	\]
	The morphism $\varphi_2$, which is given by restricting to a contractible open neighborhood of $y$, is always injective (it is the inclusion of the fixed points of the monodromy of $F_y[-1] \in \on{Mod}_A^\heartsuit$).
	Hence we are left to show that $\varphi_1$ is injective.
	For this, consider the fiber/cofiber sequence
	\[
	i^!F \to i^\ast F \to i^\ast j_\ast j^\ast F \,.
	\]
	Taking the long exact sequence of homotopy groups of the stalks at $0$ we get an exact sequence in $\on{Mod}_A^\heartsuit$
\[
\pi_1(i^!F) \to \pi_1(i^\ast F) \xrightarrow{\varphi_1} \pi_1(i^\ast j_\ast j^\ast F)
\]
	By definition of ${}^{\mathfrak{p}}\on{Perv}_{0<1}(X; \on{Mod}_A)$, we have that $\pi_1(i^!F) \simeq 0$, hence the claim follows.
\end{proof}

\begin{corollary}\label{corollary:factorisation}
	In the setting of \Cref{notation:perversity_divisor}, let $D \subset X$ be of codimension $2$ and let $\on{Spec}(A) \in \on{dAff}_k$.
	Then the functor
\[
\widetilde{i}^!p^\ast \colon \Cons_{0<1}(X; \on{Mod}_A) \to \Loc(\partial D; \on{Mod}_A)
\]
sends ${}^{\mathfrak{p}}\tau$-flat objects of $\Cons_{0<1}(X; \on{Mod}_A)$ to $\tau_{st}$-flat objects of $\Loc(\partial D; \on{Mod}_A)$
\end{corollary}

\begin{proof}
	Let $M \in \on{Mod}_A^\heartsuit$ and $F \in \Cons_{0<1}(X; \on{Mod}_A) $ be ${}^{\mathfrak{p}}\tau$-flat.
	We need to show that $\widetilde{i}^!p^\ast (F) \otimes_A M \in \Loc(X; \on{Mod}_A)^\heartsuit$. 
	By \cite[Lemma 3.3.4]{HPT26}, we have an equivalence
\[
	\widetilde{i}^!p^\ast (F) \otimes_A M \simeq \widetilde{i}^!p^\ast (F \otimes_A M).
\]
	By assumption we have $F \otimes_A M \in {}^{\mathfrak{p}}\on{Perv}_{0<1}(X; \on{Mod}_A)$, hence the result follows from \Cref{lm:preserve_t_structure}.
\end{proof}

\begin{corollary}\label{prop:shifted_lagrangian_perv}
	In the setting of \Cref{notation:perversity_divisor}, let $D \subset X$ be of codimension $2$.
	Then we have a commutative square in $\dSt$
\[\begin{tikzcd}[sep=small]
	{{}^{\mathfrak{p}}\mathbf{Perv}_D(X)} && {\mathbf{Loc}^{[0,0]}(\partial D)} \\
	\\
	{\mathbf{Cons}_D(X)} && {\mathbf{Loc}(\partial D)}
	\arrow["{\widetilde{i}^!p^\ast}", from=1-1, to=1-3]
	\arrow[hook, from=1-1, to=3-1]
	\arrow[hook, from=1-3, to=3-3]
	\arrow["{\widetilde{i}^!p^\ast}"', from=3-1, to=3-3]
\end{tikzcd}\]
where the vertical arrows are open immersions and the horizontal arrows are $(2-\dim X)$-shifted Lagrangian.
\end{corollary}

\begin{proof}
	That the above commutative diagram exists is the content of \Cref{corollary:factorisation}.
	That the vertical arrows are open immersions is the content of \Cref{thm:finiteness_conically_smooth} and \Cref{openness_flatness_perv}-(3).
	That the lower horizontal map is $(2-n)$-shifted Lagrangian is the content of \Cref{corollary:lagrangian_from_divisor}.
	Since the top horizontal map is a restriction to open substacks of a $(2-n)$-shifted Lagrangian map, it is itself $(2-n)$-shifted Lagrangian.
\end{proof}

	We proceed with a study of the symplectic leaves of the $(2-n)$-shifted Lagrangian map
\[
\widetilde{i}^!p^\ast \colon {}^{\mathfrak{p}}\mathbf{Perv}_D(X) \to \mathbf{Loc}^{[0,0]}(\partial D)\,.
\]
from \Cref{prop:shifted_lagrangian_perv}. We start by recalling the following construction by Pantev-Toën \cite{PT18}:
	
\begin{recollection}\label{recollection:bundle_of_lagrangians}
    Given $T \in \mathcal{S}$, we can consider it as a constant derived stack over $k$.
    Then a locally constant family of derived stacks over $T$ is the data of a morphism $\X \to T$ in $\dSt$.
    Let $n \in \mathbb{N}$ and $G \coloneqq GL_n$.
    Consider the stack $[G/G]$.
    Notice that $S^1$ acts on $[G/G] \simeq \mathbf{Map}(S^1, BG)$.
    Consider now an $S^1$-bundle $\partial T \to T$. This is classified by a map $\alpha \colon T \to BS^1$. 
    We then consider the family of stacks over $T$ defined by the pullback square
\[\begin{tikzcd}[sep=small]
	{\widetilde{[G/G]}_\alpha} && {[[G/G] /S^1]} \\
	\\
	T && {BS^1}
	\arrow["{\alpha'}", from=1-1, to=1-3]
	\arrow["q"', from=1-1, to=3-1]
	\arrow["p", from=1-3, to=3-3]
	\arrow["\alpha"', from=3-1, to=3-3]
\end{tikzcd}\]
    Notice that the fibers of $q$ are (non-canonically) equivalent to $[G/G]$.\\ \indent
    Let now $\lambda \in G$ and $Z_\lambda \subset G$ be its centralizer.
    Then $S^1$ acts on $B(Z)$ via the element $\lambda \in Z_\lambda = \pi_1(\on{Aut}(BZ_\lambda),\on{id})$ and the map $BZ_\lambda \to [G/G]$ is $S^1$-equivariant.
    Hence we can similarly construct $(\widetilde{BZ_\lambda})_\alpha$ and it comes with a $1$-shifted Lagrangian map $(\widetilde{BZ_\lambda})_\alpha \to \widetilde{[G/G]}_\alpha$ in the category of locally constant families of geometric derived over $T$.
    If $T$ has the homotopy type of an topological orientable closed manifold of dimension $n$, then a choice of orientation induces a $(2-n)$-shifted Lagrangian morphism on the derived stacks of global sections
\[
\mathbf{Loc}_{Z_\lambda, \alpha}(Y) \to \mathbf{Loc}_G(Y)\,.
\]
\end{recollection}

\begin{recollection}\label{recollection:compatible_symplectic_structures}
    Let $r \in \mathbb{N}$ and $k$ a field of characteristic zero.
    Then we have a canonical open immersion $BGL_r \to \mathbf{Perf}_k$ sending a vector bundle of rank $r$ to the corresponding perfect complex concentrated in degree $0$.
    By \cite{PTVV13}, the stacks $BGL_r$ and $\mathbf{Perf}_k$ carry compatible symplectic structures.
    Notice that the characteristic zero assumption can be relaxed to $\mathrm{char}(k) \neq 2$ by \cite[Theorem 6.23]{Fu25}.
\end{recollection}

\begin{recollection}\label{recollection:fixed_rank}
	Let $T$ be a locally weakly contractible topological space.
	By \Cref{standard_t_structure}, the derived stack $\mathbf{Loc}^{[0,0]}(X)$ decomposes as disjoint union of open and closed substacks
\[
\mathbf{Loc}^{[0,0]}(T) \coloneqq \bigsqcup_{r \in \mathbb{N}} \mathbf{Loc}^r(T),
\]
where $\mathbf{Loc}^r(T) \simeq \mathbf{Map}(\Pi_\infty(T), BGL_r)$ is the derived stack of local system on $T$ of rank $r$. 
	By \Cref{recollection:compatible_symplectic_structures}, if $T$ is an orientable closed connected topological manifold, then $\mathbf{Loc}^{[0,0]}(T)$ and $\mathbf{Loc}(T)$ carry compatible symplectic structures.
\end{recollection}

\begin{construction}\label{construction:fixed_monodromy}
	In the setting of \Cref{notation:perversity_divisor}, write $D$ as a disjoint union of its connected components $D_1, \ldots, D_m$.
	By \Cref{recollection:fixed_rank}, the derived stack $\mathbf{Loc}^{[0,0]}(\partial D)$ decomposes as a disjoint union of open and closed substacks
\[
\mathbf{Loc}^{[0,0]}(\partial D) \simeq \bigsqcup_{\underline{r} \in \mathbb{N}^{m}} \mathbf{Loc}^{\underline{r}}(\partial D)\,,
\]
where 
\[
\mathbf{Loc}^{\underline{r}}(\partial D) \coloneqq \prod_{i=1}^{m} \mathbf{Loc}^{r_i}(\partial D_i)\,.
\]
We define the derived stack ${}^{\mathfrak{p}}\mathbf{Perv}_D^{\underline{r}}(X)$ as the following pullback in $\dSt$:
\[\begin{tikzcd}[sep=small]
	{{}^{\mathfrak{p}}\mathbf{Perv}^{\underline{r}}_D(X)} && {\mathbf{Loc}^{\underline{r}}(\partial D)} \\
	\\
	{{}^{\mathfrak{p}}\mathbf{Perv}_D(X)} && {\mathbf{Loc}^{[0,0]}(\partial D)}
	\arrow[from=1-1, to=1-3]
	\arrow[from=1-1, to=3-1]
	\arrow[from=1-3, to=3-3]
	\arrow["{\widetilde{i}^!p^\ast}", from=3-1, to=3-3]
\end{tikzcd}\]
	Then ${{}^{\mathfrak{p}}\mathbf{Perv}_D(X)}$ decomposes as a disjoint union of open and closed substacks
\[
{}^{\mathfrak{p}}\mathbf{Perv}_D(X) \simeq \bigsqcup_{\underline{r} \in \mathbb{N}^{m}} {}^{\mathfrak{p}}\mathbf{Perv}^{\underline{r}}_D(X)\,.
\]
	For a set of elements $\lambda_i \in GL_{r_i}$, $i= 0, \ldots, m$ and in the setting of \Cref{recollection:bundle_of_lagrangians}, we define the derived stack ${}^{\mathfrak{p}}\mathbf{Perv}_D^{\underline{\lambda}}(X)$ via the pullback square in $\on{dSt}_k$
\begin{equation}\label{eq:pbsq}\begin{tikzcd}[sep=small]
	{{}^{\mathfrak{p}}\mathbf{Perv}^{\underline{r}, \underline{\lambda}}_D(X)} && {\mathbf{Loc}_{Z_{\underline{\lambda}}, \alpha}(\partial D) \coloneqq \prod_{i=1}^{m} \mathbf{Loc}_{Z_{\underline{\lambda}}, \alpha_i}(\partial D_i)} \\
	\\
	{{}^{\mathfrak{p}}\mathbf{Perv}^{\underline{r}}_D(X)} && {\mathbf{Loc}^{\underline{r}}(\partial D) \coloneqq \prod_{i=1}^{m} \mathbf{Loc}^{r_i}(\partial_iD)}
	\arrow[from=1-1, to=1-3]
	\arrow[from=1-1, to=3-1]
	\arrow[from=1-3, to=3-3]
	\arrow["{\widetilde{i}^!p^\ast}"', from=3-1, to=3-3]
\end{tikzcd}
\end{equation}
where $\alpha \colon D \to BS^1$ (respectively, $\alpha_i \colon D_i \to BS^1$) is the map classifying the $S^1$-bundle $\partial D \to D$ (respectively, $\partial D_i \to D_i$).
\end{construction}

We find the following symplectic leaves of ${}^{\mathfrak{p}}\mathbf{Perv}_D(X)$:

\begin{theorem}\label{thm:symplectic_leaves}
Consider a set of elements $\lambda_i \in GL_{r_i}$, $i= 0,\ldots, m$. Then the derived stack ${}^{\mathfrak{p}}\mathbf{Perv}^{\underline{r}, \underline{\lambda}}_D(X)$ of perverse sheaves with prescribed monodromy $\underline{\lambda}=\{\lambda_1,\dots,\lambda_m\}$ from \Cref{construction:fixed_monodromy} carries a $(2-\dim X)$-shifted symplectic structure.
\end{theorem}

\begin{proof}
	By \Cref{prop:shifted_lagrangian_perv}, \Cref{recollection:bundle_of_lagrangians} and \Cref{recollection:compatible_symplectic_structures}, the bottom horizontal and right vertical morphisms in \eqref{eq:pbsq} carry compatible $(2-\dim X)$-shifted Lagrangian structures.
	The claim then follows by the derived Lagrangian intersection theorem \cite[Theorem 2.9]{PTVV13}.
\end{proof}

\begin{example}\label{ex:Riemann_surface}
	Let $S$ be a Riemann surface and $D = \{p_1, \ldots p_n\} \subset S$ be a finite number of points.
	Let $\underline{r}= \{r_1, \ldots, r_m\} \in \mathbb{N}^r$ and $\underline{\lambda} = \{\lambda_1 \ldots \lambda_r\} \in \prod_{i=1}^{m} GL_{r_i}$.
	Then there exists a $1$-shifted Lagrangian intersection pullback square in $\on{dSt}_k$ as follows:
\[\begin{tikzcd}[sep=small]
	{{}^{\mathfrak{p}}\mathbf{Perv}^{\underline{r}, \underline{\lambda}}_D(S)} && {\prod_{i=1}^{m} BZ_{\lambda_i}} \\
	\\
	{{}^{\mathfrak{p}}\mathbf{Perv}^{\underline{r}}_D(S)} && {\prod_{i=1}^{m} \mathbf{Loc}^{r_i}(S^1) \simeq\prod_{i=1}^{m} \left[GL_{n_i}/GL_{n_i}\right] }
	\arrow[from=1-1, to=1-3]
	\arrow[from=1-1, to=3-1]
	\arrow[from=1-3, to=3-3]
	\arrow["{\widetilde{i}^!p^\ast}"', from=3-1, to=3-3]
\end{tikzcd}\]
	In particular, ${}^{\mathfrak{p}}\mathbf{Perv}^{\underline{r}, \underline{\lambda}}_D(S)$ carries a $0$-shifted symplectic structure. 
\end{example}

\begin{remark}
	Let us point out that studying the symplectic leaves of more complicated stratifications is a challenging problem.
	For example, let $(X,D)$ be a stratified space with $X$ a complex algebraic variety and $D \subset X$ a normal crossing divisor where two (or more) irreducible components intersects.
	Then we do not have an explicit description of the shifted symplectic target of the shifted Lagrangian $\mathbf{Cons}_D(X)$, other than it being a pullback of stacks of constructible sheaves.
	Since exceptional pullbacks are involved, we do not expect that this symplectic target can in general be described by a category of local systems (as in the smooth divisor case).
	For some considerations for local systems on the complement of $D$ in the case of two components $D = D_1 \cup D_2$, see \cite[Section 4.3]{PT18}.
\end{remark}

\bibliography{biblio}
\bibliographystyle{alpha}

\textsc{MC: Mathematisches Institut, Universit\"at Bonn, Endenicher Allee 60, 53115 Bonn, Germany}

\textit{Email address:} \texttt{christ@math.uni-bonn.de}\\

\textsc{EL: Sorbonne Université and Université Paris Cité, CNRS, IMJ-PRG, F-75005 Paris, France}

\textit{Email address:} \texttt{lampetti@imj-prg.fr}

\end{document}